\documentclass[10pt]{article}

\usepackage[utf8]{inputenc}
\usepackage[T1]{fontenc}
\usepackage{scalefnt}
\usepackage{natbib}
\usepackage{amsfonts}
\usepackage{mathtools}
\usepackage{amssymb}
\usepackage{amsmath}
\usepackage{amsthm}
\usepackage{caption}
\usepackage{subcaption}
\usepackage{xcolor}
\definecolor{mydarkblue}{rgb}{0,0.08,0.85}
\usepackage[colorlinks=true,bookmarks=true,linkcolor=mydarkblue,urlcolor=mydarkblue,citecolor=mydarkblue,breaklinks=true]{hyperref}
\usepackage{breakcites}
\usepackage{enumitem}
\usepackage{cases}
\usepackage{commath}
\usepackage{algorithm}
\usepackage{algpseudocode}
\usepackage{interval}
\usepackage{xfrac}
\usepackage{xspace}
\intervalconfig{soft open fences} %
\usepackage{tikz}
\usepackage{tikz-cd}
\usetikzlibrary{matrix}
\usepackage{multirow}

\newtheorem{theorem}{Theorem}[section]
\newtheorem{proposition}[theorem]{Proposition}
\newtheorem{definition}[theorem]{Definition}
\newtheorem{lemma}[theorem]{Lemma}
\newtheorem{corollary}[theorem]{Corollary}
\newtheorem{remark}[theorem]{Remark}

\newcommand{\ones}{\mathbf 1}
\newcommand{\zeros}{0}

\newcommand{\reals}{{\mathbb R}}

\DeclareMathOperator*{\argmin}{arg\,min}

\newcommand{\retr}{\mathrm{R}}
\newcommand{\grad}{\nabla}
\newcommand{\hess}{\nabla^2}

\newcommand{\tangent}{\mathrm{T}}
\newcommand{\normal}{\mathrm{N}}
\newcommand{\ptransport}[2]{\Gamma_{#1}^{#2}}

\newcommand{\Exp}{\mathrm{Exp}}
\newcommand{\Log}{\mathrm{Log}}

\newcommand{\frob}{\mathrm{F}}

\DeclareMathOperator{\sign}{sign}
\DeclareMathOperator{\diag}{diag}

\DeclareMathOperator{\dist}{dist}

\DeclareMathOperator{\vecspan}{span}

\DeclarePairedDelimiterX{\inner}[2]{\langle}{\rangle}{#1, #2}

\newcommand{\D}{\mathrm{D}}

\newcommand{\quotient}[2]{#1/{#2}}

\newcommand{\smooth}[1]{{\mathrm{C}^{#1}}}

\newcommand{\sequence}[1]{\{#1\}}

\makeatletter
\newcommand{\TODOF}[1]{\@bsphack\@esphack}
\makeatother

\newcommand\restrold[2]{{%
    \left.\kern-\nulldelimiterspace %
      #1 %
    \right|_{#2} %
  }}

\makeatletter
\newcommand{\normalscaling}{\bBigg@{0}}

\newcommand{\abitbig}{\bBigg@{1}}
\makeatother

\newcommand\restr[3]{{%
    #3.\kern-\nulldelimiterspace %
      #1 %
    #3|_{#2} %
  }}

\newcommand{\aref}[1]{\hyperref[#1]{A\ref{#1}}}
\newtheorem{assumption}{A\ignorespaces} %
\newcommand{\cref}[1]{\hyperref[#1]{C\ref{#1}}}
\newtheorem{condition}{C\ignorespaces} %
\newcommand{\levenmarq}{Levenberg--Marquardt}
\newcommand{\lanczos}{Lanczos}

\newcommand{\cauchyschwarz}{Cauchy--Schwarz}
\newcommand{\loja}{\L ojasiewicz}
\newcommand{\polyakloja}{Polyak--\loja}
\newcommand{\pl}{\ensuremath{\text{P\L}}}
\newcommand{\morsebott}{Morse--Bott}

\newcommand{\mb}{\ensuremath{\text{MB}}}

\newcommand{\cg}{CG}
\newcommand{\tcg}{t\cg{}}
\newcommand{\minres}{MINRES}

\usepackage{scalerel,stackengine}
\newcommand\pig[1]{\scalerel*[5pt]{\big#1}{%
    \ensurestackMath{\addstackgap[1.5pt]{\big#1}}}}

\newcommand{\polcoef}[2]{\vartheta_{#1}(#2)}

\newcommand{\pols}[1]{\mathcal{P}_{\leq#1}}

\newcommand{\pol}{p}
\newcommand{\qol}{q}
\newcommand{\rhopol}{\pi}
\newcommand{\trhopol}{\tilde \pi}
\newcommand{\ppol}{\varphi}
\newcommand{\qpol}{\varsigma}
\newcommand{\tqpol}{\tilde \varsigma}
\newcommand{\prexipol}{\zeta}
\newcommand{\pxplambda}{\lambda}
\newcommand{\xipol}{\xi}
\newcommand{\cparam}{c}

\newcommand{\coef}{\sigma}

\newcommand{\xidiv}{\delta}
\newcommand{\qdiv}{\tau}

\newcommand{\sdim}{d}
\newcommand{\mdim}{{\tilde d}}
\newcommand{\grade}{\ell}

\newcommand{\mat}{A}
\newcommand{\tmat}{\tilde A}
\newcommand{\weight}{b}
\newcommand{\tweight}{\tilde b}
\newcommand{\krylovmat}{K}
\newcommand{\krylovspace}{\mathcal{K}}

\newcommand{\cgvec}{v}
\newcommand{\cgres}{r}
\newcommand{\cgdir}{u}

\newcommand{\retrdistboundconst}{c_r}

\newcommand{\tcgkappa}{\kappa}
\newcommand{\power}{\theta}

\newcommand{\linearmap}{H}
\newcommand{\rtrmodel}{m}
\newcommand{\rtrrho}{\rho}

\newcommand{\rtrsufficientdecrease}{c_0}
\newcommand{\strongvsconst}{c_1}
\newcommand{\modelconst}{c_2}

\newcommand{\hesslipconst}{L_H}
\newcommand{\hessliplikeconst}{L_H'}
\newcommand{\hessapproxconst}{\beta_H}

\newcommand{\manifold}{\mathcal{M}}

\newcommand{\mfc}{f}
\newcommand{\sfc}{g}

\newcommand{\optimalset}{\mathcal{S}}
\newcommand{\optpoint}{\bar x}

\newcommand{\mfcopt}{f(\bar x)}

\newcommand{\plconstant}{\mu}

\newcommand{\muflat}{\mu^\flat}

\newcommand{\lamsharp}{\lambda^\sharp}

\newcommand{\equivrel}{\sim}

\usepackage{sectsty}
\makeatletter\def\@seccntformat#1{\protect\makebox[0pt][r]{\csname the#1\endcsname\hspace{12pt}}}\makeatother

\usepackage[verbose=true,letterpaper]{geometry}
\AtBeginDocument{
  \newgeometry{
    textheight=9in,
    textwidth=6in,
    top=1in,
    headheight=12pt,
    headsep=25pt,
    footskip=30pt
  }
}

\title{Fast convergence of trust-regions for non-isolated minima\\via analysis of CG on indefinite matrices}

\author{
  Quentin Rebjock and Nicolas Boumal\thanks{Correspondence: quentin.rebjock@epfl.ch. Ecole Polytechnique F\'ed\'erale de Lausanne (EPFL), Insitute of Mathematics. This work was supported by the Swiss State Secretariat for Education, Research and Innovation (SERI) under contract number MB22.00027.}
}

\date{\today}

\begin{document}

\maketitle

\begin{abstract}
  Trust-region methods (TR) can converge quadratically to minima where the
  Hessian is positive definite.
  However, if the minima are not isolated, then the Hessian there cannot be
  positive definite.
  The weaker \polyakloja{} (\pl{}) condition is compatible with non-isolated
  minima, and it is enough for many algorithms to preserve good local behavior.
  Yet, TR with an \emph{exact} subproblem solver lacks even basic features such
  as a capture theorem under \pl{}.

  In practice, a popular \emph{inexact} subproblem solver is the truncated
  conjugate gradient method (\tcg{}).
  Empirically, TR-\tcg{} exhibits superlinear convergence under \pl{}.
  We confirm this theoretically.

  The main mathematical obstacle is that, under \pl{}, at points arbitrarily
  close to minima, the Hessian has vanishingly small, possibly negative
  eigenvalues.
  Thus, \tcg{} is applied to ill-conditioned, indefinite systems.
  Yet, the core theory underlying \tcg{} is that of CG, which assumes a positive
  definite operator.
  Accordingly, we develop new tools to analyze the dynamics of CG in the
  presence of small eigenvalues of any sign, for the regime of interest to
  TR-\tcg{}.
\end{abstract}

\section{Introduction}\label{sec:intro}

We consider unconstrained optimization problems of the form
\begin{align*}
  \min_{x \in \manifold} \mfc(x),
\end{align*}
where $\manifold$ is a Riemannian manifold\footnote{The contributions are
  relevant for $\manifold = \mathbb{R}^n$ too. We treat the more general
  manifold case as it involves only mild overhead in notation, summarized in
  Table~\ref{table:euclidean-case}.}
and $\mfc \colon \manifold \to \reals$ is twice continuously differentiable
($\smooth{2}$).
The tangent spaces $\tangent_x \manifold$ are equipped with inner products
$\inner{\cdot}{\cdot}_x$ and associated norms $\|\cdot\|_x$ (and we omit the
subscript $x$ for brevity).

Near a local minimum $\optpoint$, classical results guarantee favorable local
convergence properties for standard algorithms when the Hessian $\hess
\mfc(\optpoint)$ is positive definite.
For example, gradient descent enjoys linear rates, while
regularized variants of Newton's method such as %
adaptive regularization with cubics (ARC) admit superlinear rates.
Those algorithms preserve their fast convergence rates even if we relax the
assumption at $\optpoint$ to assume instead the weaker \polyakloja{} (\pl{})
condition around $\optpoint$
(this is all well known: see Section~\ref{subsec:related-work} for a
literature review).

\begin{definition}\label{def:pl}
  Let $\optpoint$ be a local minimum of $\mfc$.
  We say $\mfc$ satisfies the \emph{\polyakloja{}} condition with constant
  $\plconstant > 0$ (also denoted \emph{$\plconstant$-\pl{}}) around $\optpoint$ if
  \begin{align}
    \mfc(x) - \mfcopt \leq \frac{1}{2\plconstant}\|\grad \mfc(x)\|^2
    \tag{\pl{}}
    \label{eq:pl}
  \end{align}
  for all $x$ in some neighborhood of $\optpoint$, where $\grad\mfc$ is the gradient of $\mfc$.
\end{definition}

The situation is different for trust-region methods (TR).
They are also known to enjoy superlinear convergence
near local minima with a positive definite Hessian.
However, in the literature there exist no fast (local) rates of convergence for
TR algorithms assuming only the \pl{} condition.

This is problematic if $\mfc$ has non-isolated local minima,
which is inevitable in overparameterized problems or problems with continuous symmetries.
Indeed, the Hessian cannot be positive definite at non-isolated local minima,
whereas \pl{} can (and often does) hold
there---see \citep{luo1993error} and \cite[\S4.2]{nesterov2006cubic}, or \citep{liu2022loss} in the context of machine learning.

This gap in the literature is all the more surprising considering that
\emph{(i)} empirically, practical implementations of TR behave just fine near
non-isolated minima with the \pl{} condition, and \emph{(ii)} the theoretical
guarantees for TR usually parallel those of ARC, which is known to converge
quadratically under \pl{}
\citep{nesterov2006cubic,cartis2011adaptive,yue2019quadratic,rebjock2023fast}.

Both TR and ARC involve a subproblem at each iteration.
For ARC, fast convergence rates hold under \pl{} for various subproblem solvers,
both exact and inexact.
Surprisingly, for TR, solving the subproblems \emph{exactly} can break basic capture properties under \pl{}:
we describe this in Section~\ref{subsec:intuition}.

Practical implementations of TR seldom rely on an exact subproblem solver though.
Instead, a popular alternative is the truncated conjugate gradient (\tcg{})
method (see Algorithm~\ref{alg:tcg}).
This is an \emph{inexact} subproblem solver based on the conjugate gradient
(\cg{}) algorithm.
The core of \tcg{} is the same as \cg{} but it
terminates early if it detects negative curvature,
or it exceeds the trust-region radius,
or it produces a small enough residual.

Experimentally, we observe that TR with \tcg{} enjoys favorable convergence
properties around minima where \pl{} holds (Appendix~\ref{sec:example} provides
a simple numerical example).
The fact that this good behavior relies on
a specific subproblem solver (which was not the case for ARC)
may partly explain the literature gap---and highlights yet another
remarkable property of Krylov methods.

In this paper, we explain that behavior theoretically.
Specifically, we secure superlinear convergence (but not
quadratic, see Remark~\ref{rmk:super-linear-but-not-quadratic}) for
TR with \tcg{} under \pl{}.

In our main theorem below,
\aref{assu:hess-lip}, \aref{assu:hess-lip-like}
and~\aref{assu:hess-approx} are three weak assumptions that typically hold, and
that we describe in Section~\ref{sec:application-rtr}.
For example, in the Euclidean case, they hold if the algorithm has access to the
true Hessian and the latter is locally Lipschitz continuous.
The proof is stated at the end of Section~\ref{sec:application-rtr}.

\begin{theorem}\label{th:rtr-main-theorem}
  Suppose~\aref{assu:hess-lip}, \aref{assu:hess-lip-like},
  \aref{assu:hess-approx} and~\eqref{eq:pl} hold around a local minimum
  $\optpoint$.
  We run TR with the \tcg{} subproblem solver (Algorithm~\ref{alg:tcg}) with
  parameters $\tcgkappa > 0$ and $\power \in \interval[open]{0}{1}$.
  Given any neighborhood $\mathcal{U}$ of $\optpoint$, there exists a
  neighborhood $\mathcal{V}$ of $\optpoint$ such that if an iterate enters
  $\mathcal{V}$ then the iterates converge superlinearly with order $1 +
  \power$ to some local minimum of $\mfc$ that is in $\mathcal{U}$.
\end{theorem}

To establish this theorem, we show several intermediate results of independent
interest regarding Krylov methods.
Indeed, a major obstacle to analyzing the behavior of \tcg{} without assuming
local strong convexity is that the Hessian can have small negative eigenvalues
arbitrarily close to local minima, whereas the usual analyses of \cg{} break if
the Hessian is not positive definite.
Accordingly, we propose a new analysis of \cg{} that can handle ill-conditioned, indefinite
matrices in Sections~\ref{sec:reminders-cg} and~\ref{sec:indefinite-cg}.

\begin{table}[t]
  \centering
  \begin{tabular}{|c|c|c|c|c|c|c|c|c|c|c|}
    \hline
    & $\tangent_x \manifold$ & $\inner{u}{v}$ & $\|u\|$ & $\retr_x(s)$ & $\Exp_x(s)$ & $\Log_x(y)$ & $\dist(x, y)$\\\hline
    $\manifold = \reals^n$ & $\reals^n$ & $u^\top v$ & $\sqrt{u^\top u}$ & $x + s$ & $x + s$ & $y - x$ & $\|x - y\|$\\\hline
  \end{tabular}
  \caption{
    Simplifications in the case where $\manifold$ is a Euclidean space.
  }\label{table:euclidean-case}
\end{table}

\subsection{Sufficient conditions for superlinear local convergence}\label{subsec:C0C1C2}

TR methods generate a sequence of iterates $x_k \in \manifold$ together with a
sequence of trust-region radii $\Delta_k > 0$.
At iteration $k$, a subproblem solver chooses a
step $s_k$ in the tangent space $\tangent_{x_k}\manifold$.
Then, $x_{k+1}$ is set to be either $x_{k+1} = \retr_{x_k}(s_k)$ (accepted step)
or $x_{k+1} = x_k$ (rejected step), where $\retr$ is a \emph{retraction} on
$\manifold$.
(If $\manifold = \reals^n$, then $\tangent_{x}\manifold = \reals^n$ and
typically $\retr_{x}(s) = x+s$.)
The step $s_k$ is chosen to approximately minimize a quadratic model
$\rtrmodel_k \colon \tangent_{x_k} \manifold \to \reals$ under
the constraint $\|s_k\| \leq \Delta_k$, where
\begin{align*}
    \rtrmodel_k(s) = \mfc(x_k) + \inner{s}{\grad \mfc(x_k)} + \frac{1}{2}\inner{s}{\linearmap_k[s]}.
\end{align*}
Above, $\linearmap_k$ is a symmetric linear map (often set to be $\hess f(x_k)$)
so that $\rtrmodel_k(s) \approx \mfc(\retr_{x_k}(s))$ and $\rtrmodel_k(\zeros) =
\mfc(x_k)$---Section~\ref{sec:application-rtr} provides the details.

In Propositions~\ref{prop:rtr-tcg-capture} and~\ref{prop:conv-rate}, we identify
three sufficient conditions on the subproblem solver's choice of $s_k$ in order
to secure superlinear local convergence of $\sequence{x_k}$ to a single point,
assuming \pl{}.
Such conditions transpire in other proofs of superlinear convergence,
see for example~\cite[\S4.2]{absil2007trust} for TR
and~\cite[\S4]{yue2019quadratic} for regularized Newton.

The first condition is that $\rtrmodel_k(s_k)$ (which is approximately
$\mfc(x_{k+1})$ for accepted steps) is smaller than $\rtrmodel_k(\zeros) =
\mfc(x_k)$, by some amount known as the \emph{Cauchy decrease}.
\begin{condition}\label{cond:better-than-cauchy}
  There exists a constant $\rtrsufficientdecrease > 0$ such that, for all $k$,
  the step $s_k$ satisfies
  \begin{align*}
    \rtrmodel_k(\zeros) - \rtrmodel_k(s_k) \geq \rtrsufficientdecrease \|\grad \mfc(x_k)\| \min\!\bigg(\Delta_k, \frac{\|\grad \mfc(x_k)\|^3}{\big|\inner{\grad \mfc(x_k)}{\linearmap_k[\grad \mfc(x_k)]}\big|}\bigg).
  \end{align*}
\end{condition}
Essentially all reasonable subproblem solvers
satisfy~\cref{cond:better-than-cauchy}.
The next two conditions control the behavior near local minima.
Specifically, near a local minimum $\optpoint$,
steps should be small and be good approximate critical points of the model.
\begin{condition}\label{cond:strong-vs}
  There exist a constant $\strongvsconst \geq 0$ and a neighborhood $\mathcal{U}$ of $\optpoint$ such that if an
  iterate $x_k$ is in $\mathcal{U}$ then the step $s_k$ satisfies
  \begin{align*}
    \|s_k\| \leq \strongvsconst \|\grad \mfc(x_k)\|.
  \end{align*}
\end{condition}
\begin{condition}\label{cond:small-model}
  If the iterates converge to a local minimum, there exist constants
  $\modelconst \geq 0$ and $\power > 0$ such that
  \begin{align*}
    \|\grad \rtrmodel_k(s_k)\| \leq \modelconst \|\grad \mfc(x_k)\|^{1 + \power}
  \end{align*}
  for all $k$ large enough.
\end{condition}

An \emph{exact} subproblem solver certainly
satisfies~\cref{cond:better-than-cauchy}~\citep[(7.14)]{absil2008optimization}.
If the iterates $x_k$ converge to a local minimum where the Hessian is positive
definite, then an exact subproblem solver with $\linearmap_k = \hess f(x_k)$
also satisfies~\cref{cond:strong-vs} and~\cref{cond:small-model}.
This is because eventually the method produces Newton steps $s_k = -\hess
\mfc(x_k)^{-1}[\grad \mfc(x_k)]$ with $\|s_k\| \leq \|\hess \mfc(x_k)^{-1}\|
\|\grad \mfc(x_k)\| < \Delta_k$, hence $\grad \rtrmodel_k(s_k) = \zeros$.
(See for example~\citep[(7.9) with $\mu = 0$]{absil2008optimization}.)
Things are markedly different without the positive definite Hessian assumption, as we now discuss.

\subsection{Why an exact subproblem solver can fail yet \tcg{} succeeds}\label{subsec:intuition}

In general, if we assume only the \pl{} condition, exact subproblem solvers may
fail~\cref{cond:strong-vs}.
To understand why, it is useful to note that, for $\smooth{2}$ functions, the local \pl{} condition is
\emph{equivalent} to the \morsebott{} (\mb{}) property, as we define
now---see~\cite[\S2]{rebjock2023fast}.
Let
\begin{align}
  \optimalset & = \{x \in \manifold : x \text{ is a local minimum of $\mfc$}\}
  \label{eq:def-s}
\end{align}
be the set that contains all the local minima of $\mfc$.
If $\optimalset$ is an (embedded) submanifold of $\manifold$, then at each
$\optpoint \in \optimalset$ it has a \emph{tangent space}
$\tangent_{\optpoint}\optimalset$ whose orthogonal complement is the
\emph{normal space} $\normal_{\optpoint}\optimalset$.
\begin{definition}\label{def:mb}
  We say $\mfc$ satisfies the \emph{\morsebott{}} property at a local minimum
  $\optpoint \in \optimalset$ if
  \begin{align}\label{eq:morse-bott}\tag{\mb{}}
    \optimalset \textrm{ is a $\smooth{1}$ submanifold around } \optpoint
      && \textrm{ and } &&
    \ker\hess \mfc(\optpoint) = \tangent_{\optpoint}\optimalset.
  \end{align}
  If also $\inner{v}{\hess \mfc(\optpoint)[v]} \geq \plconstant\|v\|^2$ for some
  $\mu > 0$ and all $v \in \normal_{\optpoint}\optimalset$ then we say $\mfc$
  satisfies \emph{$\plconstant$-\mb{}} at $\optpoint$.
\end{definition}
Figure~\ref{fig:mb} illustrates the \mb{} property.
\begin{figure}[t]
  \centering
  \includegraphics[width=0.5\textwidth]{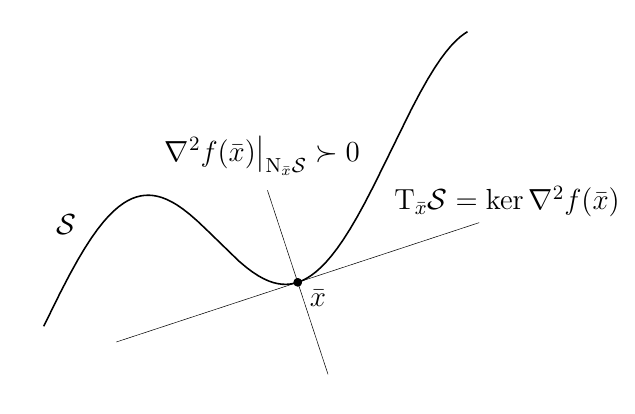}
  \caption{Illustration of the \morsebott{} property.
    The set of local minima $\optimalset$ is smooth around the point $\optpoint$.
    Here it has dimension 1 in the 2-dimensional search space $\manifold = \reals^2$.}\label{fig:mb}
\end{figure}
Notice that the dimension of $\optimalset$ as a manifold coincides
with the dimension of the kernel of $\hess\mfc(\optpoint)$.
We now suppose that $\mfc$ satisfies $\plconstant$-\pl{}
around a local minimum $\optpoint$ and exploit the fact that this
implies $\plconstant$-\mb{} at $\optpoint$
to describe the behavior of both the exact subproblem solver
and that of \tcg{} around $\optpoint$.

\paragraph{The Hessian typically has small negative eigenvalues arbitrarily
  close to $\optpoint$.}

To see this, consider a small ball $B$ around $\optpoint$ and assume for
contradiction that $\hess \mfc$ is positive semidefinite at all points in $B$.
If so, then $\mfc$ restricted to $B$ is (geodesically) convex, hence its
minimizers form a convex set.
Yet, the minimizers of $\mfc$ restricted to $B$ coincide with $\optimalset \cap B$, %
and there is no reason a priori that this ought to be convex.
Thus, save for that unusual case,
there exist points $x$ arbitrarily close to $\optpoint$ where $\hess \mfc(x)$ has a (small) negative eigenvalue.
(Appendix~\ref{sec:example} gives an explicit example.)

\paragraph{Those eigenvalues defeat the exact subproblem solver.}

Suppose that an iterate $x_k$ is at such a point where $\hess \mfc(x_k)$ has a
negative eigenvalue,
and that $\linearmap_k = \hess \mfc(x_k)$.
This implies that the exact solution $s_k$ to the TR subproblem lies on the
boundary of the trust region:
$\|s_k\| = \Delta_k$~\citep[Thm.~4.1]{nocedal2006numerical}.
This precludes capture results because the next iterate can be far ($\Delta_k$
may be large) even when $x_k$ is arbitrarily close to $\optpoint$
(see~\cite[\S4.2]{rebjock2023fast}).
In particular, $\|s_k\| = \Delta_k$ is in general
incompatible with condition~\cref{cond:strong-vs}.
This issue occurs because the exact subproblem solver is highly sensitive to
negative eigenvalues, even of small magnitude.

\paragraph{But \tcg{} automatically filters them out.}

At $\optpoint$, the Hessian has a kernel whose dimension is the same as that
of $\optimalset$ because~\eqref{eq:morse-bott} holds.
All the other eigenvalues are strictly positive.
Given a point $x$ close to $\bar x$, these positive eigenvalues remain large
compared to the others, which are concentrated around zero.
This defines a clear separation of the tangent space $\tangent_x \manifold$ in
two orthogonal complements: the primary space (large eigenvalues) and the
secondary space (eigenvalues close to zero).
Around $\optpoint$,
the gradient of $\mfc$ is almost orthogonal to the secondary space (see
Lemma~\ref{lemma:grad-image-hess}).
It follows that the Krylov space generated by $\hess \mfc(x)$ and $\grad
\mfc(x)$ essentially ignores the secondary components up to a certain number of
iterations of \cg{}.
Beyond this critical number of iterations, the iterates of \cg{} could explode
because the algorithm detects the small eigenvalues of $\hess \mfc(x)$.
However, one of the key fact we secure in this paper is that the natural
stopping criteria of \tcg{} trigger \emph{before} this explosion can happen (see
Figure~\ref{fig:vn-rn-norms}).
We analyze these dynamics in Section~\ref{subsec:tcg-c1-c2}, and deduce that
\tcg{} satisfies~\cref{cond:strong-vs} and~\cref{cond:small-model}.

\begin{figure}[t]
    \centering
    \begin{minipage}{0.495\linewidth}
        \includegraphics[width=1\linewidth]{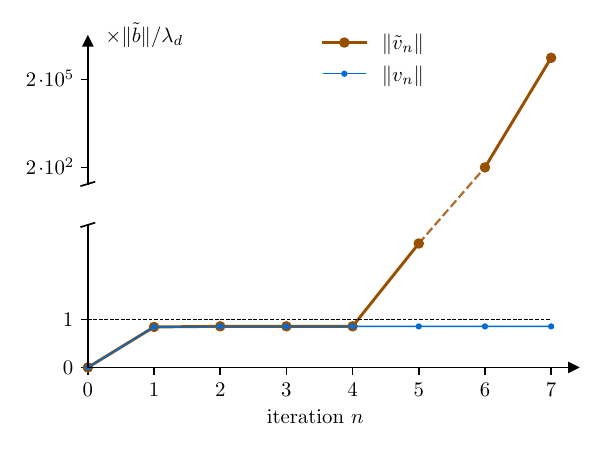}
    \end{minipage}
    \begin{minipage}{0.495\linewidth}
        \includegraphics[width=1\linewidth]{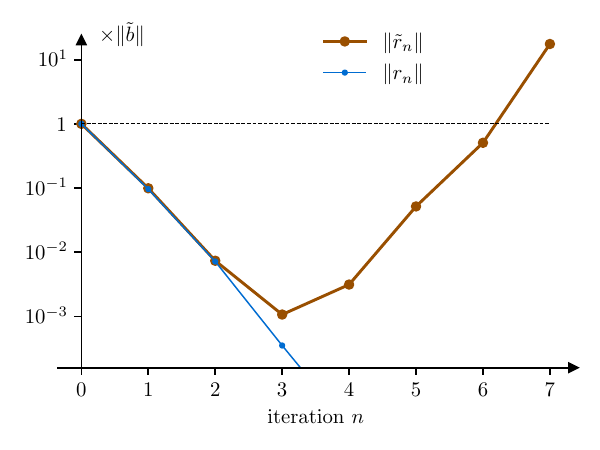}
    \end{minipage}
    \caption{
        Norms of the iterates $\tilde \cgvec_n$ and residuals $\tilde \cgres_n$
        of \cg{} on a problem $(\tmat, \tweight)$.
        Here $\tmat$ is diagonal with size $\mdim = 11$.
        For illustration, there are $\sdim = 10$ eigenvalues close to $1$ and $1$
        eigenvalue equal to 0.
        The norm of the first $\sdim$ entries of the weight vector $\tweight$ is
        normalized to 1 and the entry associated to the zero eigenvalue is
        $10^{-3}$.
        Notice how the norm of the iterate $\tilde \cgvec_n$ explodes only
        \emph{after} the residual $\tilde \cgres_n$ became small: this is why \tcg{}
        can stop \emph{before} explosion, with a good solution.
        For reference, we also plot the same quantities for the well-conditioned problem $(\mat, \weight)$ of size $10$, where the zero eigenvalue was removed.
    }\label{fig:vn-rn-norms}
\end{figure}

\subsection{Contributions}

In the first part of the paper (Sections~\ref{sec:reminders-cg} and~\ref{sec:indefinite-cg}),
we design tools to study the dynamics of \cg{} when the input matrix
has eigenvalues of small magnitude that may be negative.
\begin{itemize}
\item From Section~\ref{subsec:intuition}, we find that we need to understand
  the behavior of \cg{} on systems $(\tmat, \tweight)$ where $\tmat$ may have
  some small (possibly negative) eigenvalues, and the corresponding components
  of $\tweight$ are small.
  Existing theory does not handle that, but empirically we see that the initial
  iterates of \cg{} on $(\tmat, \tweight)$ are closely related to those of \cg{}
  on $(\mat, \weight)$, where $\mat$ is the well-conditioned, positive definite
  part of $\tmat$, and $\weight$ is the corresponding part of $\tweight$.
  We make this precise in Theorem~\ref{th:bounds-iterates}.
  See also Figure~\ref{fig:vn-rn-norms}.
\item In order to do this, we first relate the \lanczos{} polynomials associated
  to these two problems (Lemmas~\ref{lemma:identity-rhos}
  and~\ref{lemma:identity-rhos-2}).
  We deduce sufficient conditions for an iteration of \cg{} on $(\tmat,
  \tweight)$ to be well defined (Lemma~\ref{lemma:bound-roots} and
  Corollary~\ref{cor:roots-lower-bound}),
  in which case we relate the \cg{} polynomials
  associated to the two problems (Theorem~\ref{th:identity-qs}).
\item That is only made possible by first introducing a special (and possibly
  new) family of polynomials in~Section~\ref{subsec:prexipols}.
  We work out their properties with respect to the
  \lanczos{} polynomials.
\item In Section~\ref{subsec:effective-regime} we particularize to a regime where the above results are actionable.
  In particular, we exhibit an iteration of \cg{} with bounds for the iterate
  and residual norms (Lemma~\ref{lemma:c-bounds}).
\end{itemize}
The second part (Section~\ref{sec:application-rtr}) of the paper applies the
aforementioned results to TR-\tcg{}.
\begin{itemize}
\item Our main result is Theorem~\ref{th:rtr-main-theorem}.
  Given a local minimum where \pl{} holds, it ensures capture of the iterates
  and superlinear convergence for TR with the \tcg{} subproblem solver.
\item Lemma~\ref{lemma:cg-iterate-existence} is the keystone to establish
  Theorem~\ref{th:rtr-main-theorem}.
  The proof relies on Section~\ref{sec:indefinite-cg}.
  Roughly, it states that if $x$ is a point near a local minimum where \pl{}
  holds, the iterates of \cg{} on $(\hess \mfc(x), -\grad \mfc(x))$ are almost
  oblivious to the small eigenvalues of the Hessian, for a while.
\item From Lemma~\ref{lemma:cg-iterate-existence} we deduce that TR with \tcg{}
  satisfies conditions~\cref{cond:strong-vs} and~\cref{cond:small-model} around
  minima where \pl{} holds (Propositions~\ref{prop:rtr-tcg-strong-vs}
  and~\ref{prop:rtr-tcg-small-model}).
  This is because the early truncation rules of \tcg{} trigger before the small
  eigenvalues can cause harm.
\item Finally, in Section~\ref{subsec:captureandorder} we show that any
  subproblem solver that satisfies~\cref{cond:better-than-cauchy},
  \cref{cond:strong-vs} and~\cref{cond:small-model} (\tcg{} in particular) enables fast local convergence
  (Propositions~\ref{prop:rtr-tcg-capture} and~\ref{prop:conv-rate}).
\end{itemize}
As an aside, we deduce implications for optimization on quotient manifolds in Section~\ref{sec:quotients}.

\subsection{Related work}\label{subsec:related-work}

\paragraph{Krylov subspace methods.}

The history of Krylov methods is broad.
\cite{lanczos1950iteration} introduced his famous algorithm to find extreme
eigenpairs of a matrix.
Soon after, \cite{hestenes1952methods} described the \cg{}
algorithm to solve systems of linear equations.
These two algorithms are tightly related.
Standard references include~\citep{greenbaum1997iterative}
and~\citep{parlett1998symmetric}.
It is classical to analyze Krylov subspace methods with
polynomials~\cite[Ch.~2,~3]{liesen2013krylov}.
This is because both the \lanczos{} and \cg{} algorithms are linked to the
theory of orthogonal polynomials and Gauss
quadrature~\cite[Ch.~4]{golub2010matrices}.

Analyses of \cg{} normally assume a positive definite matrix, whereas we need to
handle (small) negative eigenvalues as well.
There is a rich literature on \cg{} with inexact
arithmetic~\citep{paige1971computation,greenbaum1989behavior,greenbaum1992predicting}.
See~\citep{meurant2006lanczos}
and~\citep[p.~475--476]{meurantstrakos2006lanczos} for more pointers.
Potentially, that line of work could have provided a foundation to build on.
However, to the best of our knowledge, those analyses tend to assume sufficient
positive definiteness to withstand rounding errors, and hence indefinite
matrices are not (even indirectly) covered by their conclusions.

\paragraph{Krylov and trust-region methods.}

It is possible to solve the trust-region subproblem~\eqref{eq:rtr-subproblem}
exactly.
Several algorithms have been proposed for this;
for example, see~\citep{more1983computing}, \citep{adachi2017solving},
and \citep{carmon2020first}.

Solving the subproblem is only a means to an end (minimizing the cost
function $\mfc$).
For this reason, many practical implementations of TR
solve the subproblem only \emph{approximately}.
To do this, the \lanczos{} and \cg{} algorithms (and variants) have been
extensively used as subroutines.
The truncated CG (\tcg{}) algorithm is one particularly popular
approach introduced by~\cite{toint1981towards} and~\cite{steihaug1983conjugate}.
At the same time,~\cite{dembo1982inexact,dembo1983truncated} proved that inexact
Newton methods preserve fast convergence to non-degenerate critical points even
with truncated steps.
See also ``notes and references'' in~\cite[\S7.5.1]{conn2000trust} for more
historic notes.
\cite{gould1999solving} proposed an algorithm to continue the \tcg{} process
after the algorithm reaches the boundary of the trust region.
\cite{yuan2000truncated} proved that the model decrease resulting from the
\tcg{} step is at least half that of the exact solution when the Hessian is
positive definite.

The minimal residual method (\minres{}, introduced by~\cite{paige1975solution})
is closely related to \cg{}.
\cite{fong2012cg} empirically compared \cg{} and \minres{}.
They suggested that \minres{} may be preferable when the algorithm is meant to stop
early (as is the case for the TRS), and they proposed partial theoretical explanations.
Later,~\cite{liu2022minres,liu2022newton} scrutinized the preeminence of \cg{}
to solve subproblems in Newton-type methods and also argued that \minres{} is a
reasonable choice.
They proved that \minres{} enjoys many favorable properties and proposed a
Newton-type algorithm based on it.
Analyzing the local behavior of \minres{} under the \pl{} condition could be an
interesting direction to extend our work.

\paragraph{\pl{} and local convergence.}

In seminal articles, \cite{lojasiewicz1963propriete,lojasiewicz1982trajectoires}
showed functions in a large class satisfy special inequalities, and he used them
to study continuous dynamical systems.
Concurrently,~\cite{polyak1963gradient} proved that a global
version of those inequalities is sufficient for gradient descent
to converge linearly.
Since then, the condition is often called \pl{} (also ``gradient dominance'').
Variations have been extensively invoked in optimization literature,
notably~\citep{attouch2010proximal,attouch2013convergence,bolte2014proximal,karimi2016linear}.
We focus here on the interactions between \pl{} and second-order algorithms.

\cite{nesterov2006cubic} popularized the regularized Newton algorithm
of \cite{griewank1981modification} and proved local superlinear
convergence of order $4/3$ assuming \pl{}.
Later, \cite{cartis2011adaptive,cartis2011adaptive2} proposed an adaptive
variant (ARC) and proved quadratic convergence to points where the Hessian is positive
definite.
\cite{zhou2018convergence} characterized the local convergence rate of
regularized Newton depending on the \loja{} exponent of the problem.
\cite{yue2019quadratic} proved that regularized Newton
with exact subproblem solver converges quadratically
to minimizers where the \emph{error bound} condition holds.
This condition is in fact equivalent to \pl{};
\cite{rebjock2023fast}
leveraged this to obtain quadratic convergence for ARC,
with both exact and inexact subproblem solvers.

In contrast, for trust-region methods there exist superlinear convergence
results in two cases: \emph{(i)} for general $\mfc$ with positive definite
Hessian (see for example~\cite[\S4.4]{nocedal2006numerical}
and~\citep{absil2007trust}), and \emph{(ii)} assuming \pl{} but specialized to
nonlinear least squares $x \mapsto \|F(x)\|^2$, around a global minimum
$\bar x$ satisfying $F(\bar x) = \zeros$~\citep{fan2006convergence}.
In that work, the subproblems are modified in a \levenmarq{} way, which
ensures semidefinite Hessian approximations.

\section{\cg{}: reminders and a possibly new family of polynomials}\label{sec:reminders-cg}

This section collects properties of \cg{} we need in
Section~\ref{sec:indefinite-cg}.
The first part recalls classical theory.
The second part introduces key polynomials.
In this whole section, we let $\mat$ be a $\sdim \times \sdim$ symmetric matrix
and $\weight \in \reals^\sdim$.
Since \cg{} is equivariant with respect to orthogonal transformations, we assume
without loss of generality that $\mat$ is diagonal, that is,
\begin{align*}
  \mat = \diag(\lambda_1, \dots, \lambda_\sdim) && \text{where} && \lambda_1 \geq \dots \geq \lambda_\sdim.
\end{align*}
However, we do \emph{not} assume that the input matrix $\mat$ is positive
definite nor that it is well conditioned.
We consider exact arithmetic.

Many of the results below rely on polynomials (summarized in
Table~\ref{tab:pols}).
The set $\pols{n}$ contains all univariate polynomials with real coefficients
and of degree at most $n$.
Given a polynomial $\pol$, we let $\polcoef{n}{\pol}$ denote its coefficient of
degree $n$.
Furthermore, $\restr{\mat}{\krylovspace}{\normalscaling}$ is the
restriction of a linear operator $\mat$ to a subspace $\krylovspace$.

\subsection{Background: \cg{} through the lens of optimal polynomials}\label{subsec:cg-background}

\begin{algorithm}[t]\caption{Truncated conjugate gradient (\tcg{})}\label{alg:tcg}
  \begin{algorithmic}[1]
    \State \textbf{Parameters:} $\tcgkappa > 0$, $\power \in \interval[open left]{0}{1}$
    \State \textbf{Input:} $\mat \in \reals^{\sdim \times \sdim}$ symmetric, $\weight \in \reals^\sdim$,
    $\Delta > 0$
    \State $\cgvec_0 = \zeros$, $\cgres_0 = b$, $\cgdir_0 = b$
    \If{$\|\cgres_0\| = 0$}
      \State \textbf{output $\cgvec_0$}
    \EndIf
    \For{$n = 1,2,\dots$}
      \State $\alpha_n = \frac{\|\cgres_{n - 1}\|^2}{\inner{\cgdir_{n - 1}}{\mat \cgdir_{n - 1}}}$\label{line:inner-product}
      \State $\cgvec_{n - 1}^+ = \cgvec_{n - 1} + \alpha_n \cgdir_{n - 1}$
      \If{$\inner{\cgdir_{n - 1}}{\mat \cgdir_{n - 1}} \leq 0$ or $\|\cgvec_{n -
            1}^+\| \geq \Delta$} \Comment{truncation 1}\label{line:truncation-1}
        \State $\cgvec_n = \cgvec_{n - 1} + t \cgdir_{n - 1}$ with $t \geq 0$ such that $\|\cgvec_n\| =
        \Delta$
        \State \textbf{output $\cgvec_n$}
      \EndIf\label{line:truncation-1-end}
      \State $\cgvec_n = \cgvec_{n - 1}^+$
      \State $\cgres_n = \cgres_{n - 1} - \alpha_n \mat \cgdir_{n - 1}$
      \If{$\|\cgres_n\| \leq \|\cgres_0\|\min(\|\cgres_0\|^\power, \tcgkappa)$}
        \Comment{truncation 2}\label{line:truncation-2}
        \State \textbf{output $\cgvec_n$}
      \EndIf\label{line:truncation-2-end}
      \State $\beta_n = \frac{\|\cgres_n\|^2}{\|\cgres_{n - 1}\|^2}$
      \State $\cgdir_n = \cgres_n + \beta_n \cgdir_{n - 1}$
    \EndFor{}
  \end{algorithmic}
\end{algorithm}
Algorithm~\ref{alg:tcg} describes \tcg{}.
\cg{} is the same but without the two truncation parts in
lines~\ref{line:truncation-1}--\ref{line:truncation-1-end}
and~\ref{line:truncation-2}--\ref{line:truncation-2-end}.
We always consider that the starting vector is $\cgvec_0 = \zeros$.
This section provides some classical background about \cg{}.
Given an integer $n \in \{1, \dots, d\}$, we define the $n$th Krylov matrix and
the associated subspace as
\begin{align*}
  \krylovmat_n =
  \begin{bmatrix}
    \weight & \mat \weight & \cdots & \mat^{n - 1} \weight
  \end{bmatrix}
  && \text{and} &&
                   \krylovspace_n = \vecspan \krylovmat_n.
\end{align*}
\begin{definition}\label{def:grade}
  The \emph{grade} $\grade = \grade(\mat, \weight)$ is the largest integer $n$ such
  that $\krylovspace_n$ has dimension $n$.
\end{definition}
\begin{definition}\label{def:well-defined}
    The $n$th iteration of CG is \emph{well defined} if $n \leq \grade$ and $\mat|_{\krylovspace_n} \succ 0$.
\end{definition}
Iteration $n$ is well defined exactly if the inner products in
line~\ref{line:inner-product} are (strictly) positive up to the $n$th
iteration~\cite[Thm.~38.1]{trefethen1997numerical}.
We consider only those iterations.

\paragraph{\lanczos{} polynomials.}

\cg{} is an iterative procedure to solve $Ax = b$.
The \lanczos{} algorithm is an iterative procedure to compute extreme eigenpairs.
Their iterations are related to a
sequence of monic orthogonal polynomials that we describe now.
Define the bilinear form
\begin{align}\label{eq:semi-inner-product}
  \inner{\pol}{\qol} = \sum_{i = 1}^\sdim \weight_i^2 \pol(\lambda_i)\qol(\lambda_i)
\end{align}
on $\pols{\grade}$ (the linear space of polynomials of degree up to $\grade$).
It is a semi-inner product which is positive definite on $\pols{\grade - 1}$.
The form $\inner{\cdot}{\cdot}$ induces a semi-norm $\|\cdot\|$ on
$\pols{\grade}$.
With $\polcoef{n}{p}$ denoting the coefficient of degree $n$ for the polynomial
$p$, the $n$th \emph{\lanczos{} polynomial} is the (unique) solution to the
convex optimization problem
\begin{align}\label{eq:def-rho-pol}
  \rhopol _n = \argmin_{\rhopol \in \pols{n}} \|\rhopol\|^2 \qquad \text{subject to} \qquad \polcoef{n}{\rhopol} = 1.
\end{align}
In particular, $\rhopol_n$ is monic of degree $n$.
The kernel of $\inner{\cdot}{\cdot}$ is $\vecspan \rhopol_\grade$.
The following lemma states that the polynomials $\sequence{\rhopol_i}$ are well
defined and it describes some of their properties.

\begin{lemma}\label{lemma:rho-basic-props}
  For all $n \in \{0, \dots, \grade\}$ there exists a unique solution
  $\rhopol_n$ to~\eqref{eq:def-rho-pol}.
  The $n$ roots of $\rhopol_n$ are real, distinct, and in the interval
  $\interval{\lambda_\sdim}{\lambda_1}$.
  The polynomials $\rhopol_0, \dots, \rhopol_\grade$ are orthogonal for~\eqref{eq:semi-inner-product}.
\end{lemma}
\begin{proof}
  This is classical: see for
  example~\cite[Lem.~3.2.2, Lem.~3.2.4]{liesen2013krylov}.
\end{proof}

It is also known that the roots of $\rhopol_n$ and $\rhopol_{n + 1}$ interlace
but we are not going to use this fact.

\begin{remark}\label{remark:grade}
  Given an eigenvalue $\lambda$ of $\mat$, let $P_\lambda$ be the orthogonal
  projector onto the eigenspace of $\mat$ associated to $\lambda$.
  Define the \emph{weight} of $\lambda$ as $\|P_\lambda \weight\|$.
  The grade $\grade$ of $(\mat, \weight)$ is the number of distinct eigenvalues
  of $\mat$ with (strictly) positive weight.
  It is also the smallest integer $n$ such that $\|\rhopol_n\| = 0$.
\end{remark}

With the orthogonality of $\rhopol_0, \dots, \rhopol_\grade$ and a simple
application of the Pythagorean theorem we obtain the following decomposition.

\begin{lemma}\label{lemma:ortho-decomp}
  Given $n \in \{0, \dots, \grade\}$ and a polynomial $\pol \in \pols{n}$, the
  following decomposition holds:
  \begin{align*}
    \pol = \polcoef{n}{\pol} \rhopol_n + \sum_{i = 0}^{n - 1} \frac{\inner{\pol}{\rhopol_i}}{\|\rhopol_i\|^2}\rhopol_i.
  \end{align*}
\end{lemma}
\begin{proof}
  The polynomial $q = p - \polcoef{n}{\pol}\rhopol_n$ has degree at most $n - 1$.
  Decompose $q$ in the orthogonal basis $\{\rhopol_0, \dots, \rhopol_{n - 1}\}$
  (Lemma~\ref{lemma:rho-basic-props}) and leverage $\inner{\rhopol_n}{\rhopol_i}
  = 0$ for all $i < n$ to obtain the identity.
\end{proof}

The optimality conditions of the convex optimization
problem~\eqref{eq:def-rho-pol} entirely characterize the polynomial
$\rhopol_n$.

\begin{lemma}\label{lemma:rho-opt-cond}
  For $n \in \{0, \dots, \grade\}$, the polynomial $\rhopol_n$ defined
  in~\eqref{eq:def-rho-pol} is uniquely characterized by
  \begin{align}\label{eq:rho-opt-cond}
    \inner{\pol}{\rhopol_n} = \polcoef{n}{p}\|\rhopol_n\|^2
  \end{align}
  for all $\pol \in \pols{n}$.
  In particular, $\inner{\pol}{\rhopol_n} = 0$ when $\deg \pol \leq n - 1$.
\end{lemma}
\begin{proof}
  Problem~\eqref{eq:def-rho-pol} is convex with a unique solution for the given
  $n$.
  It is hence characterized by the first-order optimality conditions:
  $\inner{q}{\rhopol_n} = 0$ for all $q \in \pols{n - 1}$.
  Notice that a polynomial $\pol \in \pols{n}$ can be written $\pol =
  \polcoef{n}{\pol} \rhopol_n + q$ where $q \in \pols{n - 1}$.
  Plugging this in the optimality condition gives $\inner{\pol}{\rhopol_n} =
  \polcoef{n}{\pol} \|\rhopol_n\|^2$.
\end{proof}

The roots of $\rhopol_n$ are called the \emph{Ritz values}.
The following classical lemma~\citep[Thm.~36.1]{trefethen1997numerical} links
them to the eigenvalues of $\mat$ restricted to the $n$th Krylov subspace.

\begin{lemma}\label{lemma:roots-and-eigenvalues}
  For all $n \leq \grade$ the roots of $\rhopol_n$ coincide with the eigenvalues
  of $\restr{\mat}{\krylovspace_n}{\normalscaling}$.
\end{lemma}

\paragraph{Connections to the \cg{} algorithm.}\label{par:connections-cg}

The polynomials $\{\rhopol_i\}$ are related to the iterations of \cg{} as we
describe now; see~\cite[\S5.6]{liesen2013krylov} for details.
Given $n \leq \grade$, the $n$th iteration of \cg{} is well defined
(Definition~\ref{def:well-defined}) if and only if
$\restr{\mat}{\krylovspace_n}{\normalscaling}$ is positive definite.
This is equivalent to $\rhopol_n$ having only positive roots by
Lemma~\ref{lemma:roots-and-eigenvalues}.
In this case, we define the degree $n$ polynomial
\begin{align}\label{eq:qpol-def}
  \qpol_n = \frac{\rhopol_n}{\rhopol_n(0)}.
\end{align}
This is a rescaled version of $\rhopol_n$ such that $\qpol_n(0) = 1$.
Moreover, let $\ppol_{n - 1}$ be the degree $n - 1$ polynomial that satisfies
$\qpol_n(x) = 1 - x \ppol_{n - 1}(x)$.
We can understand the state of \cg{} at iteration $n$ from the polynomials
$\qpol_n$ and $\ppol_{n - 1}$.
Indeed, the $n$th iterate and residual of \cg{} are given by
\begin{align*}
  \cgvec_n = \ppol_{n - 1}(\mat)\weight && \text{and} && \cgres_n = \qpol_n(\mat)\weight
\end{align*}
respectively.
It also follows that the norm of the residual is $\|\cgres_n\| = \|\qpol_n\|$,
where the norm on the right-hand side is the one induced
by~\eqref{eq:semi-inner-product}.

\paragraph{\cg{} is a minimization algorithm.}
As long as the matrix $\mat$ is positive definite on the $n$th Krylov
subspace $\krylovspace_n$, the iterates $\cgvec_n$ of \cg{} minimize a quadratic form as
follows~\cite[\S2.5.3]{liesen2013krylov}:
\begin{align}\label{eq:cg-iterates-minimize}
  \cgvec_n = \argmin_{v \in \krylovspace_n} \frac{1}{2} v^\top \mat v - \weight^\top v && \text{and} && \cgres_n = \weight - \mat \cgvec_n.
\end{align}
Let the columns of $Q_n \in \reals^{\sdim \times n}$ be an orthonormal basis of $\krylovspace_n$.
Then from~\eqref{eq:cg-iterates-minimize} we can write $\cgvec_n = Q_n(Q_n^\top
\mat Q_n^{})^{-1}Q_n^\top \weight$.
The eigenvalues of $Q_n^\top \mat Q_n$ are the Ritz values, that is, the roots
of $\rhopol_n$ owing to Lemma~\ref{lemma:roots-and-eigenvalues}.
It follows that the spectrum of the matrix $Q_n(Q_n^\top \mat Q_n^{})^{-1}Q_n^\top$
contains $\sdim - n$ zeros and the other eigenvalues are the inverses of the
roots of $\rhopol_n$.

When the input matrix $\mat$ is positive definite we can see the polynomials
$\sequence{\qpol_i}$ as the solutions of particular optimization problems.
Remember that we do \emph{not} assume that $\mat$ is positive definite in this section.
However, we will apply this result to a positive definite submatrix in
Section~\ref{subsec:effective-regime}.

\begin{lemma}\label{lemma:cg-pol-min}
  Assume that $\lambda_1, \dots, \lambda_\sdim > 0$.
  For all $n \in \{0, \dots, \grade\}$ the polynomial $\qpol_n$, as defined
  in~\eqref{eq:qpol-def}, is the unique solution to the optimization problem
  \begin{align}\label{eq:cg-pol-min}
    \min_{\qpol \in \pols{n}} \; \sum_{i = 1}^\sdim \frac{\qpol(\lambda_i)^2}{\lambda_i} \weight_i^2 \qquad \text{subject to} \qquad \qpol(0) = 1.
  \end{align}
\end{lemma}
\begin{proof}
  Problem~\eqref{eq:cg-pol-min} is convex.
  The first-order optimality conditions state that a feasible polynomial $\qpol$
  is optimal if and only if $\sum_{i = 1}^\sdim \lambda_i^{-1} p(\lambda_i)
  \qpol(\lambda_i) \weight_i^2 = 0$ for all $p \in \pols{n}$ satisfying $p(0) =
  0$.
  The latter requires that $p(x) = x q(x)$ for some $q \in \pols{n - 1}$, hence
  $\qpol$ is optimal if and only if $\sum_{i = 1}^\sdim q(\lambda_i)
  \qpol(\lambda_i) \weight_i^2 = 0$ for all $q \in \pols{n - 1}$.
  We conclude with Lemma~\ref{lemma:rho-opt-cond}.
\end{proof}

\subsection{A possibly new family of polynomials}\label{subsec:prexipols}

\begin{table}
  \centering
  \begin{tabular}{ccccc}
    Polynomial & Definition & Normalization & Comments & Relation to \cg{}\\\noalign{\vskip 1pt}\hline\noalign{\vskip 2pt}
    $\rhopol_n$ & \multirow{2}{*}{minimizers of~\eqref{eq:both-def-rho-pol}} & \multirow{2}{*}{monic} & \multirow{2}{*}{\shortstack{Lemmas~\ref{lemma:rho-basic-props}, \ref{lemma:roots-and-eigenvalues}\\and~\ref{lemma:identity-rhos-2}}} & \\
    $\trhopol_n$ & & & & \\\noalign{\vskip 1pt}\hline\noalign{\vskip 2pt}
    $\qpol_n$ & $\qpol_n = \rhopol_n/\rhopol_n(0)$ & $\qpol_n(0) = 1$ & minimizer of~\eqref{eq:cg-pol-min} & $\cgres_n = \qpol_n(\mat) \weight$\\\noalign{\vskip 1pt}
    $\tqpol_n$ & $\tqpol_n = \trhopol_n/\trhopol_n(0)$ & $\tqpol_n(0) = 1$ & Theorem~\ref{th:identity-qs} & $\tilde \cgres_n = \tqpol_n(\tmat) \tweight$\\\noalign{\vskip 2pt}\hline\noalign{\vskip 2pt}
    $\ppol_{n - 1}$ & $\qpol_n(x) = 1 - x \ppol_{n - 1}(x)$ & \multirow{2}{*}{none} & Lemma~\ref{lemma:decreasing-ppols} & $\cgvec_n = \ppol_{n - 1}(\mat) \weight$\\\noalign{\vskip 1pt}
    $\tilde \ppol_{n - 1}$ & $\tqpol_n(x) = 1 - x \tilde \ppol_{n - 1}(x)$ & & & $\tilde \cgvec_n = \tilde \ppol_{n - 1}(\tilde \mat) \tweight$\\\noalign{\vskip 2pt}\hline\noalign{\vskip 2pt}
    $\prexipol_n$ & minimizer of~\eqref{eq:def-pre-xi-pol} & $\prexipol_n(\pxplambda) = 1$ & \multirow{2}{*}{\shortstack{$\prexipol_n^j = \prexipol_n$ with $\lambda = \lambda_j$\\interlace with $\rhopol_{n + 1}$}} & \\\noalign{\vskip 1pt}
    $\xipol_n$ & $\xipol_n = \prexipol_n/\prexipol_n(0)$ & $\xipol_n(0) = 1$ & & \\\noalign{\vskip 2pt}\hline
  \end{tabular}
  \caption{Summary of polynomials used in Sections~\ref{sec:reminders-cg}
    and~\ref{sec:indefinite-cg}.\label{tab:pols}}
\end{table}
In this section, we introduce special polynomials $\sequence{\prexipol_i}$ that
enjoy convenient properties with respect to the \lanczos{} polynomials
$\sequence{\rhopol_i}$ defined in~\eqref{eq:def-rho-pol}.
We have not seen these polynomials studied elsewhere, though they are intimately
related to
the minimal residual
method (\minres{}).\footnote{It is easy to reason that $\prexipol_n$ is the translation of a polynomial that
    naturally appears when studying the \minres{} algorithm on the problem $(\mat -
    \lambda I, \weight)$.}
Their extremal and interlacing properties play a key role to understand \cg{} on
two related instances in Section~\ref{sec:indefinite-cg}.
Remember that Table~\ref{tab:pols} summarizes all the polynomials that we
manipulate.

Fix some $\pxplambda \in \reals$ for this whole section.
(Later, we invoke its results for various values of $\pxplambda$.)
Then, for each $n < \grade$ let
\begin{align}\label{eq:def-pre-xi-pol}
  \prexipol_n = \argmin_{\prexipol \in \pols{n}} \|\prexipol\|^2 \qquad \text{subject to} \qquad \prexipol(\pxplambda) = 1.
\end{align}
There is indeed a unique solution to~\eqref{eq:def-pre-xi-pol} when $n < \grade$
because $\|\cdot\|^2$ is strongly convex on $\pols{\grade - 1}$.
(We do not need this, but when $n = \grade$ there is also a unique solution
provided that $\pxplambda \notin \{\lambda_1, \dots, \lambda_\sdim\}$.)

We show below that $\prexipol_n$ is of degree $n$ and that its roots are real
and distinct.
We first derive some basic properties.
The optimality conditions of the convex optimization
problem~\eqref{eq:def-pre-xi-pol} entirely characterize the polynomial
$\prexipol_n$.

\begin{lemma}\label{lemma:prexi-opt-cond}
  For $n \in \{0, \dots, \grade - 1\}$, the polynomial $\prexipol_n$ defined
  in~\eqref{eq:def-pre-xi-pol} is uniquely characterized by
  \begin{align}\label{eq:prexi-opt-cond}
    \inner{\pol}{\prexipol_n} = \pol(\pxplambda)\|\prexipol_n\|^2
  \end{align}
  for all $\pol \in \pols{n}$.
\end{lemma}
\begin{proof}
  Problem~\eqref{eq:def-pre-xi-pol} is strongly convex for the given $n$, hence
  its unique solution is characterized by the first-order optimality conditions.
  Explicitly, $\inner{q}{\prexipol_n} = 0$ for all $q \in
  \pols{n}$ such that $q(\pxplambda) = 0$.
  Notice that a polynomial $\pol \in \pols{n}$ can be written $\pol =
  \pol(\pxplambda) \prexipol_n + q$ where $q \in \pols{n}$ satisfies
  $q(\pxplambda) = 0$.
  Plugging this in the optimality condition gives $\inner{\pol}{\prexipol_n} =
  \pol(\pxplambda)\|\prexipol_n\|^2$.
\end{proof}

\begin{lemma}\label{lemma:prexi-short-recurrence}
  The identity
  \begin{align}\label{eq:prexi-short-recurrence}
    \frac{\prexipol_n}{\|\prexipol_n\|^2} = \frac{\prexipol_{n - 1}}{\|\prexipol_{n - 1}\|^2} + \rhopol_n(\pxplambda)\frac{\rhopol_n}{\|\rhopol_n\|^2}
  \end{align}
  holds for all $n \in \{1, \dots, \grade - 1\}$.
\end{lemma}
\begin{proof}
  Decompose the polynomial $\prexipol_n$ in the orthogonal basis $\{\rhopol_0,
  \dots, \rhopol_n\}$ (Lemma~\ref{lemma:rho-basic-props}) to obtain
  \begin{align*}
    \prexipol_n = \sum_{i = 0}^n \frac{\inner{\prexipol_n}{\rhopol_i}}{\|\rhopol_i\|^2} \rhopol_i = \|\prexipol_n\|^2 \sum_{i = 0}^n \frac{\rhopol_i(\pxplambda)}{\|\rhopol_i\|^2}\rhopol_i,
  \end{align*}
  where we used the optimality condition~\eqref{eq:prexi-opt-cond} for the
  second equality.
  Divide by $\|\prexipol_n\|^2$ and subtract two consecutive equalities to
  obtain the result.
\end{proof}

The following is a variation of classical arguments at the basis of the
Cristoffel--Darboux formula.
This is an important result to control the roots of the polynomials
$\sequence{\prexipol_i}$.

\begin{lemma}\label{lemma:roots-interlace}
  Let $n \in \{0, \dots, \grade - 1\}$.
  If $\pxplambda < \lambda_\sdim$ then the roots of $\prexipol_n$ and
  $\rhopol_{n + 1}$ interlace (strictly).
  In particular, $\prexipol_n$ has degree $n$, and its roots are real, distinct,
  and in the interval $\interval{\lambda_\sdim}{\lambda_1}$.
\end{lemma}
\begin{proof}
  It is sufficient to show that $\rhopol_{n + 1}'\prexipol_n - \rhopol_{n +
    1}\prexipol_n'$ has constant (strict) sign to obtain that the roots
  interlace.
  Indeed, assume it is the case and consider two consecutive zeros $z_1 < z_2$
  of $\rhopol_{n + 1}$.
  Rolle's theorem implies that $\rhopol_{n + 1}'(z_1)$ and $\rhopol_{n + 1}'(z_2)$
  have opposite signs (recall that the zeros of $\rhopol_{n + 1}$ are simple by
  Lemma~\ref{lemma:rho-basic-props}).
  Since $\rhopol_{n + 1}'\prexipol_n - \rhopol_{n + 1}\prexipol_n'$ has constant
  sign, it follows that $\prexipol_n(z_1)\prexipol_n(z_2) < 0$.
  This implies that $\prexipol_n$ has a root in the interval
  $\interval[open]{z_1}{z_2}$.

  We now prove by induction that the polynomial $\rhopol_{n + 1}'\prexipol_n -
  \rhopol_{n + 1}\prexipol_n'$ indeed has constant sign equal to $(-1)^n$ for
  all $n \in \{0, \dots, \grade - 1\} $.
  The claim is clear when $n = 0$ because $\rhopol_1'\prexipol_0 -
  \rhopol_1\prexipol_0' = 1$.
  Let $n \in \{1, \dots, \grade - 1\}$ and suppose that the claim is true at
  order $n - 1$.
  Define the polynomial $\pol \colon x \mapsto (x - \pxplambda)\prexipol_n(x)$
  and decompose it in the orthogonal basis $\{\rhopol_0, \dots, \rhopol_{n +
    1}\}$ with Lemma~\ref{lemma:ortho-decomp}.
  For all $i < n$ observe that $\inner{\pol}{\rhopol_i} = \inner{\prexipol_n}{(x
    - \pxplambda)\rhopol_i} = 0$ using the characterization of $\prexipol_n$
  in~\eqref{eq:prexi-opt-cond}.
  We obtain
  \begin{align}\label{eq:pre-darboux-decomposition}
    \pol = \alpha_n \rhopol_{n + 1} + \beta_n \rhopol_n
  \end{align}
  where $\alpha_n = \polcoef{n + 1}{\pol}$ and $\beta_n =
  \inner{\pol}{\rhopol_n}/\|\rhopol_n\|^2$.
  From~\eqref{eq:pre-darboux-decomposition} we deduce that
  \begin{equation}
    \begin{aligned}\label{eq:darboux-decomposition}
      (x - \pxplambda)\prexipol_n(x)\prexipol_n(y) = \alpha_n\rhopol_{n + 1}(x)\prexipol_n(y) + \beta_n\rhopol_n(x)\prexipol_n(y)\\
      (y - \pxplambda)\prexipol_n(y)\prexipol_n(x) = \alpha_n\rhopol_{n + 1}(y)\prexipol_n(x) + \beta_n\rhopol_n(y)\prexipol_n(x)
    \end{aligned}
  \end{equation}
  for all $x, y$.
  From the identity~\eqref{eq:prexi-short-recurrence} and the definition of $p$
  we find that $\alpha_n = \rhopol_n(\pxplambda)\|\prexipol_n\|^2/\|\rhopol_n\|^2$,
  which is a non-zero number (see Lemma~\ref{lemma:rho-basic-props}).
  Subtracting the two identities in~\eqref{eq:darboux-decomposition} and
  dividing by $\alpha_n(x - y)$ yields
  \begin{align*}
    \frac{\rhopol_{n + 1}(x)\prexipol_n(y) - \rhopol_{n + 1}(y)\prexipol_n(x)}{x - y} &= \frac{1}{\alpha_n}\bigg(\prexipol_n(x)\prexipol_n(y) - \beta_n\frac{\rhopol_n(x)\prexipol_n(y) - \rhopol_n(y)\prexipol_n(x)}{x - y}\bigg)\\
                                                                                      &= \frac{1}{\alpha_n}\bigg(\prexipol_n(x)\prexipol_n(y) - \beta_n\frac{\|\prexipol_n\|^2}{\|\prexipol_{n - 1}\|^2}\frac{\rhopol_n(x)\prexipol_{n - 1}(y) - \rhopol_n(y)\prexipol_{n - 1}(x)}{x - y}\bigg),
  \end{align*}
  where we used~\eqref{eq:prexi-short-recurrence} for the last equality.
  Taking the limit $y \to x$ in the previous expression gives
  \begin{align}\label{eq:darboux-induction}
    \rhopol_{n + 1}'\prexipol_n - \rhopol_{n + 1}\prexipol_n' = \frac{1}{\alpha_n}\bigg(\prexipol_n^2 - \beta_n \frac{\|\prexipol_n\|^2}{\|\prexipol_{n - 1}\|^2}\big( \rhopol_n'\prexipol_{n - 1} - \rhopol_n\prexipol_{n - 1}' \big)\bigg).
  \end{align}
  We now study the signs of $\alpha_n$ and $\beta_n$.
  Recall that Lemma~\ref{lemma:rho-basic-props} ensures that the degree $n$
  polynomial $\rhopol_n$ is monic with roots in the interval
  $\interval{\lambda_\sdim}{\lambda_1}$.
  We assumed $\pxplambda < \lambda_\sdim$ so this implies $\sign
  \rhopol_n(\pxplambda) = (-1)^n$.
  From the expression $\alpha_n =
  \rhopol_n(\pxplambda)\|\prexipol_n\|^2/\|\rhopol_n\|^2$ we deduce $\sign
  \alpha_n = \sign \rhopol_n(\pxplambda) = (-1)^n$.
  Now consider $\beta_n = \inner{\pol}{\rhopol_n}/\|\rhopol_n\|^2$.
  We substitute $\rhopol_n / \|\rhopol_n\|^2$
  using~\eqref{eq:prexi-short-recurrence} and notice
  that~\eqref{eq:prexi-opt-cond} implies $\inner{(x - \pxplambda)
    \prexipol_n}{\prexipol_{n - 1}} = \inner{\prexipol_n}{(x - \pxplambda)
    \prexipol_{n - 1}} = 0$.
  This yields
  \begin{align*}
    \beta_n = \frac{1}{\rhopol_n(\pxplambda)\|\prexipol_n\|^2}\inner{\pol}{\prexipol_n} = \frac{1}{\rhopol_n(\pxplambda)\|\prexipol_n\|^2}\sum_{i = 1}^\sdim \weight_i^2 (\lambda_i - \pxplambda) \prexipol_n(\lambda_i)^2,
  \end{align*}
  where the second equality comes from the definition of the inner product
  $\inner{\cdot}{\cdot}$.
  At least one of the terms in the sum is (strictly) positive because
  $\prexipol_n$ has degree at most $n < \grade$ (see Remark~\ref{remark:grade}).
  It follows that $\sign \beta_n = \sign \rhopol_n(\pxplambda) = (-1)^n$.
  We deduce from~\eqref{eq:darboux-induction} that the sign of $\rhopol_{n +
    1}'\prexipol_n - \rhopol_{n + 1}\prexipol_n'$ is constant and equal to
  $(-1)^n$.
\end{proof}

In fact, the conclusion of Lemma~\ref{lemma:roots-interlace} remains true for
$\pxplambda > \lambda_1$ but we will not need this.

\section{Understanding \cg{} with eigenvalues around zero}\label{sec:indefinite-cg}

In this section we build tools to analyze \cg{}
when the input matrix has a small (possibly) indefinite part with eigenvalues
close to zero.
To do this, we relate the behavior of \cg{} on two instances.
Specifically, we let $\mdim \geq \sdim$ be two integers.
Given $\lambda_{1\vphantom{\mdim}}, \dots, \lambda_\mdim \in \reals$ and $\weight_{1\vphantom{\mdim}}, \dots, \weight_\mdim \in \reals$, we define
\begin{align*}
  \mat &= \diag(\lambda_1, \dots, \lambda_\sdim), && &\weight &= (\weight_1, \dots, \weight_\sdim)^\top,\\
  \tmat &= \diag(\lambda_{1\vphantom{\mdim}}, \dots, \lambda_{\sdim\vphantom{\mdim}}, \lambda_{\sdim + 1\vphantom{\mdim}}, \dots, \lambda_\mdim), && \text{and} &\tweight &= (\weight_{1\vphantom{\mdim}}, \dots, \weight_{\sdim\vphantom{\mdim}}, \weight_{\sdim + 1\vphantom{\mdim}}, \dots, \weight_\mdim)^\top.
\end{align*}
As a result, $\mat$ is a submatrix of $\tmat$ and $\weight$ is the corresponding subvector of $\tweight$.
We think of $(\mat, \weight)$ as a well-behaved reference problem (later we will
assume $\mat \succ \zeros$) and we want to understand \cg{} on $(\tmat,
\tweight)$, where $\tmat$ may have an indefinite part.
To do this, we relate the \lanczos{} polynomials of the two problems in
Section~\ref{subsec:link}.
This relation relies on the special polynomials introduced in
Section~\ref{subsec:prexipols} and involves a specific vector $\coef$ of size
$\mdim - \sdim$.
Since the \lanczos{} polynomials determine the iterates of \cg{}
(Section~\ref{subsec:cg-background}), we can deduce a relation between the
iterates of the two problems.
We find that they are close when the 1-norm of $\coef$ is small.
Finally, in Section~\ref{subsec:effective-regime} we consider a certain regime
and secure that this norm is indeed small at some iteration.
Figures~\ref{fig:vn-rn-norms} and~\ref{fig:pols} illustrate the empirical
behavior of \cg{} on two such instances.

\begin{figure}[t]
  \centering
  \includegraphics[width=0.95\textwidth]{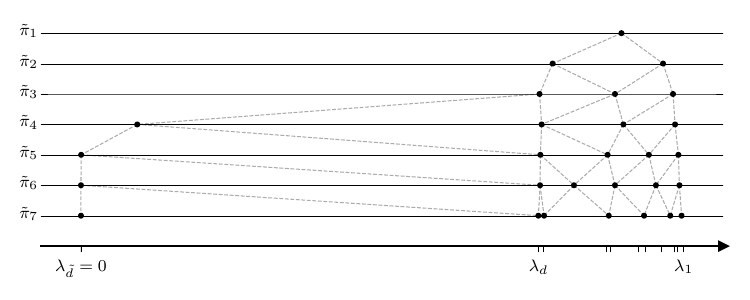}
  \caption{
    \lanczos{} polynomials $\trhopol_1, \ldots, \trhopol_7$ associated to the
    iterates $1$ to $7$ for the problem in Figure~\ref{fig:vn-rn-norms}.
    The horizontal axis shows the eigenvalues $\lambda_1 \geq \dots \geq
    \lambda_\sdim > \lambda_\mdim = 0$ of $\tmat$.
    For each iteration $n$ we plot the roots of the Lanczos polynomial
    $\trhopol_n$.
    The lines linking the roots emphasize the interlacement.
    Most of the weight of $\tweight$ lies in the interval
    $\interval{\lambda_\sdim}{\lambda_1}$.
    This is why the roots are located in this interval during the first
    iterations.
    After the fourth iteration the minimal root rapidly approaches zero.
    This causes the explosion of the iterates and residuals (see
    Figure~\ref{fig:vn-rn-norms}), but \tcg{} can stop earlier.
  }\label{fig:pols}
\end{figure}

\subsection{Tracking CG iterates on two related problems}\label{subsec:link}

In this section, we seek to understand the iterations of \cg{} with inputs
$(\tmat, \tweight)$ as they relate to those with inputs $(\mat, \weight)$.
We let $\grade$ denote the grade of $(\mat, \weight)$.
Note that the grade of $(\tmat, \tweight)$ is at least $\grade$.

\paragraph{A link between the \lanczos{} polynomials.}

We define two semi-inner products $\inner{\cdot}{\cdot}$ and
$\inner{\cdot}{\cdot}_\sim$ on $\pols{\grade}$ as
\begin{align}\label{eq:both-semi-inner-product}
  \inner{\pol}{\qol} = \sum_{i = 1}^\sdim \weight_i^2 \pol(\lambda_i)\qol(\lambda_i) && \text{and} && \inner{\pol}{\qol}_\sim = \sum_{i = 1}^\mdim \weight_i^2 \pol(\lambda_i)\qol(\lambda_i) = \inner{\pol}{\qol} + \sum_{i = \sdim + 1}^\mdim \weight_i^2 \pol(\lambda_i)\qol(\lambda_i).
\end{align}
They induce the two semi-norms $\|\cdot\|$ and $\|\cdot\|_\sim$.
We can then define the \lanczos{} polynomials of degree $n \leq \grade$ associated to each problem as
\begin{align}\label{eq:both-def-rho-pol}
  \rhopol _n = \argmin_{\rhopol \in \pols{n}} \|\rhopol\|^2 && \text{and} && \trhopol_n = \argmin_{\rhopol \in \pols{n}} \|\rhopol\|_\sim^2, && \text{both subject to} && \polcoef{n}{\rhopol} = 1.
\end{align}
To understand the iterations of \cg{} on $(\tmat, \tweight)$, we set out to
relate $\rhopol_n$ and $\trhopol_n$ using the polynomials
$\sequence{\prexipol_i}$ defined in Section~\ref{subsec:prexipols}.
For this, we rely on $\mdim - \sdim$ instances of them.
\begin{definition}\label{def:prexipol-j}
  For $\sdim < j \leq \mdim$ and $n < \grade$, let $\prexipol_n^j$ be the
  solution to~\eqref{eq:def-pre-xi-pol} with data $(\mat, \weight)$ and $\pxplambda = \lambda_j$.
\end{definition}
The conclusion of the following lemma holds without any assumption on the
eigenvalues $\lambda_1, \dots, \lambda_{\mdim}$.
In particular, we do \emph{not} assume that they are positive.
For convenience, $B, C, D \in \reals^{(\mdim - \sdim) \times (\mdim - \sdim)}$
and $w, \coef \in \reals^{\mdim - \sdim}$ below are indexed from $\sdim + 1$ to
$\mdim$.
All of these are defined at a specific iteration $n$.

\begin{lemma}\label{lemma:identity-rhos}
  Let $n \in \{1, \dots, \grade\}$ and define the matrices $B, C, D \in
  \reals^{(\mdim - \sdim) \times (\mdim - \sdim)}$ as
  \begin{align}\label{eq:matrices-for-rhos-identity}
    B = \diag(\weight_{\sdim + 1}, \dots, \weight_\mdim), && C_{ij} = \prexipol_{n - 1}^{j}(\lambda_{i}), && D = \diag\!\Big(\|\prexipol_{n - 1}^{\sdim + 1}\|^2, \dots, \|\prexipol_{n - 1}^{\mdim}\|^2\Big)
  \end{align}
  for $i, j \in \{\sdim + 1, \dots, \mdim\}$.
  Also define the vector $w = \big(\rhopol_n(\lambda_{\sdim + 1}), \dots,
  \rhopol_n(\lambda_\mdim)\big)^\top$.
  Then the matrix $D + B^2C$ is invertible and
  \begin{align}\label{eq:rho-identity}
    \trhopol_n = \rhopol_n - \sum_{j = \sdim + 1}^\mdim \coef_j \prexipol_{n - 1}^j,
  \end{align}
  where $\coef = (D + B^2C)^{-1}B^2w$.
\end{lemma}
\begin{proof}
  Given $j \in \{\sdim + 1, \dots, \mdim\}$, define the polynomial
  \begin{align} \label{eq:pjinproof}
    \pol_j = \sum_{i = 0}^{n - 1}\frac{\trhopol_i(\lambda_j)}{\|\trhopol_i\|_\sim^2}\trhopol_i.
  \end{align}
  Decompose the polynomial $\rhopol_n$ in the orthogonal basis $\{\trhopol_0,
  \dots, \trhopol_n\}$ using Lemma~\ref{lemma:ortho-decomp} (instantiated with
  the inner product $\inner{\cdot}{\cdot}_\sim$ rather than
  $\inner{\cdot}{\cdot}$).
  This yields
  \begin{align*}
    \rhopol_n = \trhopol_n + \sum_{i = 0}^{n - 1} \frac{\inner{\rhopol_n}{\trhopol_i}_\sim}{\|\trhopol_i\|_\sim^2}\trhopol_i = \trhopol_n + \sum_{i = 0}^{n - 1} \frac{\trhopol_i}{\|\trhopol_i\|_\sim^2}\bigg(\inner{\rhopol_n}{\trhopol_i} + \sum_{j = \sdim + 1}^\mdim \weight_j^2 \rhopol_n(\lambda_j) \trhopol_i(\lambda_j)\bigg),
  \end{align*}
  where the second equality comes from the
  definition~\eqref{eq:both-semi-inner-product} of the inner product
  $\inner{\cdot}{\cdot}_\sim$.
  Observe that $\inner{\rhopol_n}{\trhopol_i} = 0$ for $i \in \{0, \dots, n -
  1\}$ because of the characterization of $\rhopol_n$
  in Lemma~\ref{lemma:rho-opt-cond}.
  This leads to
  \begin{align}\label{eq:intermediate-id-1}
    \rhopol_n = \trhopol_n + \sum_{j = \sdim + 1}^\mdim \weight_j^2\rhopol_n(\lambda_j)\pol_j
  \end{align}
  with the polynomials $\pol_j$ from~\eqref{eq:pjinproof}.
  Now let $k \in \{\sdim + 1, \dots, \mdim\}$ and decompose the polynomial
  $\prexipol_{n - 1}^k$ in the orthogonal basis $\{\trhopol_0, \dots,
  \trhopol_{n - 1}\}$ to obtain
  \begin{align*}
    \prexipol_{n - 1}^k = \sum_{i = 0}^{n - 1} \frac{\inner{\prexipol_{n - 1}^k}{\trhopol_i}_\sim}{\|\trhopol_i\|_\sim^2}\trhopol_i = \sum_{i = 0}^{n - 1} \frac{\trhopol_i}{\|\trhopol_i\|_\sim^2} \bigg( \inner{\prexipol_{n - 1}^k}{\trhopol_i} + \sum_{j = \sdim + 1}^\mdim \weight_j^2 \prexipol_{n - 1}^k(\lambda_j) \trhopol_i(\lambda_j) \bigg).
  \end{align*}
  Combining this with the characterization of $\prexipol_{n - 1}^k$
  in Lemma~\ref{lemma:prexi-opt-cond} gives %
  \begin{align}\label{eq:intermediate-id-2}
    \prexipol_{n - 1}^k = \|\prexipol_{n - 1}^k\|^2 \pol_k + \sum_{j = \sdim + 1}^\mdim \weight_j^2 \prexipol_{n - 1}^k(\lambda_j) \pol_j.
  \end{align}
  Collecting the identities~\eqref{eq:intermediate-id-1}
  and~\eqref{eq:intermediate-id-2} in matrix form yields
  \begin{align}\label{eq:collected-identities}
    \trhopol_n = \rhopol_n - \begin{bmatrix}
                               \pol_{\sdim + 1} & \dots & \pol_\mdim
                             \end{bmatrix} B^2 w
    && \text{and} &&
                     \begin{bmatrix}
                       \prexipol_{n - 1}^{\sdim + 1} & \dots & \prexipol_{n - 1}^\mdim
                     \end{bmatrix} = \begin{bmatrix}
                                       \pol_{\sdim + 1} & \dots & \pol_\mdim
                                     \end{bmatrix}(D + B^2C),
  \end{align}
  where $B, C, D$ and $w$ are defined around~\eqref{eq:matrices-for-rhos-identity}.

  We now show that $M = D + B^2C$ is invertible.
  First, notice that for all $j \in \{\sdim + 1, \dots, \mdim\}$, the quantity
  $\|\prexipol_{n - 1}^j\|^2$ is positive because $n - 1 < \grade$.
  It follows that $D$ is positive definite.
  Let $G$ be the symmetric $(\mdim - \sdim) \times (\mdim - \sdim)$ Gram matrix
defined as $G_{ij} = \inner{\prexipol_{n - 1}^i}{\prexipol_{n - 1}^j}$.
  The characterizations of the polynomials $\prexipol_{n - 1}^i$ and
  $\prexipol_{n - 1}^j$ in~\eqref{eq:prexi-opt-cond} give $G_{ij} = \prexipol_{n
  - 1}^j(\lambda_i) \|\prexipol_{n - 1}^i\|^2$, or equivalently, $C = D^{-1}G$.
  According to the Woodbury formula, $M = D + B \cdot (BC)$ is invertible if the
  matrix $I + B C D^{-1} B = I + B D^{-1} G D^{-1} B$ is invertible.
  It is the case because the second term is positive semidefinite.
  Finally, substituting the polynomials $\pol_{\sdim + 1}, \dots, \pol_\mdim$ in
  the first equality of~\eqref{eq:collected-identities} by those of the second
  equality gives the result.
\end{proof}

\paragraph{Controlling the \cg{} polynomials.}

We now leverage the identity~\eqref{eq:rho-identity} in Lemma~\ref{lemma:identity-rhos}
to obtain relationships between the iterates of \cg{}
for the inputs $(\tmat, \tweight)$ and $(\mat, \weight)$.
To this end, we assume from now on that
\begin{align*}
  \lambda_1 \geq \dots \geq \lambda_\mdim, && \lambda_\sdim > 0 && \text{and} && \lambda_\sdim > \lambda_{\sdim + 1}.
\end{align*}
(We order the eigenvalues, and suppose that $\mat \succ \zeros$ and that there is
an eigenvalue gap.)
As a result, the roots of $\rhopol_n$ and $\prexipol_{n - 1}^j$ are all positive
for $n \in \{1, \dots, \grade\}$ (Lemma~\ref{lemma:rho-basic-props} and
Lemma~\ref{lemma:roots-interlace}).
This ensures in particular that the polynomials $\qpol_n = \rhopol_n /
\rhopol_n(0)$ are well defined (recall from Section~\ref{par:connections-cg}
that these polynomials determine the iterates of \cg{}).
In order to understand the $n$th iterate of \cg{} on $(\tmat, \tweight)$, we
proceed in two steps: \emph{(i)} we ensure that the iterate is indeed well
defined, meaning that the roots of $\trhopol_n$ are positive, and \emph{(ii)} we
relate the polynomial $\tqpol_n = \trhopol_n / \trhopol_n(0)$ to $\qpol_n$.

To do this, we first rewrite the identity~\eqref{eq:rho-identity} in
a way that is more comfortable to manipulate the aforementioned rescaled
polynomials $\qpol_n$ and $\tqpol_n$.
We introduce
\begin{align}\label{eq:def-xi-pols}
  \xipol_n^j = \frac{\prexipol_n^j}{\prexipol_n^j(0)} \quad\text{for all $n \in \{0, \dots, \grade - 1\}$ and $j \in \{\sdim + 1, \dots, \mdim\}$},
\end{align}
where the polynomials $\prexipol_n^j$ are as in Definition~\ref{def:prexipol-j}.
Remember that Table~\ref{tab:pols} summarizes the relationships between all the
polynomials defined above.
(Again, for convenience, $B, C, D \in \reals^{(\mdim - \sdim) \times (\mdim -
  \sdim)}$ and $w, \coef \in \reals^{\mdim - \sdim}$ below are indexed from
$\sdim + 1$ to $\mdim$. They also depend on the iteration $n$.)

\begin{lemma}\label{lemma:identity-rhos-2}
  Suppose that $\lambda_\sdim > 0$ and $\lambda_\sdim > \lambda_{\sdim + 1}$.
  Let $n \in \{1, \dots, \grade\}$ and define the matrices
  \begin{align*}
    B = \diag(\weight_{\sdim + 1}, \dots, \weight_\mdim), && C_{ij} = \xipol_{n - 1}^{j}(\lambda_{i}), && D = \diag\!\bigg(\frac{\|\xipol_{n - 1}^{\sdim + 1}\|^2}{\xipol_{n - 1}^{\sdim + 1}(\lambda_{\sdim + 1})}, \dots, \frac{\|\xipol_{n - 1}^{\mdim}\|^2}{\xipol_{n - 1}^\mdim(\lambda_\mdim)}\bigg),
  \end{align*}
  for $i, j \in \{\sdim + 1, \dots, \mdim\}$, and where $\xipol_{n - 1}^{d +
    j}$ are as in~\eqref{eq:def-xi-pols}.
  Also define $w = \big(\qpol_n(\lambda_{\sdim + 1}), \dots,
  \qpol_n(\lambda_\mdim)\big)^\top$.
  With $\coef = (D + B^2C)^{-1}B^2w$, we have
  \begin{align}\label{eq:rho-identity-2}
    \trhopol_n = \rhopol_n - \rhopol_n(0) \sum_{j = \sdim + 1}^\mdim \coef_j \xipol_{n - 1}^j.
  \end{align}
\end{lemma}
\begin{proof}
  This is just a rescaled variant of~\eqref{eq:rho-identity}, made possible by
  the assumptions on $\lambda_\sdim$ and $\lambda_{\sdim+1}$.
  Notice that $B$ is unchanged but $C, D, w$ and $\coef$ are all rescaled.
  Let $\hat C, \hat D, \hat w$ and $\hat \coef$ denote their counterparts in
  Lemma~\ref{lemma:identity-rhos}.
  Then $\hat C = C S$ where $S = \diag\!\big(\prexipol_{n - 1}^{\sdim + 1}(0), \dots,
  \prexipol_{n - 1}^\mdim(0)\big)$ is a scaling matrix.
  We leverage now that $\xipol_{n - 1}^j(\lambda_j) = 1/\prexipol_{n - 1}^j(0)$
  and $\|\xipol_{n - 1}^j\|^2/\xipol_{n - 1}^j(\lambda_j) = \|\prexipol_{n -
    1}^j\|^2/\prexipol_{n - 1}^j(0)$ for all $j \in \{\sdim + 1, \dots,
  \mdim\}$.
  This yields $\hat D = D S$.
  Finally, $\hat w = \rhopol_n(0) w$, and combining the above relations gives
  $\hat \coef = \rhopol_n(0) S^{-1} \coef$.
  Plugging it in~\eqref{eq:rho-identity} leads to
  identity~\eqref{eq:rho-identity-2}.
\end{proof}

For $n \leq \grade$, the $n$th iterate of \cg{} on $(\tmat, \tweight)$ is well
defined if (and only if) the roots of $\trhopol_n$ are all positive (see
Section~\ref{subsec:cg-background}).
Consequently, we now study the roots of $\trhopol_n$.
We prove that they are close to the roots of $\rhopol_n$ when the entries of
$\coef$ (defined in Lemma~\ref{lemma:identity-rhos-2}) are small.
For this, given a root $z$ of $\trhopol_n$, we exhibit a root $\gamma$ of
$\rhopol_n$ such that the ratio $z/\gamma$ is close to 1.
The key is that the roots of the polynomials $\xipol_{n - 1}$ and $\rhopol_n$
interlace (by~\eqref{eq:def-xi-pols} and Lemma~\ref{lemma:roots-interlace}).
This will yield a lower-bound on the roots of $\trhopol_n$.
\begin{figure}[t]
  \centering
  \includegraphics[width=1\textwidth]{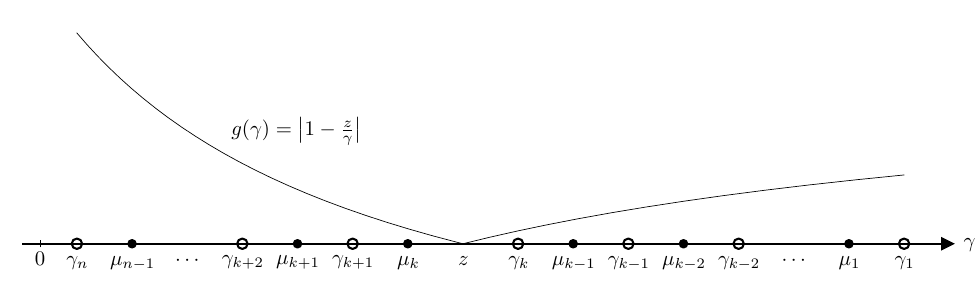}
  \caption{This figure supports the proof of Lemma~\ref{lemma:bound-roots}.
      The point $z$ may be anywhere between $\gamma_{k + 1}$ and $\gamma_{k - 1}$.
      Here, it is represented between $\mu_k$ and $\gamma_k$,
      but the proof considers all scenarios.
  }\label{fig:roots}
\end{figure}

\begin{lemma}\label{lemma:bound-roots}
  Suppose that $\lambda_\sdim > 0$ and $\lambda_\sdim > \lambda_{\sdim + 1}$.
  Let $\gamma_1 > \dots > \gamma_n$ be the roots of $\rhopol_n$ with $n \leq
  \grade$.
  If $z$ is a root of $\trhopol_n$ then
  \begin{align*}
    \min_{i = 1, \dots, n} \Big|1 - \frac{z}{\gamma_i}\Big| \leq \|\coef\|_1,
  \end{align*}
  where $\coef$ (which depends on $n$) is as in
  Lemma~\ref{lemma:identity-rhos-2}, and $\|\coef\|_1 = |\coef_{\sdim+1}| +
  \cdots + |\coef_{\mdim}|$ is its $1$-norm.
\end{lemma}
\begin{proof}
  If $z \in \{\gamma_1, \dots, \gamma_n\}$ then the result is clear.
  Otherwise, define the function $g \colon \gamma \mapsto |1 -
  \frac{z}{\gamma}|$ and let $k \in \argmin_{i = 1, \dots, n} g(\gamma_i)$.
  Let $\mu_1^j > \dots > \mu_{n - 1}^j$ denote the roots of $\xipol_{n - 1}^j$.
  Identity~\eqref{eq:rho-identity-2} yields
  \begin{align*}
    0 = \trhopol_n(z) = \prod_{i = 1}^n (z - \gamma_i) - \rhopol_n(0)\sum_{j = \sdim + 1}^\mdim \coef_j \prod_{i = 1}^{n - 1} \bigg(1 - \frac{z}{\mu_i^j} \bigg).
  \end{align*}
  Divide by $-\gamma_k \prod_{i \neq k} (z - \gamma_i)$ and apply the
  triangle inequality to obtain
  \begin{align*}
    g(\gamma_k) \leq \sum_{j = \sdim + 1}^\mdim |\coef_j| h_j && \text{where} && h_j = \frac{\prod_{i = 1}^{n - 1} g(\mu_i^j)}{\prod_{i \neq k} g(\gamma_i)}.
  \end{align*}
  We now let $j \in \{\sdim + 1, \dots, \mdim\}$ and show that $h_j \leq 1$.
  This is clear when $z = 0$ so we consider $z \neq 0$.
  Under our assumptions, the roots of $\rhopol_n$ and $\xipol_{n - 1}^j$
  interlace and are positive (as shown in Lemmas~\ref{lemma:rho-basic-props}
  and~\ref{lemma:roots-interlace}).
  First assume that $z < 0$.
  In this case the function $g$ is decreasing on $\interval[open]{0}{\infty}$
  and it follows that $k = 1$.
  Together with the interlacement of the roots, this implies that $h_j \leq 1$.
  Now assume that $z > 0$.
  The function $g$ is decreasing on $\interval[open left]{0}{z}$ and increasing
  on $\interval[open right]{z}{+\infty}$.
  This implies that $\gamma_{k + 1} < z < \gamma_{k - 1}$ (we set $\gamma_0 =
  +\infty$ and $\gamma_{n + 1} = 0$).
  We obtain the following pattern for the roots:
  \begin{equation}\label{eq:gammas-mus-interlace}
    \begin{aligned}
      0 <\gamma_n^{\phantom{j}} < \mu_{n - 1}^j < \dots < \gamma_{k + 2}^{\phantom{j}} < \mu_{k + 1}^j <& \gamma_{k + 1}^{\phantom{j}} < z_{\vphantom{k}}^{\vphantom{j}} < \gamma_{k - 1}^{\phantom{j}} < \mu_{k - 2}^j < \gamma_{k - 2}^{\phantom{j}} < \dots < \mu_1^j < \gamma_1^{\phantom{j}}\\
      \text{and} \qquad\qquad &\gamma_{k + 1}^{\phantom{j}} < \mu_k^j < \gamma_k^{\phantom{j}} < \mu_{k - 1}^j < \gamma_{k - 1}^{\phantom{j}}.
    \end{aligned}
  \end{equation}
  Figure~\ref{fig:roots} is an illustration of this interlacement and of $g$.
  Observe from~\eqref{eq:gammas-mus-interlace} that $g(\gamma_i) \geq g(\mu_{i -
    1}^j)$ for all $i \in \{n, \dots, k + 2\}$ because $g$ is decreasing on
  $\interval[open left]{0}{z}$.
  Likewise, we have $g(\mu_i^j) \leq g(\gamma_i)$ for all $i \in \{k - 2, \dots,
  1\}$ because $g$ is increasing on $\interval[open right]{z}{+\infty}$.
  From the monotonicity pattern of $g$ we deduce that $g(\mu_k^j) \leq
  \max(g(\gamma_{k + 1}), g(\gamma_k))$.
  To see this observe that it holds whether $z \in \interval[open
  left]{\gamma_{k + 1}}{\mu_k^j}$, $z \in \interval{\mu_k^j}{\gamma_k}$ or $z
  \geq \gamma_k$.
  The minimality of $\gamma_k$ then implies that $g(\mu_k^j) \leq g(\gamma_{k +
    1})$.
  Likewise, $g(\mu_{k - 1}^j) \leq \max(g(\gamma_k), g(\gamma_{k - 1})) =
  g(\gamma_{k - 1})$.
  Combining all the previous inequalities yields $h_j \leq 1$.
\end{proof}

If we additionally assume $\|\coef\|_1 < 1$, we have a useful corollary.

\begin{corollary}\label{cor:roots-lower-bound}
  Suppose that $\lambda_\sdim > 0$ and $\lambda_\sdim > \lambda_{\sdim + 1}$.
  If $\|\coef\|_1 < 1$ for some $n \leq \grade$, then the roots of $\trhopol_n$
  are lower-bounded by $\big(1 - \|\coef\|_1\big) \gamma_n$.
  In particular, they are positive.
\end{corollary}
\begin{proof}
  Let $z$ be a root of $\trhopol_n$ and $k \in \argmin_{i = 1, \dots, n} \big|1 -
  \frac{z}{\gamma_i}\big|$.
  Lemma~\ref{lemma:bound-roots} gives that $\gamma_k - z \leq \gamma_k
  \|\coef\|_1$.
  Rearranging the terms gives the result, using that $\gamma_k \geq \gamma_n >
  0$.
\end{proof}

We now combine all the results above to obtain \emph{(i)} a sufficient condition
for the $n$th iteration of \cg{} on $(\tmat, \tweight)$ to be well defined, and,
when it is the case, \emph{(ii)} a relationship between $\tqpol_n$ and
$\qpol_n$.
Here we let $\ones$ denote the vector of all ones.

\begin{theorem}\label{th:identity-qs}
  Suppose that $\lambda_\sdim > 0$ and $\lambda_\sdim > \lambda_{\sdim + 1}$.
  Let $n \in \{1, \dots, \grade\}$ and $\coef$ as in
  Lemma~\ref{lemma:identity-rhos-2}.
  If $\|\coef\|_1 < 1$ then the $n$th iteration of \cg{} on $(\tmat, \tweight)$
  is well defined, and the associated polynomial $\tqpol_n =
  \trhopol_n/\trhopol_n(0)$ satisfies
  \begin{align}\label{eq:identity-qs}
    \tqpol_n = \frac{1}{1 - s}\bigg( \qpol_n - \sum_{j = \sdim + 1}^\mdim \coef_j \xipol_{n - 1}^j \bigg), && \text{where} && s = \ones^\top \coef.
  \end{align}
\end{theorem}
\begin{proof}
  Corollary~\ref{cor:roots-lower-bound} gives that the roots of $\trhopol_n$ are
  positive.
  Lemma~\ref{lemma:roots-and-eigenvalues} then implies that the matrix
  $\restr{\tmat}{\tilde \krylovspace_n}{\abitbig}$ is positive definite, where
  $\tilde \krylovspace_n$ is the $n$th Krylov subspace associated to $(\tmat,
  \tweight)$.
  Thus, the $n$th iteration of \cg{} on $(\tmat, \tweight)$ is well defined.
  The definition of $\tqpol_n$ and the identity~\eqref{eq:rho-identity-2} provide
  \begin{align*}
    \tqpol_n = \frac{\trhopol_n}{\trhopol_n(0)} = \frac{\rhopol_n - \rhopol_n(0) \sum_{j = \sdim + 1}^\mdim \coef_j \xipol_{n - 1}^j}{\rhopol_n(0) - \rhopol_n(0) \sum_{j = \sdim + 1}^\mdim \coef_j} = \frac{\qpol_n - \sum_{j = \sdim + 1}^\mdim \coef_j \xipol_{n - 1}^j}{1 - \sum_{j = \sdim + 1}^\mdim \coef_j},
  \end{align*}
  which is the intended equality.
\end{proof}

\paragraph{From the polynomials to the iterates.}
Theorem~\ref{th:identity-qs}
explicitly links the polynomials $\qpol_n$ and $\tqpol_n$,
associated to the problems $(\mat, \weight)$ and $(\tmat, \tweight)$
respectively.
These polynomials determine the iterates of \cg{}
(see Section~\ref{par:connections-cg}).
Indeed, let $\ppol_{n - 1}$ and $\tilde \ppol_{n - 1}$ be the polynomials that
satisfy $\qpol_n(x) = 1 - x \ppol_{n - 1}(x)$ and
$\tqpol_n(x) = 1 - x \tilde \ppol_{n - 1}(x)$ respectively.
Then $\cgvec_n = \ppol_{n - 1}(\mat) \weight$ and $\tilde \cgvec_n = \tilde
\ppol_{n - 1}(\tmat) \tweight$.
We exploit this to relate the iterates on the two problems.
First, we need two helper lemmas to control
$\qpol_n - \xipol_{n - 1}^j$.

\begin{lemma}\label{lemma:decreasing-ppols}
  Suppose that $\lambda_\sdim > 0$ and $\lambda_\sdim > \lambda_{\sdim + 1}$.
  Let $n \in \{1, \dots, \grade\}$ and $j \in \{\sdim + 1, \dots, \mdim\}$.
  Define the polynomials $\pol_0, \ldots, \pol_{n - 1}$ as
  \begin{align*}
    \pol_k(x) = \frac{\qpol_n(x) - \xipol_k^j(x)}{x}.
  \end{align*}
  Then, $\|\pol_{n - 1}\| \leq \cdots \leq \|\pol_1\| \leq \|\pol_0\|
  = \|\ppol_{n - 1}\| = \|\cgvec_n\|$.
\end{lemma}
\begin{proof}
  We fix $j$ and omit for all $k$ the superscript of $\xipol_k^j$ for conciseness.
  We let $k \in \{0, \dots, n - 2\}$ and prove that $\|\pol_{k + 1}\| \leq
  \|\pol_k\|$.
  We can rewrite identity~\eqref{eq:prexi-short-recurrence} as
  \begin{align*}
    \frac{\xipol_{k + 1}(\lambda_j)}{\|\xipol_{k + 1}\|^2} \xipol_{k + 1} = \frac{\xipol_k(\lambda_j)}{\|\xipol_k\|^2} \xipol_k + \frac{\qpol_{k + 1}(\lambda_j)}{\|\qpol_{k + 1}\|^2} \qpol_{k + 1}
  \end{align*}
  using the fact that $\xipol_k(\lambda_j) = 1/\prexipol_k^j(0)$ and $\xipol_{k
    + 1}(\lambda_j) = 1/\prexipol_{k + 1}^j(0)$.
  The roots of $\xipol_k$, $\xipol_{k + 1}$ and $\qpol_{k + 1}$ are all in
  $\interval{\lambda_\sdim}{\lambda_1}$ (Lemmas~\ref{lemma:rho-basic-props}
  and~\ref{lemma:roots-interlace}) and $\xipol_k(0) = \xipol_{k + 1}(0) =
  \qpol_{k + 1}(0) = 1$.
  We deduce that all the coefficients in the identity above are (strictly)
  positive.
  Divide by $\xipol_k(\lambda_j)/\|\xipol_k\|^2$ and leverage again $\xipol_k(0)
  = \xipol_{k + 1}(0) = \qpol_{k + 1}(0) = 1$ to obtain that
  \begin{align}\label{eq:recurrence-alpha}
    \xipol_k = (1 + \alpha) \xipol_{k + 1} - \alpha \qpol_{k + 1}
  \end{align}
  for some number $\alpha \geq 0$.
  It follows that $\pol_k = \pol_{k + 1} + \alpha \qol_{k + 1}$, where we define
  the polynomial
  \begin{align*}
    \qol_{k + 1}(x) = \frac{\qpol_{k + 1}(x) - \xipol_{k + 1}(x)}{x}.
  \end{align*}
  In order to prove that $\|\pol_{k + 1}\|^2 \leq \|\pol_k\|^2$, it is sufficient
  to show that the inner product $\inner{\pol_{k + 1}}{\qol_{k + 1}}$ is
  non-negative.
  We can rewrite
  \begin{align*}
    \inner{\pol_{k + 1}}{\qol_{k + 1}} = \|\qol_{k + 1}\|^2 + \inner{\pol_{k + 1} - \qol_{k + 1}}{\qol_{k + 1}}.
  \end{align*}
  The first term is positive; let us prove the second term is non-negative.
  A Taylor expansion yields
  \begin{align*}
    x \qol_{k + 1}(x) = \qpol_{k + 1}(x) - \xipol_{k + 1}(x) = \big(\qpol_{k + 1}'(0) - \xipol_{k + 1}'(0)\big) x + x^2 h(x)
  \end{align*}
  for all $x$, where $h$ is some polynomial of degree at most $k - 1$.
  It follows that
  \begin{align*}
    \inner{\pol_{k + 1} - \qol_{k + 1}}{\qol_{k + 1}} &= \sum_{i = 1}^\sdim \big(\qpol_n(\lambda_i) - \qpol_{k + 1}(\lambda_i)\big)\bigg(\frac{\qpol_{k + 1}'(0) - \xipol_{k + 1}'(0)}{\lambda_i} + h(\lambda_i)\bigg) \weight_i^2\\
                                                      &= \big(\qpol_{k + 1}'(0) - \xipol_{k + 1}'(0)\big) \sum_{i = 1}^\sdim \frac{\qpol_n(\lambda_i) - \qpol_{k + 1}(\lambda_i)}{\lambda_i} \weight_i^2
  \end{align*}
  because $\qpol_n$ and $\qpol_{k + 1}$ are both orthogonal to $h$ for the inner
  product $\inner{\cdot}{\cdot}$ (Lemma~\ref{lemma:rho-opt-cond}).
  From~\eqref{eq:recurrence-alpha} we find that $\qpol_{k + 1}'(0) - \xipol_{k +
    1}'(0) = (1 + \alpha)^{-1}\big(\qpol_{k + 1}'(0) - \xipol_k'(0)\big)$.
  Let $\gamma_1 > \dots > \gamma_{k + 1}$ and $\mu_1 > \dots > \mu_k$ be the
  roots of $\qpol_{k + 1}$ and $\xipol_k$ respectively.
  Notice that
  \begin{align*}
    \qpol_{k + 1}'(0) = -\sum_{i = 1}^{k + 1} \frac{1}{\gamma_i} < \xipol_k'(0) = -\sum_{i = 1}^k \frac{1}{\mu_i} < 0,
  \end{align*}
  where the inequalities are because the roots of these polynomials interlace
  and are all positive (Lemmas~\ref{lemma:rho-basic-props}
  and~\ref{lemma:roots-interlace}).
  We deduce that $\qpol_{k + 1}'(0) - \xipol_{k + 1}'(0)$ is negative.
  Finally, notice that $\qpol_n$ and $\qpol_{k + 1}$ are orthogonal to the
  polynomials $x \mapsto \big(1 - \qpol_n(x)\big)/x$ and
  $x \mapsto \big(1 - \qpol_{k + 1}(x)\big)/x$ respectively (Lemma~\ref{lemma:rho-opt-cond}).
  This implies that
  \begin{align*}
    \sum_{i = 1}^\sdim \frac{\qpol_n(\lambda_i)}{\lambda_i} \weight_i^2 = \sum_{i = 1}^\sdim \frac{\qpol_n(\lambda_i)^2}{\lambda_i} \weight_i^2 && \text{and} && \sum_{i = 1}^\sdim \frac{\qpol_{k + 1}(\lambda_i)}{\lambda_i} \weight_i^2 = \sum_{i = 1}^\sdim \frac{\qpol_{k + 1}(\lambda_i)^2}{\lambda_i} \weight_i^2.
  \end{align*}
  Combine this with the minimization property of $\qpol_n$ given in
  Lemma~\ref{lemma:cg-pol-min} to conclude that the inner
  product $\inner{\pol_{k + 1} - \qol_{k + 1}}{\qol_{k + 1}}$ is non-negative.
\end{proof}

\begin{lemma}\label{lemma:bound-diff-q-xipol}
  Suppose that $\lambda_\sdim > 0$ and $\lambda_\sdim > \lambda_{\sdim + 1}$.
  Let $n \in \{1, \dots, \grade\}$ and $j \in \{\sdim + 1, \dots, \mdim\}$.
  Then $\big|\qpol_n(x) - \xipol_{n - 1}^j(x)\big| \leq |\qpol_n(x) -
  1|$ for all $x \leq \lambda_\sdim$.
\end{lemma}
\begin{proof}
  Let $\gamma_1 > \dots > \gamma_n$ and $\mu_1 > \dots > \mu_{n - 1}$ be the
  roots of $\qpol_n$ and $\xipol_{n - 1}^j$ respectively.
  Then we can write
  \begin{align*}
    \qpol_n(x) = \prod_{i = 1}^n \bigg(1 - \frac{x}{\gamma_i}\bigg) && \text{and} && \xipol_{n - 1}^j(x) = \prod_{i = 1}^{n - 1} \bigg(1 - \frac{x}{\mu_i}\bigg)
  \end{align*}
  for all $x \in \reals$.
  The roots of $\qpol_n$ and $\xipol_{n - 1}^j$ interlace, and are all at least
  $\lambda_d$ (Lemmas~\ref{lemma:rho-basic-props}
  and~\ref{lemma:roots-interlace}).
  We deduce from these properties and from the expressions above that
  $\qpol_n(x) \geq \xipol_{n - 1}^j(x) \geq 1$ for all $x \leq
  0$ and $1 \geq \xipol_{n - 1}^j(x) \geq \qpol_n(x) \geq 0$ for all
  $x \in \interval{0}{\lambda_\sdim}$.
  The claim holds in both cases.
\end{proof}

We now compare the iterates $\cgvec$ and $\tilde \cgvec$ of \cg{} on $(\mat,
\weight)$ and $(\tmat, \tweight)$ at a given iteration.
The following bounds show that, in a certain regime, the first $\sdim$
components of $\tilde \cgvec$ are close to $\cgvec$, and the last $\mdim -
\sdim$ components are close to zero.
This means that (in this regime) \cg{} on $(\tmat, \tweight)$ essentially ignores the eigenvalues $\lambda_{\sdim + 1\vphantom{\mdim}}, \dots, \lambda_{\mdim}$ and the weights $\weight_{\sdim + 1\vphantom{\mdim}}, \dots, \weight_\mdim$.

\begin{theorem}\label{th:bounds-iterates}
  Suppose that $\lambda_\sdim > 0$ and $\lambda_\sdim > \lambda_{\sdim + 1}$.
  Let $n \in \{1, \dots, \grade\}$ and $\coef$ as in
  Lemma~\ref{lemma:identity-rhos-2}.
  Suppose that $\|\coef\|_1 < 1$ so that the $n$th iteration of \cg{} on
  $(\tmat, \tweight)$ is well defined (Theorem~\ref{th:identity-qs}).
  Let $\cgvec$ and $\tilde \cgvec$ denote the $n$th iterates of \cg{} on $(\mat,
  \weight)$ and $(\tmat, \tweight)$ respectively.
  Then
  \begin{align*}
    \|\tilde \cgvec_{1:\sdim} - \cgvec\| \leq \frac{\|\coef\|_1}{1 - \|\coef\|_1}\|\cgvec\| && \text{and} && |\tilde \cgvec_i| \leq \frac{|\weight_i|}{1 - \|\coef\|_1} \bigg| \frac{\big(1 - \lambda_i/\lambda_\sdim\big)^n - 1}{\lambda_i} \bigg|
  \end{align*}
  for all $i \in \{\sdim + 1, \dots, \mdim\}$, where $\tilde \cgvec_{1:\sdim}$
  extracts the first $\sdim$ entries of $\tilde \cgvec$.
  (The absolute value in the expression on the right evaluates to
  $n/\lambda_\sdim$ in the limit where $\lambda_i = 0$.)
\end{theorem}
\begin{proof}
  Define the polynomials
  \begin{align*}
    \ppol_{n - 1}(x) = \frac{1 - \qpol_n(x)}{x}, && \tilde \ppol_{n - 1}(x) = \frac{1 - \tilde \qpol_n(x)}{x} && \text{and} && \pol_{n - 1}^j(x) = \frac{\qpol_n(x) - \xipol_{n - 1}^j(x)}{x}.
  \end{align*}
  The iterates satisfy $\cgvec = \ppol_{n - 1}(\mat)\weight$ and $\tilde
  \cgvec = \tilde \ppol_{n - 1}(\tmat)\tweight$ as recalled in
  Section~\ref{par:connections-cg}.
  We exploit the link given in Theorem~\ref{th:identity-qs} to relate $\cgvec$
  and $\tilde \cgvec$.
  Rearranging the terms in~\eqref{eq:identity-qs} gives
  \begin{align}\label{eq:qdiff}
    \tqpol_n - \qpol_n = \frac{1}{1 - s} \sum_{j = \sdim + 1}^\mdim \coef_j\big(\qpol_n - \xipol_{n - 1}^j\big) && \text{where} && s = \ones^\top \coef.
  \end{align}
  For all $i \in \{1, \dots, \sdim\}$ we have $\tilde \cgvec_i - \cgvec_i =
  \big(\tilde \ppol_{n - 1}(\lambda_i) - \ppol_{n - 1}(\lambda_i)\big) \weight_i$.
  Combining this with~\eqref{eq:qdiff} gives
  \begin{align*}
    \|\tilde \cgvec_{1:\sdim} - \cgvec\|^2 = \frac{1}{(1 - s)^2} \sum_{i = 1}^\sdim \bigg( \sum_{j = \sdim + 1}^\mdim \coef_j \pol_{n - 1}^j(\lambda_i) \weight_i \bigg)^2 = \frac{1}{(1 - s)^2} \sum_{j = \sdim + 1}^\mdim \sum_{k = \sdim + 1}^\mdim \coef_j\coef_k \inner{\pol_{n - 1}^j}{\pol_{n - 1}^k},
  \end{align*}
  where $\pol_{n - 1}^j$ is the polynomial we defined above.
  Lemma~\ref{lemma:decreasing-ppols} and the \cauchyschwarz{} inequality imply
  that $|\inner{\pol_{n - 1}^j}{\pol_{n - 1}^k}| \leq \|\cgvec\|^2$ for all $j, k$.
  Combine this with the triangle inequality to obtain $\|\tilde \cgvec_{1:\sdim}
  - \cgvec\|^2 \leq (1 - s)^{-2} \|\coef\|_1^2 \|\cgvec\|^2$.
  Now let $i \in \{\sdim + 1, \dots, \mdim\}$.
  From~\eqref{eq:qdiff} we compute
  \begin{align*}
    \tilde \ppol_{n - 1}(\lambda_i) = \frac{1}{\lambda_i}\bigg( 1 - \qpol_n(\lambda_i) + \frac{1}{1 - s} \sum_{j = \sdim + 1}^\mdim \coef_j \big( \xipol_{n - 1}^j(\lambda_i) - \qpol_n(\lambda_i) \big)\bigg).
  \end{align*}
  Then, by the triangle inequality and Lemma~\ref{lemma:bound-diff-q-xipol},
  $|\tilde \ppol_{n - 1}(\lambda_i)| \leq |\lambda_i|^{-1} |\qpol_n(\lambda_i) -
  1| \big( 1 + (1 - s)^{-1} \|\coef\|_1 \big)$.
  The polynomial $\qpol_n$ has $n$ real roots $\gamma_1 \geq \cdots \geq
  \gamma_n$, with $\gamma_n \geq \lambda_\sdim$
  (Lemma~\ref{lemma:rho-basic-props}).
  Thus, $\qpol_n(\lambda_i) = (1-\lambda_i/\gamma_1) \cdots
  (1-\lambda_i/\gamma_n) \geq 0$ (each factor is positive since $\lambda_i <
  \lambda_\sdim$ and $\lambda_\sdim > 0$).
  If $\lambda_i \leq 0$, then $1 \leq \qpol_n(\lambda_i) \leq (1 -
  \lambda_i/\lambda_\sdim)^n$.
  If $\lambda_i \geq 0$, then $1 \geq \qpol_n(\lambda_i) \geq (1 -
  \lambda_i/\lambda_\sdim)^n$.
  In all cases, $|\qpol_n(\lambda_i) - 1| \leq |(1 - \lambda_i/\lambda_\sdim)^n
  - 1|$.
  This gives the bound on $|\tilde \cgvec_i|$.
\end{proof}

Note that the bound $\|\cgvec\| \leq \lambda_\sdim^{-1} \|\weight\|$ always
holds when $\lambda_\sdim > 0$, because for $\mat \succ 0$ all iterates are smaller than the solution $\mat^{-1}\weight$ (see references in Lemma~\ref{lemma:iterates-cg-grow}).
Thus, the results above imply in particular $\|\tilde \cgvec_{1:\sdim} - \cgvec\|
\leq \lambda_\sdim^{-1}(1 - \|\coef\|_1)^{-1}\|\coef\|_1\|\weight\|$.

\subsection{Implications for our regime of interest}\label{subsec:effective-regime}

The important hypothesis in Theorems~\ref{th:identity-qs}
and~\ref{th:bounds-iterates} is that the entries of $\coef$ are small.
In this section, we show that this is indeed the case for the regime that will
matter in the local convergence analysis of trust-region methods in
Section~\ref{sec:application-rtr}.
Specifically, we assume that
\begin{align}\label{eq:effective-regime}
  |\lambda_{\sdim + 1}|, \dots, |\lambda_\mdim| \ll \lambda_\sdim
  && \text{and} &&
  \weight_{\sdim+1}^2 + \cdots + \weight_\mdim^2 \ll \weight_1^2 + \cdots + \weight_\sdim^2.
\end{align}
That is, the last $\mdim - \sdim$ eigenvalues of $\tmat$ are close to zero, and
the last $\mdim - \sdim$ components of $\tweight$ carry little weight.
Under these conditions, the matrix $C$ (defined in
Lemma~\ref{lemma:identity-rhos-2}) is close to the rank one matrix $\ones
\ones^\top$ and we can control the entries of $\coef$.

The polynomials $\xipol_{n - 1}$ and $\qpol_n$ satisfy $\xipol_{n - 1}(0) =
\qpol_n(0) = 1$ so the values $\xipol_{n - 1}(x)$ and $\qpol_n(x)$
are close to one when $x$ is close to zero.
To quantify this, we define
\begin{align}\label{eq:def-kappa-tau}
  \xidiv = \max_{i, j \in \{\sdim + 1, \dots, \mdim\}} \big|\xipol_{n - 1}^j(\lambda_{i}) - 1\big| && \text{and} && \qdiv = \max_{i \in \{\sdim + 1, \dots, \mdim\}} \big|\qpol_n(\lambda_i) - 1\big|.
\end{align}
Note that these two numbers depend on the iteration $n$ but it will always be
clear from context.
For each coefficient $\coef_j$, we define below a reference quantity $\beta_j =
b_j^2/D_{jj}$ and show $\coef_j = \beta_j + o(\beta_j)$ in the
regime~\eqref{eq:effective-regime} (the matrix $D$ comes from
Lemma~\ref{lemma:identity-rhos-2}).
We start with a helper lemma.

\begin{lemma}\label{lemma:general-psd-bound}
  Given $S, M \succeq 0$,
  the matrix $X = (I + S M)^{-1} S$ satisfies
  \begin{align*}
    \max_{ij} |X_{ij}| \leq \max_{ij} |S_{ij}|.
  \end{align*}
\end{lemma}
\begin{proof}
  We can factorize $M = K K^\top$ for some matrix $K$.
  The Woodbury formula provides
  \begin{align*}
    X = \pig( I - S K (I + K^\top S K)^{-1} K^\top\pig) S = S^{\sfrac{1}{2}} \pig( I - S^{\sfrac{1}{2}} K (I + K^\top S K)^{-1} K^\top S^{\sfrac{1}{2}} \pig) S^{\sfrac{1}{2}}.
  \end{align*}
  Let $S^{\sfrac{1}{2}} K = U \Sigma V^\top$ be a singular value decomposition.
  We can rewrite the previous identity as
  \begin{align*}
    X = S^{\sfrac{1}{2}} Y S^{\sfrac{1}{2}} && \text{where} && Y = U (I + \Sigma^2)^{-1} U^\top.
  \end{align*}
  Let $w$ be the column of $S^{\sfrac{1}{2}}$ of maximal $2$-norm.
  Notice that $\max_{ij} |S_{ij}| = \|w\|^2$.
  Now let $u$ and $v$ be the $i$th and $j$th columns of $S^{\sfrac{1}{2}}$
  respectively.
  Then
  \begin{align*}
    |X_{ij}| = |u^\top Y v| \leq \|u\|\|v\| \leq \|w\|^2 = \max_{ij} |S_{ij}|,
  \end{align*}
  where the first inequality is because $\|Y\| \leq 1$ and the second inequality
  is the maximality of $w$.
\end{proof}

\begin{lemma}\label{lemma:equiv-sigma}
  Under the assumptions (and with the notation) of
  Lemma~\ref{lemma:identity-rhos-2}, and with $\xidiv, \qdiv$ as
  in~\eqref{eq:def-kappa-tau}, the entries of $\coef$ satisfy
  \begin{align*}
    \big|\coef_j - \beta_j\big| \leq \Big(\qdiv + (1 + \qdiv)(1 + \xidiv) \|D^{-1}\| \|B\|_\frob^2 \Big) \beta_j && \text{where} && \beta = D^{-1}B^2.
  \end{align*}
\end{lemma}
\begin{proof}
  Apply the Woodbury formula to $\coef = (D + B^2 \cdot C)^{-1} B^2 w$ to obtain
  \begin{align*}
    \coef &= D^{-1} B^2 \big(I - X B^2\big) w && \text{where} && X = \big(I + C D^{-1} B^2\big)^{-1} C D^{-1}.
  \end{align*}
  If we let $G$ denote the Gram matrix $G_{ij} = \inner{\xipol_{n -
      1}^i}{\xipol_{n - 1}^j}$ then $C = D^{-1} G$ (in the same way as in the
  proof of Lemma~\ref{lemma:identity-rhos}).
  It follows that $C D^{-1} = D^{-1} G D^{-1}$ is positive semidefinite (because
  $D$ is positive definite).
  By Lemma~\ref{lemma:general-psd-bound} (with $S = C D^{-1}$ and $M = B^2$), it
  follows that
  \begin{align*}
    \max_{ij} |X_{ij}| \leq \max_{ij} \big|(C D^{-1})_{ij}\big| \leq (1 + \xidiv) \|D^{-1}\|.
  \end{align*}
  Finally, the $j$th entry of $\coef$ satisfies $\coef_j = \beta_j\big(w_j - ( X B^2
  w )_j\big)$, hence
  \begin{align*}
    \big|\coef_j - \beta_j\big| = \beta_j \bigg| (w_j - 1) - \sum_{k = \sdim + 1}^{\mdim} X_{jk} B_{kk}^2 w_k \bigg|,
  \end{align*}
  which we can bound with the triangle inequality and using $|w_k - 1| \leq
  \qdiv$ for all $k$ due to~\eqref{eq:def-kappa-tau}.
\end{proof}

\begin{corollary}\label{cor:bound-sigma}
  Under the assumptions (and with the notation) of
  Lemma~\ref{lemma:identity-rhos-2}, and with $\xidiv, \qdiv$ as
  in~\eqref{eq:def-kappa-tau}, the entries of $\coef$ satisfy
  \begin{align*}
    |\coef_j| \leq (1 + \qdiv)\Big(1 + (1 + \xidiv)\|D^{-1}\|\|B\|_\frob^2\Big) D_{jj}^{-1} B_{jj}^2.
  \end{align*}
\end{corollary}
\begin{proof}
  The triangle inequality gives $|\coef_j| \leq \beta_j + |\coef_j - \beta_j|$.
  Conclude with Lemma~\ref{lemma:equiv-sigma}.
\end{proof}

The bounds given in Lemma~\ref{lemma:equiv-sigma} and
Corollary~\ref{cor:bound-sigma} involve the entries of the
matrix $D$ (defined in Lemma~\ref{lemma:identity-rhos-2}).
We need to control them, and more specifically, to lower-bound
the norms $\|\xipol_{n - 1}^j\|$ for $j \in \{\sdim + 1, \dots, \mdim\}$.
We now (implicitly) exhibit an iteration for which we can secure such bounds.
Remember that $\grade$ denotes the grade of $(\mat, \weight)$.

\begin{lemma}\label{lemma:c-existence}
  We use the assumptions and notation of Lemma~\ref{lemma:identity-rhos-2}, and
  $\xidiv, \qdiv$ are as in~\eqref{eq:def-kappa-tau}.
  Given a parameter $\cparam \in \interval[open
  left]{0}{\lambda_\sdim^{-1}\|\weight\|^2}$, there exists an iteration $n \in
  \{1, \dots, \grade\}$ such that $\|\qpol_n\|^2 \leq \lambda_1 \cparam$ and
  $\|\xipol_{n - 1}^j\|^2 \geq \lambda_\sdim \cparam$ for all $j \in \{\sdim +
  1, \dots, \mdim\}$.
  In particular, this implies $\|\coef\|_1 \leq \omega$ at iteration $n$, where
  \begin{align*}
    \eta = \qdiv + (1 + \qdiv)(1 + \xidiv)^2 (\lambda_\sdim \cparam)^{-1} \|B\|_\frob^2 && \text{and} && \omega = (1 + \eta)(1 + \xidiv)(\lambda_\sdim \cparam)^{-1}\|B\|_\frob^2.
  \end{align*}
\end{lemma}
\begin{proof}
  Let $n \in \{1, \dots, \grade\}$ be the \emph{smallest} integer such that
  \begin{align}\label{eq:qpol-proof-opt}
    \min_{\qpol \in \pols{n}} \; \sum_{j = 1}^\sdim \frac{\qpol(\lambda_j)^2}{\lambda_j} \weight_j^2 \leq \cparam \qquad \text{subject to} \qquad \qpol(0) = 1.
  \end{align}
  Such an integer exists because when $n = \grade$ the minimum
  value of the above optimization problem is zero (by picking a polynomial
  whose roots are exactly the eigenvalues of $\mat$ with positive weight, see
  Remark~\ref{remark:grade}).
  Lemma~\ref{lemma:cg-pol-min} states that $\qpol_n$ is the solution of the
  optimization problem~\eqref{eq:qpol-proof-opt}.
  From the definition of the semi-norm $\|\cdot\|$
  in~\eqref{eq:both-semi-inner-product} it follows that $\|\qpol_n\|^2 \leq \lambda_1
  \cparam$.
  Since $n$ is the smallest integer with the property above, the cost function
  of~\eqref{eq:qpol-proof-opt} is larger than $c$ when evaluated at any feasible
  polynomial of degree $n-1$ or less.
  In particular, for all $j$ the polynomial $\xipol_{n - 1}^j$ is feasible
  for~\eqref{eq:qpol-proof-opt} because $\xipol_{n - 1}^j(0) = 1$.
  Again from the definition of the semi-norm $\|\cdot\|$
  in~\eqref{eq:both-semi-inner-product}, we deduce that $\|\xipol_{n - 1}^j\|^2 \geq
  \lambda_\sdim \cparam$ for all $j \in \{\sdim + 1, \dots, \mdim\}$.
  (This inequality also holds when $n = 1$ because $\xipol_0^j \equiv 1$ and we
  assumed $c \leq \lambda_\sdim^{-1}\|\weight\|^2$.)
  In particular, it implies that $\|D^{-1}\| \leq (1 + \xidiv)(\lambda_\sdim
  \cparam)^{-1}$.
  Apply Corollary~\ref{cor:bound-sigma} to obtain
  \begin{align*}
    \|\coef\|_1 \leq (1 + \qdiv)\Big(1 + (1 + \xidiv)\|D^{-1}\|\|B\|_\frob^2\Big) \sum_{j = \sdim + 1}^\mdim D_{jj}^{-1} B_{jj}^2 \leq (1 + \eta)(1 + \xidiv)(\lambda_\sdim c)^{-1}\|B\|_\frob^2 = \omega,
  \end{align*}
  where $\eta$ and $\omega$ are defined in this lemma statement.
\end{proof}

Provided the specific iteration $n$ from Lemma~\ref{lemma:c-existence} is
well defined on the problem $(\tmat, \tweight)$,
we can derive upper-bounds for the norms of its
corresponding iterate and residual.
We could obtain such bounds from Theorem~\ref{th:bounds-iterates}.
Here we obtain different ones using the special regime~\eqref{eq:effective-regime}.

\begin{lemma}\label{lemma:c-bounds}
  We use the assumptions and notation of Lemma~\ref{lemma:identity-rhos-2}, and
  $\xidiv, \qdiv$ are as in~\eqref{eq:def-kappa-tau}.
  Given a parameter $\cparam \in \interval[open
  left]{0}{\lambda_\sdim^{-1}\|\weight\|^2}$, let $\eta$ and $\omega$ be as in
  Lemma~\ref{lemma:c-existence}, and assume that $\omega < 1$.
  Then there exists an integer $n \in \{1, \dots, \grade\}$ such that the $n$th
  iteration of \cg{} on $(\tmat, \tweight)$ is well defined.
  Moreover, the $n$th iterate $\tilde \cgvec_n$ and residual $\tilde \cgres_n$
  satisfy
  \begin{align*}
    \|\tilde \cgvec_n\| \leq \frac{1}{(1 - \omega)
    \lambda_\sdim} \|\tweight\| && \text{and} && \|\tilde \cgres_n\|^2 \leq \frac{1}{(1 - \omega)^2}\bigg(
                     \lambda_1 \cparam + (1 + \eta)(1 + \xidiv)\omega\|B\|_\frob^2 + ( 1 +
                     \qdiv )^2 \|B\|_\frob^2 \bigg).
  \end{align*}
\end{lemma}
\begin{proof}
  Let $n$ be an integer as in Lemma~\ref{lemma:c-existence}.
  Since $\|\coef\|_1 \leq \omega < 1$, we can apply Theorem~\ref{th:identity-qs}
  and deduce that the $n$th iteration of \cg{} on $(\tmat, \tweight)$ is well
  defined.
  Corollary~\ref{cor:roots-lower-bound} gives that $(1 - \omega) \lambda_\sdim$
  is a lower-bound on the roots of $\tqpol_n$.
  Let $\tilde \krylovspace_n$ be the $n$th Krylov space associated to $(\tmat,
  \tweight)$.
  Lemma~\ref{lemma:roots-and-eigenvalues} ensures that
  $\restr{\tmat}{\tilde \krylovspace_n}{\normalscaling} \succeq (1 -
  \omega) \lambda_\sdim I$.
  As recalled in Section~\ref{par:connections-cg}, if we let $Q_n$ be an
  orthonormal basis of $\tilde \krylovspace_n$ then we can write the $n$th
  iterate of \cg{} as
  \begin{align*}
    \tilde \cgvec_n = Q_n\big(Q_n^\top \tmat Q_n^{} \big)^{-1}Q_n^\top \tweight.
  \end{align*}
  It follows that $\|\tilde \cgvec_n\| \leq (1 - \omega)^{-1} \lambda_\sdim^{-1}
  \|\tweight\|$.
  We now bound the norm of the residual.
  The identity~\eqref{eq:identity-qs} and the
  definition~\eqref{eq:both-semi-inner-product} of $\|\cdot\|_\sim$ give
  \begin{align*}
    \|\tilde \cgres_n\|^2 = \|\tqpol_n\|_\sim^2 = \frac{1}{(1 - s)^{2}} \bigg( \|\qpol_n\|^2 + \|p\|^2 + \sum_{i = \sdim + 1}^\mdim \Big(\qpol_n(\lambda_i) - \pol(\lambda_i)\Big)^2 \weight_i^2 \bigg) && \text{where} && \pol = \sum_{j = \sdim + 1}^\mdim \coef_j \xipol_{n - 1}^j
  \end{align*}
  and $s = \ones^\top \coef$.
  Here, we used the fact that $\inner{\qpol_n}{p} = 0$.
  We now bound each term separately.
  First, we have $\|\qpol_n\|^2 \leq \lambda_1 c$
  (Lemma~\ref{lemma:c-existence}).
  The triangle inequality and Corollary~\ref{cor:bound-sigma} give
  \begin{align*}
    \|\pol\| \leq \sum_{j = \sdim + 1}^\mdim \|\coef_j \xipol_{n - 1}^j\| \leq (1 + \eta) \sum_{j = \sdim + 1}^\mdim D_{jj}^{-1} B_{jj}^2 \|\xipol_{n - 1}^j\| = (1 + \eta) \sum_{j = \sdim + 1}^\mdim \xipol_{n - 1}^j(\lambda_j) \|\xipol_{n - 1}^j\|^{-1} B_{jj}^2,
  \end{align*}
  where we used the definition of $D$ in Lemma~\ref{lemma:identity-rhos-2} for
  the last equality.
  It follows that $\|\pol\|^2 \leq (1 + \eta)^2(1 + \xidiv)^2(\lambda_\sdim
  \cparam)^{-1}\|B\|_\frob^4$.
  We now consider the third term.
  For all $i \in \{\sdim + 1, \dots, \mdim\}$ we have
  \begin{align*}
    \qpol_n(\lambda_i) - \pol(\lambda_i) = (1 - s)\qpol_n(\lambda_i) + \sum_{j = \sdim + 1}^\mdim \coef_j\big( \qpol_n(\lambda_i) - \xipol_{n - 1}^j(\lambda_i) \big).
  \end{align*}
  By Lemma~\ref{lemma:bound-diff-q-xipol} and since $|s| \leq \|\coef\|_1 \leq \omega < 1$, the above is upper-bounded in absolute value as
  \begin{align*}
      |\qpol_n(\lambda_i) - \pol(\lambda_i)| \leq
        (1-s)\left( 1 + \qdiv + \frac{\|\coef\|_1}{1-s} \qdiv \right) \leq
        (1-s) \left( 1 + \qdiv + \frac{\qdiv \omega}{1 - \omega} \right) \leq
        (1-s) \frac{1 + \qdiv}{1 - \omega}.
  \end{align*}
  Thus,
  \begin{align*}
    \frac{1}{(1 - s)^2} \sum_{i = \sdim + 1}^\mdim \Big(\qpol_n(\lambda_i) - \pol(\lambda_i)\Big)^2 \weight_i^2 \leq \frac{1}{(1-\omega)^2} \left( 1 + \qdiv \right)^2 \|B\|_\frob^2.
  \end{align*}
  Combine the three bounds above
  to obtain the announced inequality.
\end{proof}

We conclude this section with two bounds on the quantities $\xidiv$ and $\qdiv$,
showing that they are close to zero in the regime~\eqref{eq:effective-regime}.

\begin{lemma}\label{lemma:bounds-kappa-tau}
  Suppose that $\lambda_\sdim > 0$ and $\lambda_\sdim > \lambda_{\sdim + 1}$.
  For all $n \in \{1, \dots, \grade\}$, the quantities $\xidiv$ and $\qdiv$, as
  defined in~\eqref{eq:def-kappa-tau}, are bounded as
  \begin{align*}
    \xidiv \leq \big(1 + \varepsilon/\lambda_\sdim\big)^{\grade - 1} - 1 && \text{and} && \qdiv \leq \big(1 + \varepsilon/\lambda_\sdim\big)^\grade - 1,
  \end{align*}
  where $\varepsilon = \max_{j = \sdim + 1, \dots, \mdim} \,|\lambda_j|$.
\end{lemma}
\begin{proof}
  Let $\gamma_1, \dots, \gamma_n \in \interval{\lambda_\sdim}{\lambda_1}$ be the
  roots of the polynomial $\qpol_n$.
  Then for all $x \leq \lambda_\sdim$ we have
  \begin{align*}
    |\qpol_n(x) - 1| = \bigg| \prod_{i = 1}^n \big( 1 - x/\gamma_i \big) - 1 \bigg| \leq \prod_{i = 1}^n \big( 1 + |x|/\gamma_i \big) - 1 \leq \big( 1 + |x|/\lambda_\sdim \big)^n - 1.
  \end{align*}
  This gives the inequality for $\qdiv$.
  A similar argument yields the inequality for $\xidiv$.
\end{proof}

\section{Application to trust-region algorithms}\label{sec:application-rtr}

In this section we study the local convergence of trust-region algorithms (TR)
to non-isolated minima, assuming the \polyakloja{} condition~\eqref{eq:pl}.
We secure superlinear convergence when the subproblem solver is \tcg{}
(Algorithm~\ref{alg:tcg}).
For this, we leverage the new analysis of \cg{} developed in
Section~\ref{sec:indefinite-cg}.

We let $\retr \colon \tangent\manifold \to \manifold$ denote a
retraction~\cite[\S4.1]{absil2008optimization}, and use the notation $\retr_x(s)
= \retr(x, s)$ for convenience.
This is a smooth map from the tangent bundle $\tangent\manifold$ such that
$\retr_x(\zeros) = x$ and $\D \retr_x(\zeros) = I$ for all $x \in \manifold$.
In the Euclidean case we usually pick $\retr_x(s) = x + s$.

TR produces a sequence $\sequence{(x_k, \Delta_k)}$, where $x_k$ is the
current iterate and $\Delta_k$ is the trust-region radius.
At iteration $k$, we define a second-order model of $\mfc$ around $x_k$ as
\begin{align}\label{eq:rtr-model}\tag{TRM}
  \rtrmodel_k(s) = \mfc(x_k) + \inner{s}{\grad \mfc(x_k)} + \frac{1}{2}\inner{s}{\linearmap_k[s]},
\end{align}
where $\linearmap_k$ is a linear map close to $\hess \mfc(x_k)$---see
Assumption~\aref{assu:hess-approx}.
From this model, a step $s_k$ is computed by (usually approximately) solving the
trust-region subproblem
\begin{align}\label{eq:rtr-subproblem}\tag{TRS}
  \min_{s_k \in \tangent_{x_k}\manifold} \rtrmodel_k(s_k) \qquad \text{subject to} \qquad \|s_k\| \leq \Delta_k.
\end{align}
The point $x_k$ and radius $\Delta_k$ are then updated depending on the ratio
\begin{align}\label{eq:rtr-rho}
  \rtrrho_k = \frac{\mfc(x_k) - \mfc(\retr_{x_k}(s_k))}{\rtrmodel_k(\zeros) - \rtrmodel_k(s_k)}.
\end{align}
(If the denominator is zero, we let $\rtrrho_k = 1$.)
Specifically, given parameters $\rtrrho' \in \interval[open]{0}{\frac{1}{4}}$
and $\bar\Delta > 0$, the update rules for the state are
\begin{align}\label{eq:rtr-dynamics}
  x_{k + 1} = \begin{cases}
                \retr_{x_k}(s_k) &\text{if $\rtrrho_k > \rtrrho'$},\\
                x_k &\text{otherwise},
              \end{cases}
  \qquad\quad
  \Delta_{k + 1} = \begin{cases}
                     \frac{1}{4}\Delta_k &\text{if $\rtrrho_k < \frac{1}{4}$},\\
                     \min(2\Delta_k, \bar\Delta) &\text{if $\rtrrho_k > \frac{3}{4}$ and $\|s_k\| = \Delta_k$},\\
                     \Delta_k &\text{otherwise}.
                   \end{cases}
\end{align}
An iteration $k$ is \emph{successful} when $\rtrrho_k > \rtrrho'$.

The main result of this section was stated as Theorem~\ref{th:rtr-main-theorem},
with three assumptions around a local minimum $\optpoint$.
The first two require that the Hessian is sufficiently regular (note that
\aref{assu:hess-lip} implies~\aref{assu:hess-lip-like} with $\hessliplikeconst =
\hesslipconst$ when $\retr = \Exp$---see for
example~\citep[Cor.~10.56]{boumal2020introduction}).
The last assumption further requires that the map $\linearmap_k$ is sufficiently
close to $\hess \mfc(x_k)$.
It holds in particular if $\linearmap_k = \hess \mfc(x_k)$ or $\linearmap_k =
\hess(\mfc \circ \retr_{x_k})(0)$.
\begin{assumption}\label{assu:hess-lip}
  The Hessian $\hess \mfc$ is locally $\hesslipconst$-Lipschitz continuous
  around $\optpoint$ for some $\hesslipconst \geq 0$.
\end{assumption}
\begin{assumption}\label{assu:hess-lip-like}
  There exists a constant $\hessliplikeconst \geq 0$ such that the
  Lipschitz-type inequality
  \begin{align}\label{eq:hess-lip-like}
    \mfc(\retr_{x}(s)) - \mfc(x) - \inner{s}{\grad \mfc(x)} - \frac{1}{2}\inner{s}{\hess \mfc(x)[s]} \leq \frac{\hessliplikeconst}{6}\|s\|^3
  \end{align}
  holds for all $x$ sufficiently close to $\optpoint$ and all $s$ sufficiently small.
\end{assumption}
\begin{assumption}\label{assu:hess-approx}
  The linear maps $\linearmap_0, \linearmap_1, \ldots$ in~\eqref{eq:rtr-model}
  are symmetric.
  There is a constant $\hessapproxconst \geq 0$ such that
  \begin{align}\label{eq:hess-approx}
    \|\linearmap_k - \hess \mfc(x_k)\| \leq \hessapproxconst \|\grad \mfc(x_k)\|
  \end{align}
  whenever the iterate $x_k$ is sufficiently close to $\optpoint$.
\end{assumption}

\subsection{The subproblem solver \tcg{}
  satisfies~\cref{cond:better-than-cauchy},~\cref{cond:strong-vs}
  and~\cref{cond:small-model}}\label{subsec:tcg-c1-c2}

Recall that \tcg{} refers to Algorithm~\ref{alg:tcg} and \cg{} refers to the
standard conjugate gradient algorithm, that is, \tcg{} without the two
truncation parts.
Given an iterate $x_k$, we feed $\mat = \linearmap_k$, $\weight = -\grad
\mfc(x_k)$ and $\Delta = \Delta_k$ to the subproblem solver \tcg{}.
It runs a certain number $n \leq \dim \manifold$ of iterations and we set the
step $s_k$ to be the resulting \tcg{} iterate $v_n$.

In Section~\ref{subsec:C0C1C2} we introduced the
conditions~\cref{cond:better-than-cauchy}, \cref{cond:strong-vs}
and~\cref{cond:small-model}.
Here we prove that \tcg{} satisfies those three conditions under~\eqref{eq:pl}.
Given an iterate $x_k$, the Cauchy step is the minimizer of the
subproblem~\eqref{eq:rtr-subproblem} with the additional constraint that $s_k
\in \vecspan \grad \mfc(x_k)$.
It satisfies the sufficient decrease condition~\cref{cond:better-than-cauchy}
with constant $\rtrsufficientdecrease = 1/2$~\citep[Thm.~6.3.1]{conn2000trust}.
The Cauchy step happens to be the first iterate of \tcg{}, and subsequent
iterates can only improve.
So, it follows that \tcg{} also satisfies~\cref{cond:better-than-cauchy} with
$\rtrsufficientdecrease = 1/2$---see~\citep[\S7.5.1]{conn2000trust}.
The rest of this section focuses on conditions~\cref{cond:strong-vs}
and~\cref{cond:small-model}.

\paragraph{Local behavior of \cg{}.}

To analyze \tcg{} we rely on the developments from
Section~\ref{sec:indefinite-cg}.
Suppose that $\mfc$ is $\plconstant$-\eqref{eq:pl} around a local minimum
$\optpoint$.
Then, since $\mfc$ is $\smooth{2}$, it is also $\plconstant$-\eqref{eq:morse-bott} at $\optpoint$, as explained
in Section~\ref{subsec:intuition}.
(Note that when $\mfc$ is only $\smooth{1}$ its solution set may not be
differentiable, even under \pl{}.)
In particular, the set $\optimalset$ of local minima~\eqref{eq:def-s} is a smooth submanifold of $\manifold$
around $\optpoint$ and $\dim \optimalset = \dim \ker \hess \mfc(\optpoint)$.
Let $\sdim$ be the codimension of $\optimalset$ around $\optpoint$.
The $\sdim$ nontrivial eigenvalues of $\hess \mfc(\optpoint)$ are at least
$\plconstant > 0$.
The continuity of the eigenvalues of $\hess \mfc$ implies that there exists a
neighborhood $\mathcal{U}$ of $\optpoint$ such that for all $x \in \mathcal{U}$
the orthogonal projector $P(x)$ onto the top $\sdim$ eigenspace of $\hess
\mfc(x)$ is well defined.
We first recall that $\grad \mfc(x)$ aligns primarily with that
eigenspace when $x$ is close to $\optpoint$.

\begin{lemma}\label{lemma:grad-image-hess}
  Suppose~\eqref{eq:pl} and~\aref{assu:hess-lip} hold around $\optpoint \in
  \optimalset$.
  Then, with the same notation as above, %
  \begin{align*}
    \grad \mfc(x) & = P(x) \grad \mfc(x) + O\big(\|\grad \mfc(x)\|^2\big)
  \end{align*}
  as $x \to \optpoint$ and $x \in \mathcal{U}$. %
\end{lemma}
\begin{proof}
  See~\cite[Lem.~2.14]{rebjock2023fast}.
\end{proof}

The algorithm \tcg{} is \cg{} with three stopping criteria.
The first two (lines~\ref{line:truncation-1}--\ref{line:truncation-1-end})
trigger when the iteration is not well defined (that is, a non-positive
eigenvalue is detected) or when the iterate exits the trust region.
In these cases the output of \tcg{} lies at the boundary of the trust region.
The third stopping criterion
(lines~\ref{line:truncation-2}--\ref{line:truncation-2-end}) triggers if the
residual norm is small.
In this case the output of \tcg{} is the iterate of \cg{}.

The conditions~\cref{cond:strong-vs} and~\cref{cond:small-model} bound the norms
of the iterate and its residual for the output of \tcg{}.
To secure them we need the third stopping criterion to trigger (before the two
others do), as otherwise we cannot control the norm of the output (which would
lie at the boundary of the trust region).
The purpose of the next lemma is to argue that this indeed happens around local
minima where \pl{} holds.
To do this, we identify a particular iteration of \cg{} with appropriate bounds
on both the iterate and the residual.
This is because \cg{} implicitly filters out the small eigenvalues for a certain number of
iterations.
This result is central to establish the upcoming convergence guarantees of TR
with \tcg{}.
From this, we argue later in this section that it is indeed the third stopping
criterion that triggers.

\begin{lemma}\label{lemma:cg-iterate-existence}
  Assume that $\mfc$ is $\plconstant$-\eqref{eq:pl} and
  $\hess \mfc$ is locally Lipschitz (\aref{assu:hess-lip})
  around $\optpoint \in \optimalset$.
  Let $\sdim$ be the codimension of $\optimalset$ around $\optpoint$.
  Given $\muflat < \plconstant$, $\power \in \interval[open right]{0}{1}$ and
  $\hessapproxconst \geq 0$, there exists a neighborhood $\mathcal{U}$ of
  $\optpoint$ such that for each $x \in \mathcal{U}$ and symmetric map $H \colon
  \tangent_x\manifold \to \tangent_x\manifold$ with $\|H - \hess \mfc(x)\| \leq
  \hessapproxconst \|\grad \mfc(x)\|$ the following holds.
  There exists an integer $n \in \{0, \dots, \sdim\}$ such that the $n$th
  iteration of \cg{} on $(H, -\grad \mfc(x))$ is well defined and, moreover, the
  \cg{} iterate $\cgvec_n$ and residual $\cgres_n$ satisfy
  \begin{align}\label{eq:cg-iterate-guarantees}
    \|\cgvec_n\| \leq \frac{1}{\muflat}\|\grad \mfc(x)\| && \text{and} && \|\cgres_n\| \leq \|\grad \mfc(x)\|^{1 + \power}.
  \end{align}
\end{lemma}
\begin{proof}
  Apply~\cite[Cor.~2.17]{rebjock2023fast} to confirm that $\mfc$ satisfies
  $\plconstant$-\eqref{eq:morse-bott} at $\optpoint$.
  Let $\mdim$ be the dimension of $\manifold$ (the domain of $\mfc$) and let
  $\sdim \leq \mdim$ be the codimension of the solution set $\optimalset$
  (around $\optpoint$).
  For each $x \in \manifold$ we need to control the eigenvalues of all symmetric
  linear maps $H \colon \tangent_x \manifold \to \tangent_x \manifold$ which
  satisfy $\|H - \hess \mfc(x)\| \leq \hessapproxconst \|\grad \mfc(x)\|$.
  To this end, we define three functions to control the eigenvalues of the map
  $H$ depending on $x$.
  First, introduce
  \begin{align*}
    \lambda_-(x) = \lambda_\sdim\big(\hess \mfc(x)\big) - \hessapproxconst \|\grad \mfc(x)\|
    && \text{and} &&
    \lambda_+(x) = \lambda_1\big(\hess \mfc(x)\big) + \hessapproxconst \|\grad \mfc(x)\|
  \end{align*}
  to bracket the top $\sdim$ eigenvalues of $\linearmap$.
  Then, let
  \begin{align*}
    \varepsilon(x) = \hessapproxconst \|\grad \mfc(x)\| + \max_{j = \sdim + 1, \dots, \mdim} \big|\lambda_j\big(\hess \mfc(x)\big)\big|
  \end{align*}
  to bound the last $\mdim - \sdim$ eigenvalues.
  Standard theory for eigenvalue perturbation confirms that
  \begin{align*}
    \lambda_-(x) \leq \lambda_\sdim(H) \leq \lambda_1(H) \leq \lambda_+(x) && \text{and} && \max_{j = \sdim + 1, \dots, \mdim} |\lambda_j(H)| \leq \varepsilon(x).
  \end{align*}
  The functions $\lambda_-$, $\lambda_+$ and $\varepsilon$ are continuous.
  They also satisfy $\lambda_-(\optpoint) \geq \plconstant$,
  $\lambda_+(\optpoint) = \|\hess \mfc(\optpoint)\|$ and $\varepsilon(\optpoint)
  = 0$.
  There exists a neighborhood $\mathcal{U}$ of $\optpoint$ such that
  $\varepsilon(x) < \lambda_-(x)$ for all $x \in \mathcal{U}$.
  In particular, for all $x \in \mathcal{U}$ and map $H$ as prescribed, the
  orthogonal projector $P_H$ onto the top $\sdim$ eigenspace of $H$ is well
  defined.
  Let also $P(x)$ denote the orthogonal projector onto the top $\sdim$
  eigenspace of $\hess \mfc(x)$.
  Using the equality\footnote{This is because $\|P(x)(I - P_H)v\| \leq \|(P_H -
    P(x))v\|$ for all $v$, with equality when $v$ is an eigenvector of $P_H -
    P(x)$.
  }
  $\|P_H - P(x)\| = \|P(x) (I - P_H)\|$, Davis--Kahan's
  theorem~\citep[Thm.~VII.3.1]{bhatia1997matrixanalysis} implies that
  \begin{align*}
    \|P_H - P(x)\| \leq \frac{\|H - \hess \mfc(x)\|}{\lambda_\sdim\big(\hess \mfc(x)\big) - \varepsilon(x)} \leq \frac{\hessapproxconst \|\grad \mfc(x)\|}{\lambda_\sdim\big(\hess \mfc(x)\big) - \varepsilon(x)}.
  \end{align*}
  Combining this inequality with Lemma~\ref{lemma:grad-image-hess} yields that
  there exists a constant $C \geq 0$ such that
  \begin{align}\label{eq:grad-aligned-h}
    \|\grad \mfc(x) - P_H \grad \mfc(x)\|^2 \leq C \|\grad \mfc(x)\|^4
  \end{align}
  for all $x$ sufficiently close to $\optpoint$ and map $H$ as prescribed.
  If need be,
  restrict $\mathcal{U}$ to a smaller neighborhood of $\optpoint$
  so that~\eqref{eq:grad-aligned-h}, $\|\grad \mfc(x)\|^2 \leq 1$ and $C \|\grad
  \mfc(x)\|^2 \leq 1$ hold for all $x$ in $\mathcal{U}$.
  We define the functions
  \begin{align}\label{eq:baruxsquxsq}
    u(x) = \|\grad \mfc(x)\|^2\big(1 - C \|\grad \mfc(x)\|^2\big) && \text{and} && \bar u(x) = C \|\grad \mfc(x)\|^4
  \end{align}
  on $\mathcal{U}$.
  From~\eqref{eq:grad-aligned-h} we deduce that for all $x \in \mathcal{U}$ and
  map $H$ as prescribed we have
  \begin{align}\label{eq:ph-u-ubar}
    \|(I - P_H) \grad \mfc(x)\|^2 \leq \bar u(x)
    && \text{and} &&
    u(x) \leq \|P_H \grad \mfc(x)\|^2 \leq \|\grad \mfc(x)\|^2,
  \end{align}
  where we used the identity $\|P_H \grad \mfc(x)\|^2 = \|\grad \mfc(x)\|^2 -
  \|(I - P_H) \grad \mfc(x)\|^2$ for the lower-bound on the right.
  At a point $x$ with appropriate map $H$ we aim to invoke %
  Lemma~\ref{lemma:c-bounds} with parameter $c(x, H) =
  \lambda_\sdim(H)^{-1}u(x)^{\frac{3 + \power}{2}}$.
  (Notice that $u(x) \leq 1$, hence $c(x, H) \leq \lambda_\sdim(H)^{-1}u(x)$;
  together with $u(x) \leq \|P_H \grad \mfc(x)\|^2$, this will ensure that $c(x,
  H)$ is an appropriate choice of $c$ for Lemma~\ref{lemma:c-bounds}.)
  To this end, consider the bounds in
  Lemma~\ref{lemma:bounds-kappa-tau} and define
  \begin{align*}
    \xidiv(x) &= \big(1 + \varepsilon(x)/\lambda_-(x)\big)^{\sdim - 1} - 1, & \eta(x) & = \qdiv(x) + \big(1 + \qdiv(x)\big) \big(1 + \xidiv(x)\big)^2 \bar u(x)/u(x)^{\frac{3 + \power}{2}},\\
    \qdiv(x) &= \big(1 + \varepsilon(x)/\lambda_-(x)\big)^\sdim - 1, & \omega(x) & = \big(1 + \eta(x)\big) \big(1 + \xidiv(x)\big) \bar u(x)/u(x)^{\frac{3 + \power}{2}}.
  \end{align*}
  These functions are continuous around $\optpoint$ and their value at
  $\optpoint$ is zero since, by~\eqref{eq:baruxsquxsq},
  \begin{align*}
    \bar u(x)/u(x)^{\frac{3 + \power}{2}} = C \|\grad \mfc(x)\|^{1 - \power}\big(1 - C \|\grad \mfc(x)\|^2\big)^{-(3 + \power)/2} %
  \end{align*}
  for all $x \in \mathcal{U}$.
  For the hypotheses of Lemma~\ref{lemma:c-bounds} to hold, we further restrict
  the neighborhood $\mathcal{U}$ so that $\omega(x) < 1$ for all $x \in
  \mathcal{U}$.
  Now pick a specific $x \in \mathcal{U}$ with map $H$ as prescribed.
  If $\grad \mfc(x) = \zeros$ then the iteration $n = 0$ satisfies the
  requirements~\eqref{eq:cg-iterate-guarantees}.
  Suppose now that $\grad \mfc(x) \neq \zeros$.
  Apply Lemma~\ref{lemma:c-bounds} to the pair $(H, -\grad \mfc(x))$
  with $c = c(x, H)$.
  Combine it with the bounds given in Lemma~\ref{lemma:bounds-kappa-tau}.
  It yields the existence of an integer $n$ such that \emph{(i)} the $n$th
  iteration of \cg{} is well defined, \emph{(ii)} the iterate $\cgvec_n$
  satisfies $\|\cgvec_n\| \leq \lambda_-(x)^{-1}(1 - \omega(x))^{-1}\|\grad
  \mfc(x)\|$, and \emph{(iii)} the residual $\cgres_n$ satisfies
  \begin{align*}
    \|\cgres_n\|^2 \leq
    \big(1 - \omega(x)\big)^{-2}
    \bigg(\frac{\lambda_+(x)}{\lambda_-(x)}u(x)^{\frac{3 + \power}{2}}
    + \big(1 + \eta(x)\big)^2\big(1 + \xidiv(x)\big)^2 \bar u(x)^2 / u(x)^{\frac{3 + \power}{2}}
    + \big(1 + \qdiv(x)\big)^2 \bar u(x) \bigg).
  \end{align*}
  First notice that the bound $\|\cgvec_n\| \leq \|\grad \mfc(x)\|/\muflat$
  holds when $x$ is sufficiently close to $\optpoint$ because
  $\lambda_-(\optpoint) \geq \plconstant > \muflat$ and $\omega(\optpoint) = 0$.
  Now consider the residual norm.
  When $x$ is sufficiently close to $\optpoint$ the following inequalities hold:
  \begin{align*}
    u(x)^{\frac{3 + \power}{2}} \leq \|\grad \mfc(x)\|^{3 + \power}, &&&& \bar u(x)^2 / u(x)^{\frac{3 + \power}{2}} \leq 2 C^2 \|\grad \mfc(x)\|^{5 - \power} && \text{and} && \bar u(x) \leq C \|\grad \mfc(x)\|^4.
  \end{align*}
  We deduce that $\|\cgres_n\| \leq \|\grad \mfc(x)\|^{1 + \power}$ if $x$ is
  sufficiently close to $\optpoint$ because $3 + \power > 2 + 2 \power$.
  (The strict inequality $\power < 1$ plays a role here, making it possible to
  absorb multiplicative factors.)
\end{proof}

Note that the conclusion of Lemma~\ref{lemma:cg-iterate-existence} is immediate
if one makes the stronger assumption $\hess \mfc(\optpoint) \succ \zeros$.
In that case we can choose the iteration $n$ to be the grade of $(H, -\grad
\mfc(x))$, as per Definition~\ref{def:grade}, to obtain the requirements on
$\cgvec_n$ and $\cgres_n$.
However, this is in general not possible when only \pl{} holds.
This is because after a certain number of iterations, \cg{} starts to detect the
small eigenvalues of $\linearmap$.
When this happens, the iterates of the algorithm typically explode and the
bounds in~\eqref{eq:cg-iterate-guarantees} no longer hold.
We illustrate this phenomenon in Figures~\ref{fig:vn-rn-norms}
and~\ref{fig:pols}.
Remarkably, \tcg{} stops automatically before that happens, as we argue subsequently.

\begin{remark}
  Lemma~\ref{lemma:cg-iterate-existence} does not specify the size of the
  neighborhood $\mathcal{U}$ in which the conclusion holds because we have only
  little control over it.
  Notice that the magnitudes of $\xidiv$ and $\qdiv$, as defined
  in~\eqref{eq:def-kappa-tau}, might depend on the dimension of the problem, as
  reflected by Lemma~\ref{lemma:bounds-kappa-tau}.
  This induces for the size of $\mathcal{U}$ a dependency on the codimension of
  $\optimalset$.
  However, the poor bounds for $\xidiv$ and $\qdiv$ in
  Lemma~\ref{lemma:bounds-kappa-tau} are due to the \emph{negative} eigenvalues.
  As a result, it should be possible to derive much sharper bounds in the
  regions near $\optpoint$ where all the eigenvalues are positive.
\end{remark}

\paragraph{Application to TR with \tcg{}.}
A first consequence of Lemma~\ref{lemma:cg-iterate-existence} is that TR with
\tcg{} satisfies~\cref{cond:strong-vs} around points where \pl{} holds.
To show this, we use the following classical result.

\begin{lemma}\label{lemma:iterates-cg-grow}
  The iterates of \cg{} and \tcg{} grow in norm: $0 = \|\cgvec_0\| \leq \|\cgvec_1\| \leq \|\cgvec_2\| \leq \cdots$.
\end{lemma}
\begin{proof}
  See~\cite[Thm.~7.5.1]{conn2000trust} or~\cite[Thm.~7.3]{nocedal2006numerical}.
\end{proof}

\begin{proposition}\label{prop:rtr-tcg-strong-vs}
  Consider \tcg{} with parameters $\tcgkappa > 0$ and $\power \in
  \interval[open]{0}{1}$.
  Assume that~\aref{assu:hess-lip}, \aref{assu:hess-approx} and
  $\plconstant$-\eqref{eq:pl} hold around $\optpoint \in \optimalset$.
  Given $\muflat < \plconstant$, there exists a neighborhood of $\optpoint$
  where the \tcg{} steps satisfy~\cref{cond:strong-vs} with constant
  $\strongvsconst = 1/\muflat$.
\end{proposition}
\begin{proof}
  Let $\mathcal{U}$ be a neighborhood of $\optpoint$ as in
  Lemma~\ref{lemma:cg-iterate-existence} for the parameters $\muflat$, $\power$
  and $\hessapproxconst$ (\aref{assu:hess-approx}).
  Shrink $\mathcal{U}$ so that $\|\grad \mfc(x)\|^\power \leq \tcgkappa$ for all
  $x \in \mathcal{U}$.
  Given an iterate $x_k \in \mathcal{U}$ with $\linearmap_k$, let $n$ be as provided by
  Lemma~\ref{lemma:cg-iterate-existence} (so the
  bounds~\eqref{eq:cg-iterate-guarantees} hold).
  Algorithm~\ref{alg:tcg} on $(\linearmap_k, -\grad \mfc(x_k))$ terminates
  either with $\cgvec_n$ (because the residual norm stopping criterion on
  line~\ref{line:truncation-2} triggers), or earlier.
  In all cases, Lemma~\ref{lemma:iterates-cg-grow} implies the returned
  \tcg{} step has norm at most that of $\cgvec_n$, which itself is at most
  $\frac{1}{\muflat}\|\grad \mfc(x_k)\|$.
\end{proof}

We now leverage Lemma~\ref{lemma:cg-iterate-existence} to
establish~\cref{cond:small-model}.
The residual of \cg{} satisfies $\cgres_n = -\grad \rtrmodel(\cgvec_n)$, where
$\rtrmodel$ is the model~\eqref{eq:rtr-model}.
It follows that~\eqref{eq:cg-iterate-guarantees} also provides a bound on $\grad \rtrmodel(\cgvec_n)$.
We first state a helper lemma.

\begin{lemma}\label{lemma:ratios-converge}
  Suppose that~\aref{assu:hess-lip-like} and~\aref{assu:hess-approx} hold around
  $\optpoint \in \optimalset$.
  Also assume that the steps $s_k$ satisfy~\cref{cond:better-than-cauchy}
  and~\cref{cond:strong-vs} around $\optpoint$.
  For all $\varepsilon > 0$ there exists a neighborhood $\mathcal{U}$ of
  $\optpoint$ such that if an iterate $x_k$ is in $\mathcal{U}$ then $\rtrrho_k
  \geq 1 - \varepsilon$, where $\rtrrho_k$ is the ratio~\eqref{eq:rtr-rho}.
\end{lemma}
\begin{proof}
  See~\cite[Prop.~4.10]{rebjock2023fast}.
\end{proof}

\begin{proposition}\label{prop:rtr-tcg-small-model}
  Let $\sequence{x_k}$ be a sequence that TR generates using the \tcg{}
  subproblem solver (Algorithm~\ref{alg:tcg}) with parameters $\tcgkappa > 0$
  and $\power \in \interval[open]{0}{1}$.
  Suppose that $\sequence{x_k}$ converges to a point $\optpoint \in \optimalset$
  around which $\mfc$ is $\plconstant$-\eqref{eq:pl}.
  Also assume that~\aref{assu:hess-lip}, \aref{assu:hess-lip-like}
  and~\aref{assu:hess-approx} hold around $\optpoint$.
  Then the sequence $\sequence{(x_k, s_k)}$ satisfies~\cref{cond:small-model}
  with constants $\modelconst = 1$ and $\power$.
\end{proposition}
\begin{proof}
  The subproblem solver \tcg{} satisfies~\cref{cond:better-than-cauchy}
  and~\cref{cond:strong-vs} around $\optpoint$ (see the beginning of
  Section~\ref{subsec:tcg-c1-c2} and Proposition~\ref{prop:rtr-tcg-strong-vs}).
  Lemma~\ref{lemma:ratios-converge} gives that $\liminf_{k \to +\infty}
  \rtrrho_k \geq 1$.
  This implies that the radii $\sequence{\Delta_k}$ are bounded away from zero
  because the update mechanism~\eqref{eq:rtr-dynamics} does not decrease the
  radius when $\rtrrho_k \geq \frac{1}{4}$.
  Given $\muflat < \plconstant$, the \tcg{} parameter $\power$, and the constant
  $\hessapproxconst$ from~\aref{assu:hess-approx}, let $\mathcal{U}$ be a
  neighborhood of $\optpoint$ as in Lemma~\ref{lemma:cg-iterate-existence}.
  There exists an integer $K$ such that for all $k \geq K$ we have
  \begin{align*}
    x_k \in \mathcal{U}, && \|\grad \mfc(x_k)\|^\power \leq \tcgkappa && \text{and} && \frac{1}{\muflat}\|\grad \mfc(x_k)\| < \Delta_k.
  \end{align*}
  Now let $k \geq K$.
  Lemma~\ref{lemma:cg-iterate-existence} provides an integer $n$ such that the
  $n$th iteration of \cg{} on $(\linearmap_k, -\grad \mfc(x_k))$ is well
  defined.
  Moreover, the $n$th iterate $\cgvec_n$ and residual $\cgres_n$ satisfy
  \begin{align*}
    \|\cgvec_n\| \leq \frac{1}{\muflat}\|\grad \mfc(x_k)\| < \Delta_k && \text{and} && \|\cgres_n\| \leq \|\grad \mfc(x_k)\|^{1 + \power}.
  \end{align*}
  The residual norm $\|\cgres_n\|$ is compatible with the termination criterion
  in line~\ref{line:truncation-2}.
  We deduce that \tcg{} performs a certain number $p \leq n$ of iterations on
  the inputs $(\linearmap_k, -\grad \mfc(x_k))$.
  Recall that the iterates grow in norm
  (Lemma~\ref{lemma:iterates-cg-grow}).
  Hence at iteration $p$ the termination criteria in
  line~\ref{line:truncation-1} are incompatible with \emph{(i)} the fact that
  the $p$th iteration is well defined and \emph{(ii)} the inequality
  $\|\cgvec_p\| \leq \|\cgvec_n\| < \Delta_k$.
  It follows that \tcg{} must terminate from the stopping criterion on the
  residual norm in line~\ref{line:truncation-2}.
  We obtain the inequality $\|\grad \rtrmodel_k(s_k)\| \leq \|\grad
  \mfc(x_k)\|^{1 + \power}$, and so~\cref{cond:small-model} holds.
\end{proof}

\subsection{Capture and order of convergence}
\label{subsec:captureandorder}

In this section we consider a general subproblem solver compatible
with~\cref{cond:better-than-cauchy}, \cref{cond:strong-vs}
and~\cref{cond:small-model},
with \tcg{} as the motivating example.
We derive a capture result and superlinear rates of convergence.

\paragraph{Capture of the iterates.}

First, we show that~\cref{cond:better-than-cauchy} and~\cref{cond:strong-vs}
imply that the iterates of TR locally satisfy a condition known as the
\emph{strong decrease property}~\eqref{eq:strong-decrease} (see~\citep{absil2005convergence} and
also~\citep[\S3.2]{rebjock2023fast} for more literature).
The proofs below combine the bound~\cref{cond:strong-vs} and arguments that appear
in~\cite[Thm.~4.4]{absil2005convergence}.

\begin{lemma}\label{lemma:ratio-hess-inner}
  Suppose that~\aref{assu:hess-approx} holds around $\optpoint \in \optimalset$.
  Given $\lamsharp > \lambda_{\max}(\hess \mfc(\optpoint))$, there exists a
  neighborhood $\mathcal{U}$ of $\optpoint$ such that if an iterate $x_k$ is in
  $\mathcal{U}$ then
  \begin{align}\label{eq:ratio-hess-inner}
    \frac{\|\grad \mfc(x_k)\|^3}{\inner{\grad \mfc(x_k)}{\linearmap_k[\grad \mfc(x_k)]}} \geq \frac{1}{\lamsharp}\|\grad \mfc(x_k)\|.
  \end{align}
  (We define the expression on the left as $+\infty$ when the denominator is
  zero.)
\end{lemma}
\begin{proof}
  Let $\mathcal{U}$ be a neighborhood of $\optpoint$ such that $\|\hess
  \mfc(x)\| + \hessapproxconst \|\grad \mfc(x)\| \leq \lamsharp$ for all $x \in
  \mathcal{U}$, where $\hessapproxconst$ is the constant that appears
  in~\aref{assu:hess-approx}.
  Suppose that $x_k$ is in $\mathcal{U}$.
  \cauchyschwarz{} yields
  \begin{align*}
    \inner{\grad \mfc(x_k)}{\linearmap_k[\grad \mfc(x_k)]} \leq \lamsharp\|\grad \mfc(x_k)\|^2
  \end{align*}
  because $\|\linearmap_k\| \leq \|\hess \mfc(x_k)\| + \|\linearmap_k - \hess
  \mfc(x_k)\| \leq \lamsharp$.
  We obtain~\eqref{eq:ratio-hess-inner} from this.
\end{proof}

\begin{lemma}\label{lemma:retr-dist-bound}
  Suppose~\cref{cond:strong-vs} holds in a neighborhood $\mathcal{U}$ of
  $\optpoint \in \optimalset$.
  Possibly after restricting $\mathcal{U}$, there exists $\retrdistboundconst
  \geq 1$ such that
  \begin{align}\label{eq:retr-dist-bound}
    \dist(\retr_{x_k}(s_k), x_k) \leq \retrdistboundconst \|s_k\|
  \end{align}
  for all iterates $x_k$ in $\mathcal{U}$ and steps $s_k$, where $\dist$ is the Riemannian distance on $\manifold$.
\end{lemma}
\begin{proof}
  This is a consequence of~\cite[Lem.~6]{ring2012optimization} because $\retr$
  is smooth.
\end{proof}

In the Euclidean case, $\dist(x, y) = \|x - y\|$ and for $\retr_x(s) = x+s$ we can clearly take $\retrdistboundconst = 1$.
In general, we can take $\retrdistboundconst > 1$ as close to $1$ as desired
after sufficiently restricting the neighborhood.

\begin{lemma}\label{lemma:strong-decrease}
  Assume that~\aref{assu:hess-approx} holds around a local minimum $\optpoint
  \in \optimalset$.
  Let TR generate iterates $\sequence{x_k}$ using a subproblem solver
  satisfying~\cref{cond:better-than-cauchy} and~\cref{cond:strong-vs} with
  constants $\rtrsufficientdecrease$ and $\strongvsconst$ around $\optpoint$.
  Given $\lamsharp > \lambda_{\max}(\hess \mfc(\optpoint))$, there exists a
  neighborhood $\mathcal{U}$ of $\optpoint$ such that if $x_k$ is in
  $\mathcal{U}$ then
  \begin{equation}\label{eq:strong-decrease}
    \begin{aligned}
      &\mfc(x_k) - \mfc(x_{k + 1}) \geq \frac{\rtrsufficientdecrease\rtrrho'}{\retrdistboundconst} \min\!\Big(1, \frac{1}{\strongvsconst\lamsharp}\Big) \|\grad \mfc(x_k)\| \dist(x_k, x_{k + 1})\\
      &\text{and}\quad x_k \in \optimalset \;\;\Rightarrow\;\; x_{k + 1} = x_k,
    \end{aligned}
  \end{equation}
  where $\rtrrho'$ is as in the algorithm definition~\eqref{eq:rtr-dynamics} and
  $\retrdistboundconst$ is as in Lemma~\ref{lemma:retr-dist-bound}.
\end{lemma}
\begin{proof}
  The implication $x_k \in \optimalset \Rightarrow x_{k + 1} = x_k$ is
  immediate because the gradient is zero on the optimal set $\optimalset$
  and~\cref{cond:strong-vs} then implies that $s_k = \zeros$.
  We now prove the lower-bound on the function decrease.
  Let $\mathcal{U}$ be a neighborhood of $\optpoint$ where the inequalities
  in~\cref{cond:strong-vs},~\eqref{eq:ratio-hess-inner}
  and~\eqref{eq:retr-dist-bound} hold.
  Consider $x_k \in \mathcal{U}$.
  Starting from the inequality in~\cref{cond:better-than-cauchy},
  apply~\eqref{eq:ratio-hess-inner} then~\cref{cond:strong-vs} to confirm that
  \begin{align}\label{eq:model-decrease-intermediate}
      \rtrmodel_k(0) - \rtrmodel_k(s_k) \geq \rtrsufficientdecrease\min\!\left(\Delta_k, \frac{\|s_k\|}{\strongvsconst\lamsharp}\right) \|\grad \mfc(x_k)\|.
  \end{align}
  Since $\Delta_k \geq \|s_k\|$, factor $\|s_k\|$ out of the min and use
  $\dist(x_k, x_{k + 1}) \leq \retrdistboundconst \|s_k\|$ (owing to
  Lemma~\ref{lemma:retr-dist-bound}) to claim
  \begin{align}\label{eq:model-decrease-intermediate-bis}
      \rtrmodel_k(0) - \rtrmodel_k(s_k) \geq \frac{\rtrsufficientdecrease}{\retrdistboundconst} \min\!\left(1, \frac{1}{\strongvsconst\lamsharp}\right) \|\grad \mfc(x_k)\| \dist(x_k, x_{k+1}).
  \end{align}
  Now two cases can happen from the TR dynamics~\eqref{eq:rtr-dynamics}.
  Either the step is rejected ($x_{k+1} = x_k$) and~\eqref{eq:strong-decrease} holds trivially.
  Or the step is accepted ($x_{k+1} = \retr_{x_k}(s_k)$) because
  $\mfc(x_k) - \mfc(\retr_{x_k}(s_k)) \geq
  \rtrrho'\big(\rtrmodel_k(\zeros) - \rtrmodel_k(s_k)\big)$,
  in which case~\eqref{eq:strong-decrease} also holds.
  (Note: $\rtrrho'$ could be improved to $1-o(1)$.)
\end{proof}

Assuming additionally that $\mfc$ is $\plconstant$-\eqref{eq:pl} and $\hess
\mfc$ is locally Lipschitz around $\optpoint$ (\aref{assu:hess-lip}), and given
$\muflat < \plconstant$, we know that \tcg{} satisfies~\cref{cond:strong-vs}
with constant $1/\muflat$ (Proposition~\ref{prop:rtr-tcg-strong-vs}).
In this case we can specialize Lemma~\ref{lemma:strong-decrease} and
obtain~\eqref{eq:strong-decrease} with constants $\rtrsufficientdecrease = 1/2$
and $\strongvsconst = 1/\muflat$.

The strong decrease property notably leads to a capture result:
see~\citep{absil2005convergence} and comments in~\citep[\S3.2]{rebjock2023fast}.
To secure it, we need to ensure that TR accumulates only at critical points.
That is effectively the case.
However, to formalize this globally we would need a few technical assumptions (see for
example~\citep[Thm.~7.4.4]{absil2008optimization} and
\citep[Prop.~6.25]{boumal2020introduction}).
Fortunately, we only need this property locally around local minima with a \pl{}
condition.
As it turns out, this frees us from those technicalities.

\begin{proposition}[Capture]\label{prop:rtr-tcg-capture}
  Suppose that~\aref{assu:hess-lip-like}, \aref{assu:hess-approx}
  and~\eqref{eq:pl} hold around $\optpoint \in \optimalset$.
  Let TR generate a sequence $\sequence{x_k}$ with a subproblem solver that
  satisfies~\cref{cond:better-than-cauchy} and~\cref{cond:strong-vs}.
  Given a neighborhood $\mathcal{U}$ of $\optpoint$, there exists a neighborhood
  $\mathcal{V} \subseteq \mathcal{U}$ of $\optpoint$ such that if an iterate of TR enters
  $\mathcal{V}$ then all the subsequent iterates stay in $\mathcal{U}$ and the
  sequence converges to some point $x_\infty \in \mathcal{U} \cap \optimalset$.
\end{proposition}
\begin{proof}
  We can restrict $\mathcal{U}$ so that all critical points in $\mathcal{U}$ are
  in $\optimalset$ because~\eqref{eq:pl} holds around $\optpoint$.
  We want to apply~\cite[Cor.~3.7]{rebjock2023fast} to $\mathcal{U}$.
  This requires the algorithm to satisfy two properties called \emph{vanishing
    steps} and \emph{bounded path length} in that reference.
  The vanishing steps property holds around $\optpoint$ because we
  assume~\cref{cond:strong-vs}.
  Moreover, the strong decrease from Lemma~\ref{lemma:strong-decrease} together
  with~\eqref{eq:pl} imply the bounded path length property
  (see~\cite[Lem~3.8]{rebjock2023fast}).
  To apply the stated corollary, we finally need to show that if
  $\sequence{x_k}$ accumulates at a point $x_\infty \in \mathcal{U}$ then
  $x_\infty \in \optimalset$.
  Restrict $\mathcal{U}$ so that if an iterate $x_k$ is in $\mathcal{U}$ then
  the associated ratio satisfies $\rtrrho_k \geq \frac{1}{4}$
  (Lemma~\ref{lemma:ratios-converge}).
  Given a number $\lamsharp > \lambda_{\max}(\hess \mfc(\optpoint))$, further
  restrict $\mathcal{U}$ so that the conclusions of
  Lemmas~\ref{lemma:ratio-hess-inner} and~\ref{lemma:strong-decrease} hold.
  Now invoke~\citep[Prop.~3.5]{rebjock2023fast}.
  We obtain an open neighborhood $\mathcal{V} \subseteq \mathcal{U}$ of
  $\optpoint$ such that if an iterate enters $\mathcal{V}$ then all subsequent
  iterates stay in $\mathcal{U}$.
  Suppose that the sequence $\sequence{x_k}$ accumulates at some point $x_\infty
  \in \mathcal{V}$.
  There exists an integer $K$ such that $x_k \in \mathcal{U}$ for all $k \geq
  K$.
  As a result the trust-region radii eventually stop decreasing: there exists a
  number $\Delta_{\min} > 0$ such that $\Delta_k \geq \Delta_{\min}$ for \emph{all} $k$.
  Suppose for contradiction that $\grad \mfc(x_\infty) \neq \zeros$.
  Then, consider condition~\cref{cond:better-than-cauchy} together with~\eqref{eq:ratio-hess-inner}, $\Delta_k \geq \Delta_{\min}$ and the fact that $\|\nabla \mfc(x_k)\|$ is bounded away from zero along a subsequence convergent to $x_\infty$.
  From this, deduce the existence of a number $\omega > 0$
  such that
  \begin{align*}
    \mfc(x_k) - \mfc(x_{k + 1}) = \rtrrho_k \big(\rtrmodel_k(\zeros) - \rtrmodel_k(s_k)\big) \geq \omega
  \end{align*}
  for infinitely many iterations $k$, all successful owing to $\rtrrho_k \geq 1/4$.
  This is incompatible with the facts that $\sequence{x_k}$ accumulates at
  $x_\infty$
  and that $\mfc(x_k)$ is decreasing.
  It follows that $\grad \mfc(x_\infty) = \zeros$
  and so $x_\infty \in \optimalset$.
\end{proof}

\paragraph{Order of convergence.}\label{par:order}

Now that local convergence to a point is secured,
we turn our focus to the convergence rate of TR.
As we show, the conditions~\cref{cond:better-than-cauchy},~\cref{cond:strong-vs}
and~\cref{cond:small-model} are sufficient to secure superlinear rates.

\begin{proposition}\label{prop:conv-rate}
  Suppose that~\aref{assu:hess-lip},~\aref{assu:hess-lip-like}
  and~\aref{assu:hess-approx} hold around $\optpoint \in \optimalset$.
  Let TR generate iterates $\sequence{x_k}$ converging to $\optpoint$ using a
  subproblem solver
  satisfying~\cref{cond:better-than-cauchy},~\cref{cond:strong-vs}
  and~\cref{cond:small-model} around $\optpoint$.
  Then the sequence $\sequence{\|\grad \mfc(x_k)\|}$ converges superlinearly to
  zero with order at least $1 + \power$.
\end{proposition}
\begin{proof}
  Let $\rtrsufficientdecrease$, $\strongvsconst$ and $\modelconst$ be the
  constants associated to~\cref{cond:better-than-cauchy},~\cref{cond:strong-vs}
  and~\cref{cond:small-model}.
  From Lemma~\ref{lemma:ratios-converge} and~\eqref{eq:rtr-dynamics} we deduce
  that all the steps are eventually successful.
  We consider large enough iterations $k$ so that it is the case.
  Since $\dist(x_k, x_{k+1}) \to 0$, the vector $\Log_{x_k}(x_{k + 1})$ is well
  defined for all $k$ large enough.
  (Recall from Table~\ref{table:euclidean-case} that $\Log_x(y) = y - x$ in the
  Euclidean case.)
  Let $v_k = s_k - \Log_{x_k}(x_{k + 1})$ (this is zero in the Euclidean case).
  Let $\ptransport{x_{k + 1}}{x_k}$ denote parallel transport along the
  minimizing geodesic connecting $x_{k+1}$ to $x_{k}$---this too is well defined
  for large enough $k$.
  (In the Euclidean case, $\ptransport{x_{k + 1}}{x_k}$ is identity.)
  The triangle inequality then provides:
  \begin{align}\label{eq:grad-bound}
    \|\grad \mfc(x_{k + 1})\| &= \big\|\ptransport{x_{k + 1}}{x_k}\grad \mfc(x_{k + 1}) - \grad \mfc(x_k) - \hess \mfc(x_k)[\Log_{x_k}(x_{k + 1})]\nonumber\\
                              &\qquad\qquad\qquad- \hess \mfc(x_k)[v_k] - \big(\linearmap_k - \hess \mfc(x_k)\big)[s_k] + \grad \rtrmodel_k(s_k)\big\|\nonumber\\
                              &\leq \frac{\hesslipconst}{2}\dist(x_k, x_{k + 1})^2 + \|\hess \mfc(x_k)\| \|v_k\| + \hessapproxconst \|\grad \mfc(x_k)\|\|s_k\| + \modelconst \|\grad \mfc(x_k)\|^{1 + \power},
  \end{align}
  where we invoked the Lipschitz continuity of the Hessian~\aref{assu:hess-lip}
  \citep[Cor.~10.56]{boumal2020introduction}, the bound
  in~\aref{assu:hess-approx}, and the bound in~\cref{cond:small-model}.
  The condition~\cref{cond:strong-vs} together with
  Lemma~\ref{lemma:retr-dist-bound} give that $\dist(x_k, x_{k + 1}) \leq
  \retrdistboundconst \|s_k\| \leq \retrdistboundconst \strongvsconst \|\grad
  \mfc(x_k)\|$ for large enough $k$.
  We now bound the quantity $\|v_k\|$.
  There exists a neighborhood of the zero section of the tangent bundle such
  that $(x, s) \mapsto (x, \retr_x(s))$ is a
  diffeomorphism~\citep[Cor.~10.27]{boumal2020introduction}.
  Moreover, the inverse function theorem implies that $\D \retr_x^{-1}(x) = I$
  for all $x \in \manifold$ because $\D \retr_x(\zeros) = I$.
  It follows that there exists a neighborhood $\mathcal{U}$ of $\optpoint$ such
  that for all $x, y \in \mathcal{U}$ we have
  \begin{align*}
    \retr_x^{-1}(y) = \retr_x^{-1}(x) + \D\retr_x^{-1}(x)[\Log_x(y)] + O(\dist(x, y)^2) = \Log_x(y) + O(\dist(x, y)^2).
  \end{align*}
  In particular, using the identity $s_k = \retr_{x_k}^{-1}(x_{k + 1})$, we find
  that there exists a constant $c_v$ such that $\|v_k\| \leq c_v\dist(x_k, x_{k
    + 1})^2$ holds for large enough $k$.
  It follows from~\eqref{eq:grad-bound} that
  \begin{align*}
    \|\grad \mfc(x_{k + 1})\| \leq \bigg( \retrdistboundconst^2 \strongvsconst^2 \Big(\frac{\hesslipconst}{2} + c_v \|\hess \mfc(x_k)\| \Big) + \strongvsconst \hessapproxconst \bigg) \|\grad \mfc(x_k)\|^2 + \modelconst \|\grad \mfc(x_k)\|^{1 + \power}
  \end{align*}
  for large enough $k$, showing the superlinear convergence of the sequence
  $\sequence{\|\grad \mfc(x_k)\|}$.
\end{proof}

From this result, we can deduce the superlinear convergence of the sequences
$\sequence{\dist(x_k, \optimalset)}$ and $\sequence{\mfc(x_k) - \mfcopt}$ as
follows.
Suppose that~\eqref{eq:pl} holds around $\optpoint$.
This is equivalent to the error bound condition, meaning that $\plconstant
\dist(x, \optimalset) \leq \|\grad \mfc(x)\|$ for all $x$ sufficiently close to
$\optpoint$ (see for example~\cite[Thm.~2]{karimi2016linear}
or~\cite[Rmk.~2.10]{rebjock2023fast}).
This implies that $\sequence{\dist(x_k, \optimalset)}$ converges to zero
superlinearly with order at least $1 + \power$.
Finally we can deduce that $\sequence{\mfc(x_k) - \mfcopt}$ converges to zero
with the same rate from~\eqref{eq:pl}.

\begin{proof}[Proof of Theorem~\ref{th:rtr-main-theorem}]
  First invoke Propositions~\ref{prop:rtr-tcg-strong-vs}
  and~\ref{prop:rtr-tcg-small-model} to obtain that \tcg{}
  satisfies~\cref{cond:strong-vs} and~\cref{cond:small-model}.
  The result then follows from Propositions~\ref{prop:rtr-tcg-capture}
  and~\ref{prop:conv-rate}.
\end{proof}

\begin{remark}\label{rmk:super-linear-but-not-quadratic}
  Theorem~\ref{th:rtr-main-theorem} ensures superlinear convergence but not
  quadratic convergence.
  And indeed, without further assumptions on $\mfc$ our proof cannot provide
  quadratic convergence for TR with \tcg{}.
  The reason is that the parameter $\power$ in
  Lemma~\ref{lemma:cg-iterate-existence} has to be \emph{strictly} less than
  $1$.
  We show here that this cannot be improved.
  Consider the function $\mfc(x, y) = \frac{3}{16}(1 + \frac{64}{3}x^2)y^2$ and
  the path $c(\varepsilon) = (\frac{1}{8}\sqrt{1 - \varepsilon},
  \sqrt{\varepsilon})$ for a small parameter $\varepsilon \geq 0$.
  The function $\mfc$ is polynomial and satisfies~\eqref{eq:pl} around $c(0)$.
  Moreover, the eigenvalues of $\hess \mfc(c(\varepsilon))$ are both positive
  for $\varepsilon$ sufficiently small.
  Consider \cg{} with inputs $\big(\hess \mfc(c(\varepsilon)), -\grad
  \mfc(c(\varepsilon))\big)$.
  Let $\cgres_1(\varepsilon)$ denote the residual of the first iteration and
  $\cgvec_2(\varepsilon)$ the iterate of the second iteration.
  Then, together with the gradient, they satisfy
  \begin{align*}
    \|\cgres_1(\varepsilon)\| \sim \varepsilon && \|\grad \mfc(c(\varepsilon))\|^2 \sim \varepsilon/4 && \text{and} && \|\cgvec_2(\varepsilon)\| \sim 1/6\varepsilon
  \end{align*}
  as $\varepsilon \to 0$, where $\sim$ denotes asymptotic equivalence.
  In particular, the inequality $\|\cgres_1(\varepsilon)\| \leq \|\grad
  \mfc(c(\varepsilon))\|^2$ cannot hold when $\varepsilon$ is sufficiently close
  to zero, and $\|\cgvec_2(\varepsilon)\| / \|\grad \mfc(c(\varepsilon))\| \to
  +\infty$ as $\varepsilon \to 0$.
  We conclude that both iterates of \cg{} are incompatible with the
  bounds~\eqref{eq:cg-iterate-guarantees}, even when $\varepsilon$ is
  arbitrarily close to zero.
  Since $\|\cgvec_2(\varepsilon)\| \to +\infty$ as $\varepsilon \to 0$, this
  also shows that \tcg{} with parameter $\power = 1$ suffers from the same
  shortcomings as the exact subproblem solver: it is not possible to ensure
  capture of the iterates (the condition~\cref{cond:strong-vs} breaks).
\end{remark}

\section{A note about quotient manifolds} \label{sec:quotients}

Symmetry is a common source of non-isolated minima in optimization.
Explicitly, suppose we seek to minimize $\mfc \colon \manifold \to \reals$ and
there exists an equivalence relation $\sim$ on $\manifold$ such that $x \sim y
\implies \mfc(x) = \mfc(y)$.
Under certain (well-understood) circumstances, the quotient space $\quotient{\manifold}{\equivrel}$ is itself a manifold, and all equivalence classes are submanifolds of $\manifold$ with equal dimension~\citep[\S3.4.1]{absil2008optimization}.

In this case, any local minimizer $\optpoint$ of $\mfc$ belongs to an equivalence class $[\optpoint] = \{ x \in \manifold : x \sim \optpoint \}$ of points which are all local minimizers.
If the symmetry stems from the action of a Lie group with positive dimension (e.g., rotations or translations), then each equivalence class is itself a submanifold of $\manifold$ of positive dimension.

Under those circumstances, the Hessian $\hess \mfc(\optpoint)$ cannot be positive definite.
However, the Hessian may well be positive definite upon passing to the quotient.
Explicitly, let $\varphi \colon \manifold \to \quotient{\manifold}{\sim}$ be the canonical projection $\varphi(x) = [x]$.
Then, the symmetries of $\mfc$ ensure that there exists a function $\sfc \colon \quotient{\manifold}{\sim} \to \reals$ such that $\mfc = \sfc \circ \varphi$.
If the symmetries are the only reason why $\hess \mfc(\optpoint) \not\succ 0$, then $\hess \sfc([\optpoint])$ is positive definite~\citep[Ex.~9.46]{boumal2020introduction}.
This means that if we run optimization algorithms on the quotient manifold, then we can expect all the good convergence properties that normally come with a positive definite Hessian,
including quadratic local convergence for TR-\tcg{}.
This forms part of the motivation for studying optimization on quotient manifolds---see~\citep{absil2008optimization}
and~\citep[Ch.~9]{boumal2020introduction}.

Under those same circumstances though, it is typical to observe that TR-\tcg{} enjoys fast local convergence even if we disregard the symmetries and run the algorithm on $\mfc \colon \manifold \to \reals$ directly.
We can now understand this as follows.
If $\hess \sfc([\optpoint]) \succ \zeros$, then $\mfc$ satisfies~\eqref{eq:pl} around $\optpoint$ (see for example~\cite[\S1.2]{rebjock2023fast}).
Thus, Theorem~\ref{th:rtr-main-theorem} applies, guaranteeing superlinear local convergence for TR-\tcg{}.
The role of \tcg{} appears to be instrumental.

To go beyond the assumption $\hess \sfc([\bar x]) \succ \zeros$, note that when $\quotient{\manifold}{\sim}$ is a \emph{Riemannian}
quotient~\cite[\S3.6.2]{absil2008optimization} of $\manifold$, the function
$\mfc$ is $\plconstant$-\eqref{eq:pl} around a minimum $\bar x$ if and only if
$\sfc$ is $\plconstant$-\eqref{eq:pl} around $[\bar x]$.
This equivalence readily follows from the equalities $\mfc(x) = \sfc([x])$ and
$\|\grad \mfc(x)\| = \|\grad \sfc([x])\|$ for all $x \in \manifold$.
If $\quotient{\manifold}{\sim}$ is a quotient manifold with another Riemannian
metric, then the equivalence still holds but possibly with different \pl{}
constants.

\bibliographystyle{plainnat}
\bibliography{references}

\begin{thebibliography}{50}
\providecommand{\natexlab}[1]{#1}
\providecommand{\url}[1]{\texttt{#1}}
\expandafter\ifx\csname urlstyle\endcsname\relax
  \providecommand{\doi}[1]{doi: #1}\else
  \providecommand{\doi}{doi: \begingroup \urlstyle{rm}\Url}\fi

\bibitem[Absil et~al.(2005)Absil, Mahony, and Andrews]{absil2005convergence}
P-A Absil, Robert Mahony, and Benjamin Andrews.
\newblock Convergence of the iterates of descent methods for analytic cost
  functions.
\newblock \emph{SIAM Journal on Optimization}, 16\penalty0 (2):\penalty0
  531--547, 2005.

\bibitem[Absil et~al.(2007)Absil, Baker, and Gallivan]{absil2007trust}
P-A Absil, Christopher~G Baker, and Kyle~A Gallivan.
\newblock Trust-region methods on {R}iemannian manifolds.
\newblock \emph{Foundations of Computational Mathematics}, 7\penalty0
  (3):\penalty0 303--330, 2007.

\bibitem[Absil et~al.(2008)Absil, Mahony, and Sepulchre]{absil2008optimization}
P-A Absil, Robert Mahony, and Rodolphe Sepulchre.
\newblock \emph{Optimization algorithms on matrix manifolds}.
\newblock Princeton University Press, 2008.

\bibitem[Adachi et~al.(2017)Adachi, Iwata, Nakatsukasa, and
  Takeda]{adachi2017solving}
Satoru Adachi, Satoru Iwata, Yuji Nakatsukasa, and Akiko Takeda.
\newblock Solving the trust-region subproblem by a generalized eigenvalue
  problem.
\newblock \emph{SIAM Journal on Optimization}, 27\penalty0 (1):\penalty0
  269--291, 2017.

\bibitem[Attouch et~al.(2010)Attouch, Bolte, Redont, and
  Soubeyran]{attouch2010proximal}
H{\'e}dy Attouch, J{\'e}r{\^o}me Bolte, Patrick Redont, and Antoine Soubeyran.
\newblock Proximal alternating minimization and projection methods for
  nonconvex problems: An approach based on the {K}urdyka--{{\L}}ojasiewicz
  inequality.
\newblock \emph{Mathematics of operations research}, 35\penalty0 (2):\penalty0
  438--457, 2010.

\bibitem[Attouch et~al.(2013)Attouch, Bolte, and
  Svaiter]{attouch2013convergence}
H{\'e}dy Attouch, J{\'e}r{\^o}me Bolte, and Benar~Fux Svaiter.
\newblock Convergence of descent methods for semi-algebraic and tame problems:
  proximal algorithms, forward-backward splitting, and regularized
  {G}auss--{S}eidel methods.
\newblock \emph{Mathematical Programming}, 137\penalty0 (1):\penalty0 91--129,
  2013.

\bibitem[Bhatia(1997)]{bhatia1997matrixanalysis}
R.~Bhatia.
\newblock \emph{Matrix Analysis}.
\newblock Springer New York, 1997.
\newblock \doi{10.1007/978-1-4612-0653-8}.

\bibitem[Bolte et~al.(2014)Bolte, Sabach, and Teboulle]{bolte2014proximal}
J{\'e}r{\^o}me Bolte, Shoham Sabach, and Marc Teboulle.
\newblock Proximal alternating linearized minimization for nonconvex and
  nonsmooth problems.
\newblock \emph{Mathematical Programming}, 146\penalty0 (1):\penalty0 459--494,
  2014.

\bibitem[Boumal(2023)]{boumal2020introduction}
Nicolas Boumal.
\newblock \emph{An introduction to optimization on smooth manifolds}.
\newblock Cambridge University Press, 2023.

\bibitem[Carmon and Duchi(2020)]{carmon2020first}
Yair Carmon and John~C Duchi.
\newblock First-order methods for nonconvex quadratic minimization.
\newblock \emph{SIAM Review}, 62\penalty0 (2):\penalty0 395--436, 2020.

\bibitem[Cartis et~al.(2011{\natexlab{a}})Cartis, Gould, and
  Toint]{cartis2011adaptive}
Coralia Cartis, Nicholas~IM Gould, and Philippe~L Toint.
\newblock Adaptive cubic regularisation methods for unconstrained optimization.
  {P}art {I}: motivation, convergence and numerical results.
\newblock \emph{Mathematical Programming}, 127\penalty0 (2):\penalty0 245--295,
  2011{\natexlab{a}}.

\bibitem[Cartis et~al.(2011{\natexlab{b}})Cartis, Gould, and
  Toint]{cartis2011adaptive2}
Coralia Cartis, Nicholas~IM Gould, and Philippe~L Toint.
\newblock Adaptive cubic regularisation methods for unconstrained optimization.
  {P}art {II}: worst-case function- and derivative-evaluation complexity.
\newblock \emph{Mathematical programming}, 130\penalty0 (2):\penalty0 295--319,
  2011{\natexlab{b}}.

\bibitem[Conn et~al.(2000)Conn, Gould, and Toint]{conn2000trust}
Andrew~R Conn, Nicholas~IM Gould, and Philippe~L Toint.
\newblock \emph{Trust region methods}.
\newblock SIAM, 2000.

\bibitem[Dembo and Steihaug(1983)]{dembo1983truncated}
Ron~S Dembo and Trond Steihaug.
\newblock Truncated-{N}ewton algorithms for large-scale unconstrained
  optimization.
\newblock \emph{Mathematical Programming}, 26\penalty0 (2):\penalty0 190--212,
  1983.

\bibitem[Dembo et~al.(1982)Dembo, Eisenstat, and Steihaug]{dembo1982inexact}
Ron~S Dembo, Stanley~C Eisenstat, and Trond Steihaug.
\newblock Inexact {N}ewton methods.
\newblock \emph{SIAM Journal on Numerical analysis}, 19\penalty0 (2):\penalty0
  400--408, 1982.

\bibitem[Fan(2006)]{fan2006convergence}
Jinyan Fan.
\newblock Convergence rate of the trust region method for nonlinear equations
  under local error bound condition.
\newblock \emph{Computational Optimization and Applications}, 34\penalty0
  (2):\penalty0 215--227, 2006.

\bibitem[Fong and Saunders(2012)]{fong2012cg}
David Chin-Lung Fong and Michael Saunders.
\newblock {CG} versus {MINRES}: An empirical comparison.
\newblock \emph{Sultan Qaboos University Journal for Science [SQUJS]},
  17\penalty0 (1):\penalty0 44--62, 2012.

\bibitem[Golub and Meurant(2010)]{golub2010matrices}
Gene~H. Golub and G\'erard Meurant.
\newblock \emph{Matrices, Moments and Quadrature with Applications}.
\newblock Princeton University Press, 2010.

\bibitem[Gould et~al.(1999)Gould, Lucidi, Roma, and Toint]{gould1999solving}
Nicholas~IM Gould, Stefano Lucidi, Massimo Roma, and Philippe~L Toint.
\newblock Solving the trust-region subproblem using the {L}anczos method.
\newblock \emph{SIAM Journal on Optimization}, 9\penalty0 (2):\penalty0
  504--525, 1999.

\bibitem[Greenbaum(1989)]{greenbaum1989behavior}
Anne Greenbaum.
\newblock Behavior of slightly perturbed {L}anczos and conjugate-gradient
  recurrences.
\newblock \emph{Linear Algebra and its Applications}, 113:\penalty0 7--63,
  1989.

\bibitem[Greenbaum(1997)]{greenbaum1997iterative}
Anne Greenbaum.
\newblock \emph{Iterative Methods for Solving Linear Systems}.
\newblock Society for Industrial and Applied Mathematics, January 1997.

\bibitem[Greenbaum and Strakos(1992)]{greenbaum1992predicting}
Anne Greenbaum and Zdenek Strakos.
\newblock Predicting the behavior of finite precision {L}anczos and conjugate
  gradient computations.
\newblock \emph{SIAM Journal on Matrix Analysis and Applications}, 13\penalty0
  (1):\penalty0 121--137, 1992.

\bibitem[Griewank(1981)]{griewank1981modification}
Andreas Griewank.
\newblock The modification of {N}ewton’s method for unconstrained
  optimization by bounding cubic terms.
\newblock Technical Report {NA/12}, Department of Applied Mathematics and
  Theoretical Physics, University of Cambridge, 1981.

\bibitem[Hestenes and Stiefel(1952)]{hestenes1952methods}
Magnus~R Hestenes and Eduard Stiefel.
\newblock Methods of conjugate gradients for solving linear systems.
\newblock \emph{Journal of research of the National Bureau of Standards},
  49\penalty0 (6):\penalty0 409--436, 1952.

\bibitem[Karimi et~al.(2016)Karimi, Nutini, and Schmidt]{karimi2016linear}
Hamed Karimi, Julie Nutini, and Mark Schmidt.
\newblock Linear convergence of gradient and proximal-gradient methods under
  the {P}olyak--{{\L}}ojasiewicz condition.
\newblock In \emph{Joint European Conference on Machine Learning and Knowledge
  Discovery in Databases}, pages 795--811. Springer, 2016.

\bibitem[Lanczos(1950)]{lanczos1950iteration}
Cornelius Lanczos.
\newblock An iteration method for the solution of the eigenvalue problem of
  linear differential and integral operators.
\newblock \emph{Journal of Research of the National Bureau of Standards},
  45\penalty0 (4), 1950.

\bibitem[Liesen and Strako\v{s}(2013)]{liesen2013krylov}
J{\"o}rg Liesen and Zdenek Strako\v{s}.
\newblock \emph{{K}rylov subspace methods: principles and analysis}.
\newblock Oxford University Press, 2013.

\bibitem[Liu et~al.(2022)Liu, Zhu, and Belkin]{liu2022loss}
Chaoyue Liu, Libin Zhu, and Mikhail Belkin.
\newblock Loss landscapes and optimization in over-parameterized non-linear
  systems and neural networks.
\newblock \emph{Applied and Computational Harmonic Analysis}, 59:\penalty0
  85--116, 2022.

\bibitem[Liu and Roosta(2022{\natexlab{a}})]{liu2022minres}
Yang Liu and Fred Roosta.
\newblock {MINRES}: From negative curvature detection to monotonicity
  properties.
\newblock \emph{SIAM Journal on Optimization}, 32\penalty0 (4):\penalty0
  2636--2661, 2022{\natexlab{a}}.

\bibitem[Liu and Roosta(2022{\natexlab{b}})]{liu2022newton}
Yang Liu and Fred Roosta.
\newblock A {N}ewton-{MR} algorithm with complexity guarantees for nonconvex
  smooth unconstrained optimization.
\newblock \emph{arXiv preprint arXiv:2208.07095}, 2022{\natexlab{b}}.

\bibitem[\L{}ojasiewicz(1963)]{lojasiewicz1963propriete}
Stanislaw \L{}ojasiewicz.
\newblock Une propri{\'e}t{\'e} topologique des sous-ensembles analytiques
  r{\'e}els.
\newblock \emph{Les {\'e}quations aux d{\'e}riv{\'e}es partielles},
  117:\penalty0 87--89, 1963.

\bibitem[\L{}ojasiewicz(1982)]{lojasiewicz1982trajectoires}
Stanislaw \L{}ojasiewicz.
\newblock Sur les trajectoires du gradient d’une fonction analytique.
\newblock \emph{Seminari di geometria}, 1983:\penalty0 115--117, 1982.

\bibitem[Luo and Tseng(1993)]{luo1993error}
Zhi-Quan Luo and Paul Tseng.
\newblock Error bounds and convergence analysis of feasible descent methods: a
  general approach.
\newblock \emph{Annals of Operations Research}, 46\penalty0 (1):\penalty0
  157--178, 1993.

\bibitem[Meurant(2006)]{meurant2006lanczos}
G{\'e}rard Meurant.
\newblock \emph{The {L}anczos and Conjugate Gradient Algorithms: From Theory to
  Finite Precision Computations}.
\newblock Society for Industrial and Applied Mathematics, 2006.

\bibitem[Meurant and Strako{\v{s}}(2006)]{meurantstrakos2006lanczos}
G{\'e}rard Meurant and Zden{\v{e}}k Strako{\v{s}}.
\newblock The {L}anczos and conjugate gradient algorithms in finite precision
  arithmetic.
\newblock \emph{Acta Numerica}, 15:\penalty0 471--542, 2006.

\bibitem[Mor{\'e} and Sorensen(1983)]{more1983computing}
Jorge~J Mor{\'e} and Danny~C Sorensen.
\newblock Computing a trust region step.
\newblock \emph{SIAM Journal on scientific and statistical computing},
  4\penalty0 (3):\penalty0 553--572, 1983.

\bibitem[Nesterov and Polyak(2006)]{nesterov2006cubic}
Yurii Nesterov and Boris~T Polyak.
\newblock Cubic regularization of {N}ewton method and its global performance.
\newblock \emph{Mathematical Programming}, 108\penalty0 (1):\penalty0 177--205,
  2006.

\bibitem[Nocedal and Wright(2006)]{nocedal2006numerical}
Jorge Nocedal and Stephen~J Wright.
\newblock \emph{Numerical Optimization}.
\newblock Springer New York, 2006.

\bibitem[Paige(1971)]{paige1971computation}
Christopher~C Paige.
\newblock \emph{The computation of eigenvalues and eigenvectors of very large
  sparse matrices}.
\newblock PhD thesis, University of London, 1971.

\bibitem[Paige and Saunders(1975)]{paige1975solution}
Christopher~C Paige and Michael~A Saunders.
\newblock Solution of sparse indefinite systems of linear equations.
\newblock \emph{SIAM journal on numerical analysis}, 12\penalty0 (4):\penalty0
  617--629, 1975.

\bibitem[Parlett(1998)]{parlett1998symmetric}
Beresford~N. Parlett.
\newblock \emph{The Symmetric Eigenvalue Problem}.
\newblock Society for Industrial and Applied Mathematics, 1998.

\bibitem[Polyak(1963)]{polyak1963gradient}
Boris~T Polyak.
\newblock Gradient methods for the minimisation of functionals.
\newblock \emph{USSR Computational Mathematics and Mathematical Physics},
  3\penalty0 (4):\penalty0 864--878, 1963.

\bibitem[Rebjock and Boumal(2023)]{rebjock2023fast}
Quentin Rebjock and Nicolas Boumal.
\newblock Fast convergence to non-isolated minima: four equivalent conditions
  for $\mathrm{C}^2$ functions.
\newblock \emph{arXiv preprint arXiv:2303.00096}, 2023.

\bibitem[Ring and Wirth(2012)]{ring2012optimization}
Wolfgang Ring and Benedikt Wirth.
\newblock Optimization methods on {R}iemannian manifolds and their application
  to shape space.
\newblock \emph{SIAM Journal on Optimization}, 22\penalty0 (2):\penalty0
  596--627, 2012.

\bibitem[Steihaug(1983)]{steihaug1983conjugate}
Trond Steihaug.
\newblock The conjugate gradient method and trust regions in large scale
  optimization.
\newblock \emph{SIAM Journal on Numerical Analysis}, 20\penalty0 (3):\penalty0
  626--637, 1983.

\bibitem[Toint(1981)]{toint1981towards}
Philippe Toint.
\newblock Towards an efficient sparsity exploiting {N}ewton method for
  minimization.
\newblock In \emph{Sparse matrices and their uses}, pages 57--88. Academic
  press, 1981.

\bibitem[Trefethen and Bau(1997)]{trefethen1997numerical}
Lloyd~N Trefethen and David Bau.
\newblock \emph{Numerical linear algebra}.
\newblock Society for Industrial and Applied Mathematics, 1997.

\bibitem[Yuan(2000)]{yuan2000truncated}
Yaxiang Yuan.
\newblock On the truncated conjugate gradient method.
\newblock \emph{Mathematical Programming}, 87:\penalty0 561--573, 2000.

\bibitem[Yue et~al.(2019)Yue, Zhou, and Man-Cho~So]{yue2019quadratic}
Man-Chung Yue, Zirui Zhou, and Anthony Man-Cho~So.
\newblock On the quadratic convergence of the cubic regularization method under
  a local error bound condition.
\newblock \emph{SIAM Journal on Optimization}, 29\penalty0 (1):\penalty0
  904--932, 2019.

\bibitem[Zhou et~al.(2018)Zhou, Wang, and Liang]{zhou2018convergence}
Yi~Zhou, Zhe Wang, and Yingbin Liang.
\newblock Convergence of cubic regularization for nonconvex optimization under
  {K\L{}} property.
\newblock \emph{Advances in Neural Information Processing Systems}, 31, 2018.

\end{thebibliography}

\clearpage
\appendix

\section{Simple example with superlinear convergence}\label{sec:example}

Consider the cost function $\mfc \colon \reals^n \times \reals^n \to \reals$
defined as
\begin{align*}
  \mfc(x, y) = \frac{1}{2} \sum_{i = 1}^n \big(y_i - \sin(x_i)\big)^2 = \frac{1}{2}\|y - \sin(x)\|^2,
\end{align*}
where $\sin$ is applied entrywise.
The set of minimizers $\optimalset = \big\{(x, y) \in \reals^n \times
\reals^n : y = \sin(x)\big\}$ is an embedded submanifold of $\reals^n \times \reals^n$.
Standard computations yield the gradient
\begin{align*}
  \grad \mfc(x, y) = \Big(-\cos(x) \odot \big(y - \sin(x)\big), y - \sin(x)\Big).
\end{align*}
It is clear that $\|\grad \mfc(x, y)\|^2 \geq \|y - \sin(x)\|^2 = 2\mfc(x, y)$,
hence $\mfc$ is globally $1$-\eqref{eq:pl}.

We run TR-\tcg{} with $n = 100$ and a random initialization.
Figure~\ref{fig:exp} displays the norm of the gradient at the iterates.
The convergence appears superlinear.
\begin{figure}
  \centering
  \includegraphics[width=0.6\textwidth]{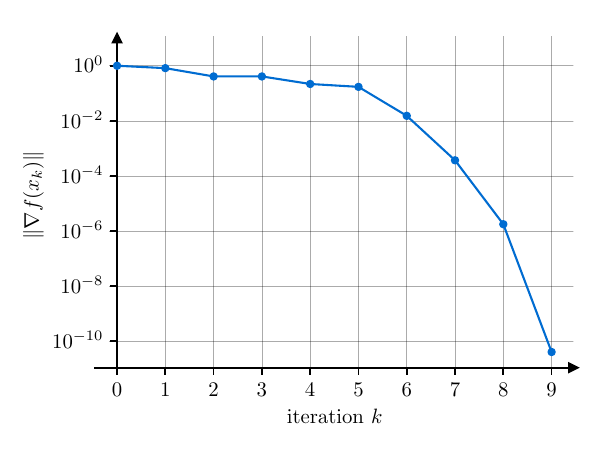}
  \caption{The gradient norm exhibits a pattern typical of superlinear
    convergence.}\label{fig:exp}
\end{figure}

We now briefly argue that the Hessian has a negative eigenvalue in the vicinity
of $\optimalset$.
Consider the case $n = 1$ for simplicity.
The Hessian is the $2 \times 2$ matrix
\begin{align*}
  \hess \mfc(x, y) =
  \begin{bmatrix}
    \sin(x)\big(y - \sin(x)\big) + \cos(x)^2 & -\cos(x)\\ -\cos(x) & 1
  \end{bmatrix}.
\end{align*}
There is a negative eigenvalue if the determinant $\sin(x)\big(y - \sin(x)\big)$
is negative.
Given $(\bar x, \bar y) \in \optimalset$, we can find points $(x, y)$
arbitrarily close to $(\bar x, \bar y)$ such that this quantity is negative.

\end{document}